\documentclass[11pt]{amsart}
\usepackage[foot]{amsaddr}
\usepackage{fullpage}
\usepackage[utf8]{inputenc}
\usepackage{mathtools}
\usepackage{amsmath,amsthm}
\usepackage{amssymb}
\usepackage{color}
\usepackage{bbm}
\usepackage{hyperref}
\usepackage{enumitem}
\numberwithin{equation}{section}
\allowdisplaybreaks

\newtheorem{theorem}{Theorem}[section]
\newtheorem{corollary}[theorem]{Corollary}
\newtheorem{proposition}[theorem]{Proposition}
\newtheorem{lemma}[theorem]{Lemma}

\theoremstyle{definition}

\theoremstyle{remark}
\newtheorem{remark}[theorem]{Remark}

\title{\textbf{Shock formation of the Burgers-Hilbert equation}}
\author{Ruoxuan Yang}
\address{Department of Mathematics, Massachusetts Institute of Technology, Cambridge, MA 02139} \email{rxyang@mit.edu}
\date{}

\begin{document}

\begin{abstract}
We prove finite time blowup of the Burgers-Hilbert equation. We construct smooth initial data with finite $H^5$-norm such that the $L^\infty$-norm of the spatial derivative of the solution blows up in finite time. The blowup is an asymptotic self-similar shock at one single point with an explicitly computable blowup profile. The blowup profile is a cusp with H\"older $1/3$ continuity. The blowup time and  location are described in terms of explicit ODEs. Our proof uses a transformation to modulated self-similar variables which is compatible with the Hilbert transform, the quantitative properties of the stable self-similar solution to the inviscid Burgers equation, an $L^2$-estimate in self-similar variables, and pointwise estimates for the Hilbert transform and for transport equations. 
\end{abstract}

\maketitle


\section{Introduction}
The Burgers-Hilbert (BH) equation consists of an inviscid Burgers equation with a source term given by the Hilbert transform
\begin{equation}
    \partial_t u+u\partial_xu=H[u],\label{eq:BH}
\end{equation}
where the Hilbert transform is defined for $f:\mathbb{R}\to\mathbb{R}$ by
\begin{align*}
H[f](x):=\frac{1}{\pi}\mathrm{p.v.}\int_\mathbb{R}\frac{f(y)}{x-y}\,dy,\quad \widehat{H[f]}(\xi)=-i\,\mathrm{sgn}(\xi)\hat{f}(\xi),
\end{align*}
where p.v. stands for principal value.

The BH equation is Hamiltonian in the sense that
\begin{align*}
\partial_t u+\partial_x\Big(\frac{\delta \mathcal{H}}{\delta u}\Big)=0,\quad \mathcal{H}(u)=\int_\mathbb{R}\Big(\frac{1}{6}u^3+\frac{1}{2}u\Lambda^{-1}u\Big)\,dx,
\end{align*}
where $\Lambda=(-\Delta)^\frac{1}{2}=H\partial_x$ has symbol $|\xi|$.

The BH equation was first obtained by Marsden and Weinstein \cite{Marsden-Weinstein} as a quadratic
approximation for the motion
of a free boundary of a vortex patch. Later, Biello and Hunter \cite{Biello-Hunter} showed, using formal asymptotic expansions, that the BH equation is an effective equation for small-amplitude motions of a planar vorticity discontinuity between two two-dimensional inviscid incompressible shear flows of different rates, and very recently the validity of this approximation is proved in \cite{Hunter-andmore}. Hence, the BH equation is a model equation for more complicated fluid systems.

Since the Hilbert transform is a skew-adjoint singular integral operator of order zero, and $H^2=-I$, the source term $H[u]$ in \eqref{eq:BH} is $L^2$-conservative but non-smoothing. The linearized equation $\partial_t u=H[u]$ is non-dispersive. The initial value problem 
\begin{align*}
\partial_t u=H[u],\quad u(x,0)=u_0(x)
\end{align*}
has general solutions \cite{Biello-Hunter} 
\begin{align*}
u(x,t)=u_0(x)\cos{t}+H[u_0](x)\sin{t},
\end{align*}
so the solutions oscillate with constant frequency 1. Thus the BH equation is a model equation for waves with a constant, nonzero linearized frequency. Dimensional analysis in \cite{Biello-Hunter} showed that the BH equation is an appropriate equation to describe Hamiltonian surface
waves with a constant frequency.

The initial value problem \begin{equation}
     \partial_t u+u\partial_x u=H[u],\qquad u(x,0)=u_0(x)\label{eq:BHinit}
\end{equation}
is locally well posed in $H^k(\mathbb{R})$ for $k>3/2$, with the same proof as for the inviscid Burgers equation. We include the proof in Appendix \ref{ap:wellposed}. For uniqueness, if two solutions in $\mathcal{C}\big([0,T];H^k(\mathbb{R})\big)\cap \mathcal{C}^1\big((0,T);H^{k-2}(\mathbb{R})\big)$ for $k>3/2$ to the initial value problem \eqref{eq:BHinit} agree on an open subset of $\mathbb{R}\times [0,T]$, then they are identical on $\mathbb{R}\times [0,T]$ \cite{Kenig}.

There are also some small amplitude, even, zero-mean, $2\pi$-periodic initial data that have global smooth traveling wave solutions. Such solutions can be found by an application of bifurcation theory \cite{Hunter}. Bressan and
Nguyen \cite{Bressan-Nguyen} proved that a global-in-time weak entropy solution exits if $u_0\in L^2(\mathbb{R})$; moreover, for this solution, the function $t\mapsto \|u(\cdot,t)\|_{L^2}$ is non-increasing and $u(\cdot,t)\in L^\infty(\mathbb{R})$ for all $t>0$. The proof in \cite{Bressan-Nguyen} does not show that the solution is of bounded variation. The uniqueness of weak entropy solutions is open. As an intermediate situation, Bressan and Zhang constructed piecewise continuous solutions to the BH equation with a single shock \cite{Bressan-Zhang}.

In the opposite direction, numerical simulations show that smooth initial data typically form shocks, i.e.\ $\partial_x u\to\infty$ in finite time. For example, a numerical solution of the BH equation with initial datum $u_0(x)=\sin{x}$ has a typical Burgers steepening and forms a logarithmically cusped shock like the cusp of $|x|\log|x|$ at $x=0$ \cite{Hunter}. From the analytic side, Castro, C\'{o}rdoba, and Gancedo \cite{C-C-G} proved finite time blowup of the $\mathcal{C}^{1,\delta}$-norm with $0<\delta<1$ for initial data $u_0\in L^2(\mathbb{R})\cap C^{1,\delta}(\mathbb{R})$ that has an $x_0\in\mathbb{R}$ such that $H[u_0](x_0)>0$ and $u_0(x_0)\geq (32\pi\|u_0\|^2_{L^2})^\frac{1}{3}$. However, the proof is by contradiction together with a maximal principle for a certain integral operator; it does not show that $\partial_xu$ blows up as observed in numerical simulations, or give an estimate on the blowup time.

In this paper we establish finite time blowup of the BH equation. Since we need to introduce many notations to state our main result, here just like \cite{B-S-V1}, \cite{B-S-V2} we only give a rough statement. The precise statement is given in Theorem \ref{thm:main} below.

\begin{theorem}[Rough statement of the main result]
    There exists an open set of smooth initial data with minimum initial slope equal to $-\epsilon^{-1}$ for $\epsilon>0$ sufficiently small, such that smooth solutions to the BH equation from these initial data blow up within time $O(\epsilon)$. The blowup is an explicit asymptotically self-similar shock at only one  point, and this is the first time the solutions become singular.  The blowup profile is an explicitly computable cusp with H\"older 1/3 continuity. Both the location where $\partial_xu$ blows up and the time of the first singularity are described in terms of explicit ODEs. 
\end{theorem}

\begin{remark}[Comparison to existing numerical and analytic results on shock formation]
In \cite{Biello-Hunter}, Biello and Hunter studied numerical solutions to the BH equation with small-amplitude initial data 
$$u_0(x)=\epsilon[2\cos{x} + \cos{2(x + 2\pi^2)}]$$ for small $\epsilon>0$, and found numerically a relation between the singularity formation time $T_s$ and $\epsilon$: $T_s\sim \epsilon^{-1}$ for $\epsilon\geq 1/\pi$ and $T_s\sim 2.37\epsilon^{-2}$ for $\epsilon\leq 1/2\pi$. The latter enhanced lifespan for smooth initial data $u_0$ such that $\|u_0\|_{H^2(\mathbb{R})}\leq \epsilon\ll 1$ was later proved in \cite{Hunter-Ifrim} by a normal form transformation and in \cite{Hunter-I-T-Wong} by a modified energy method. In contrast, the initial data we construct here are of large $H^2$-norms ($\|u_0\|_{H^2(\mathbb{R})}\gtrsim \|\partial_xu_0\|_{L^\infty(\mathbb{R})}=\epsilon^{-1}$)
and the amplitudes are $O(1)$ in general.

We also mention the recent result of Saut and Wang in \cite{Saut-Wang} in which the authors proved the wave breaking (shock formation) of the BH equation by a different method.
\end{remark}

\begin{remark}[The difficulties and a way to overcome them]\label{rem:difficulty}
Our work is inspired by that of Buckmaster, Shkoller and Vicol \cite{B-S-V1}, \cite{B-S-V2} on shock formation for isentropic compressible Euler equations. Their proofs are based on the modulated self-similar blowup technique, which was developed by Merle and Raphael in \cite{Merle}, \cite{Merle-Raphael} for the nonlinear Schr\"odinger equation, and Merle and Zaag in \cite{Merle-Zaag} for the nonlinear heat equation. This technique is useful to describe the formation of singularities for many equations, including \cite{Collot-Ghoul-Masmoudi} for Burgers equation with transverse viscosity, which is the first paper to use self-similar analysis to study singularity formation on a modified Burgers equation, and  \cite{Elgindi}, \cite{Elgindi-Ghoul-Masmoudi} for 3D incompressible Euler equation with axis-symmetry, and many others. 

The two papers \cite{B-S-V1}, \cite{B-S-V2} on Euler equations serve as a great guide for our problem. Indeed, similarly to \cite{B-S-V1}, \cite{B-S-V2}, the stable stationary solution $\overline{U}$ of the 1D self-similar Burgers equation described in Section \ref{sec:stableprofile} below is the correct asymptotic blowup profile for the BH equation, and the solution $U$ in self-similar variables converges pointwisely to a rescaled version of $\overline{U}$ as the self-similar time $s$ approaches infinity, see Theorem \ref{thm:self}. Still, the blueprint presented there, whose detailed arguments rely heavily on finite speed of propagation, is not a priori applicable to the BH equation. The difficulty comes from the Hilbert transform term.

First, it is not obvious, but is indeed the case after a careful check in Appendix \ref{ap:derivation}, that the Hilbert transform is compatible with the self-similar coordinate transformation $(x,t)\mapsto(X,s)$ defined in \eqref{eq:selftransform}. Second, the solution, in both the original physical variables and the self-similar variables, must lie in Sobolev spaces so that the Hilbert transform is well-defined. This requires $L^2$ estimates as Step 1 in Section \ref{sec:strategy} below. Third, the modulated self-similar blowup technique is a pointwise argument, but the Hilbert transform is a nonlocal operator that is not bounded in $L^\infty$. Nevertheless, this technique is promising: if we start with initial data such that the advection effect $u\partial_xu$ is already very strong, i.e.\ minimum initial slope is $-1/\epsilon$, the non-smoothing Hilbert transform term should not be able to deplete the growth of $\partial_xu$. This is further explained by the $e^{-s}$ factor in front of $H[U+e^\frac{s}{2}\kappa]$ and $H[\partial_X^jU]$ in the equations in self-similar variables \eqref{eq:ansatz}, \eqref{eq:1partialU}, \eqref{eq:npartialU}, which offsets the log factor loss in the pointwise bound of $H[\partial_X^jU]$. On the other hand, the $e^{-s}$ would not be enough to offset the log factor loss had we wished to prove a uniform spatial decay for the spatial derivative of the self-similar profile. The problem occurs in the far field. This creates difficulty for proving the uniqueness of the blowup point, since the blowup point corresponds to the far field behavior. To overcome this, we prove a temporal decay outside a (self-similar-time-dependent) compact set, which is weaker than the uniform spatial decay but still yields uniqueness of the blowup point and H\"older $1/3$ continuity of the blowup profile (see Step 3 in Section \ref{sec:strategy} and Remark \ref{rem:step3} below).

We hope that our methods of proving finite-time blowup of the BH equation can provide some ideas on finite-time blowup results of other non-local fluid dynamic equations. 
\end{remark}

\begin{remark}[Future directions]
As is briefly mentioned in Section \ref{sec:stableprofile}, there is a family of self-similar profiles $U_i$, $i\in\mathbb{N}$ of the inviscid Burgers equation, each of which solves
\begin{align*}
-\frac{1}{2i}U_i+(\frac{2i+1}{2i}X+U_i)\partial_XU_i=0,
\end{align*}
and in this paper, we use the only stable profile $U_1$ (denoted as $\overline{U}$ throughout the paper). The stability allows us to construct shock solutions to the BH equation within an open neighborhood in $H^5(\mathbb{R})$, see Corollary \ref{cor:open}. For the inviscid Burgers equation, one can find initial data leading to blowup solutions asymptotic to each $U_i$ respectively \cite{Collot-Ghoul-Masmoudi}: if $u_0\in\mathcal{C}^\infty(\mathbb{R})$, $\partial_xu_0$ is minimal at $x=0$ with 
\begin{align*}
\partial_xu_0(0)<0,\quad\partial_x^ju_0(0)=0\text{ for }j=2,...,2i,\quad \partial_x^{2i+1}u_0(0)>0,
\end{align*}
then $u$ blows up at time $T=-1/\partial_xu_0(0)$ at $x=0$ and $u$ approaches $U_i$ in the sense described in \cite{Collot-Ghoul-Masmoudi}. 

A natural next question would be to construct blowup solutions to the BH equation that are, after the self-similar transformation, asymptotic to unstable $U_i$'s. However, one needs to be more careful, since $U_i$ for $i\geq 2$ has $2(i-1)$ instability directions. For instance, if we want the solution to approach $U_2$, we need to accurately pick the second and third derivatives $\partial_X^2U(0,-\log\epsilon)$ and $\partial_X^3U(0,-\log\epsilon)$, of the initial data in self-similar variables, to prevent the solution from approaching $U_1$ instead. This will appear in a future paper. Such a construction for the 2D compressible Euler equation has already appeared in \cite{B-I}.

A more interesting yet difficult problem would be to find blowup solutions close to $U_i$ for each $i$ in a unified approach. But we think that bootstrapping alone will not be sufficient and must be combined with other techniques. However, the tools we use here to overcome the nonlocality and lack of pointwise boundedness of the Hilbert transform will also be useful for the resolution of this bigger problem.\label{rem:future}
\end{remark}

\subsection{Strategy of the proof}\label{sec:strategy}
We change the physical variables $(x,t)$ to the modulated self-similar variables $(X,s)$ by the transformation \eqref{eq:selftransform} below, and map $u$ to $U$ by \eqref{eq:uUrelation} accordingly. The role of the dynamic modulation variables is to enforce the constraints \eqref{eq:constraints}: the function $\tau(t)$ provides precise information on the blowup time, the function $\xi(t)$ provides precise information on the blowup location, while $\kappa(t)=u\big(\xi(t),t\big)$ controls the amplitude at the blowup location.

We then rewrite the BH equation into an equivalent self-similar evolution equation \eqref{eq:ansatz} for $U$. We obtain the evolution equations \eqref{eq:tau}, \eqref{eq:kappaxi} for $\tau,\kappa,\xi$ from the constraints \eqref{eq:constraints}. The blowup time $T_*$ corresponds to $s=+\infty$ and the proposed blowup location $x_*:=\xi(T_*)$ corresponds to $X=0$. Shock formation at $x_*$ corresponds to $\partial_XU(0,s)=-1$, and the regularity at other points corresponds to the spatial and temporal decays of $\partial_XU(X,s)$. Thus Theorem \ref{thm:main} follows directly from Theorem \ref{thm:self}.

The proof of Theorem \ref{thm:self} utilizes a bootstrap argument: we first choose a fixed initial datum satisfying some quantitative properties, then assume that the solution satisfies some worsened quantitative properties. Using these assumptions, we a posteriori show that we can improve those bounds by a factor $<1$. Thus those properties can be propagated in time as long as the smooth solution exists, and we close the bootstrap argument and obtain a solution with those properties. The pointwise convergence is not part of the bootstrap as  we will prove it a posteriori.

\textbf{Step 1:} We derive the self-similar time evolution of $\|U(\cdot,s)+e^\frac{s}{2}\kappa\|_{L^2}$ and prove that $\|\partial_XU(\cdot,s)\|_{L^2}$ and $\|\partial^5_XU(\cdot,s)\|_{L^2}$ are bounded for all $s$, thus showing a unique solution in $H^5(\mathbb{R})$ exists for all $s$. We will also use these $L^2$-estimates to bound the Hilbert transform terms. We now explain why we need to go up to the 5th order derivative. In the next step, we will need $L^\infty$-information on $\partial^j_XU$ for $j=0,...,4$. To give an $L^\infty$-estimate of $\partial_X^jU(\cdot,s)$ we need to estimate $\|H[\partial_X^jU](\cdot,s)\|_{L^\infty}$ as it appears in the forcing term, and this requires one more derivative. 

\textbf{Step 2:} We prove that $U(\cdot,s)$ is close to the 1D self-similar Burgers profile $\overline{U}$ defined as \eqref{eq:barUformula} in $\mathcal{C}^1([-\frac{1}{2}e^{\frac{3}{2}s},\frac{1}{2}e^{\frac{3}{2}s}])$. To this end, we deal with the region close to 0 and the middle field away from zero separately. For the region close to 0, we first estimate the $\partial_X^4\widetilde{U}$  near 0, where $\widetilde{U}:=U-\overline{U}$, then use the Fundamental Theorem of Calculus repeatedly to estimate the 3rd to 0th derivative of $\widetilde{U}$. For the middle field, we use the exponential escape of the Lagrangian trajectories associated with $U$ away from 0,  
together with Gr\"onwall type transport estimates. The transport estimates are weighted estimates on $\widetilde{U}$, $\partial_X\widetilde{U}$, and $\partial^2_XU$. For $\partial_X^3U$ we do not need a weight, and we must stop at some derivative so that we are able to close the bootstrap argument. In principle, a non-weighted bound for $\partial_X^2U$ would be enough, because it is not tied to the blowup, but since it appears in the forcing term for $\partial_X^3U$, in order to prove the boundedness of $\partial_X^3U$ we need a certain decay of $\partial_X^2U$ as well. An important observation about the forcing terms, in particular, the Hilbert transform part, is that we can use the temporal decay $e^{-s}$ to kill the log factor from the pointwise bound on the Hilbert transform term and also gain some spatial decay in the range $|X|\leq \frac{1}{2}e^{\frac{3}{2}s}$.

\textbf{Step 3:} The behavior of $U(X,s)$ for $(X,s)$ such that $|X|\geq \frac{1}{2}e^{\frac{3}{2}s}$ is different. Since $\overline{U}$ is not  $L^2$-integrable, while $U(\cdot,s)+e^\frac{s}{2}\kappa\in L^2(\mathbb{R})$, we do not expect $U$ and $\overline{U}$ to be close at large $X$ (we do expect them to be close within a larger and larger range when $s$ increases as in Step 2, since $\|U(\cdot,s)+e^\frac{s}{2}\kappa\|_{L^2}$ is growing in $s$). Instead, we prove that $\|U(\cdot,s)+e^\frac{s}{2}\kappa\|_{L^\infty}$ is bounded. For the derivatives of $U$, we cannot gain any spatial decay for the forcing terms from the temporal decay $e^{-s}$ any more. In order to still be able to achieve uniqueness of the blowup point, we prove that $e^s\partial_XU$, $e^s\partial_X^2U$ and $\partial_X^3U$ are bounded. Again the temporal decay for $\partial_X^2U$ is needed to estimate $\partial_X^3U$. 

\textbf{Step 4:} We prove the estimates for the time derivatives of $\tau$ and $\xi$ thus obtaining precise information on the blowup time and location. 

\textbf{Step 5:} After closing the bootstrap, we prove the pointwise convergence of $U(X,s)$ to $\overline{U}_\nu$, a rescaled version of $\overline{U}$ defined in Theorem \ref{thm:self}. The rescaling is for matching the third spatial derivative at $X=0$. To do so, we consider the equation \eqref{eq:eqtildenu} for the difference $\widetilde{U}_\nu:=U-\overline{U}_\nu$, then use the Taylor expansion of $\widetilde{U}_\nu$ at $X=0$ and the (new) Lagrangian trajectory associated to \eqref{eq:eqtildenu} to propagate to all $X\in\mathbb{R}$.  

\begin{remark}\label{rem:step3}
Step 3 is what differs most from previous treatments concerning the compressible Euler equations \cite{B-S-V1}, \cite{B-S-V2}. Since the self-similar evolution equation for the 3D compressible Euler equation has finite speed of propagation, solutions have (time-dependent) compact support if the initial data have compact support. From this, one can use a commutator estimate with the weight $(1+X^2)^{1/3}$ to gain spatial decays of higher order derivatives from an even higher Sobolev norm along with decays of lower order derivatives. Hence, the relevant forcing terms  have the desired spatial decay even though bounding them involve higher order derivatives. The very high Sobolev norm and compact support are what close the argument.    

However, as we state in Remark \ref{rem:difficulty} above, the self-similar evolution equation \eqref{eq:ansatz} for the BH equation does not have finite speed of propagation due to the presence of the nonlocal Hilbert transform in \eqref{eq:BH}. Since $(1+X^2)^{1/3}$ is unbounded or not even in BMO (bounded mean oscillation), no commutator estimate is available for higher order derivatives on the whole real line. Without decay of $\partial_X^2U$, if we suppose $|\partial_XU(X,s)|\leq C(1+X^2)^{-\frac{1}{3}}$ for all $(X,s)$, then the best bound we can find for $H[\partial_XU](X,s)$ when $(X,s)$ is such that $|X|\geq \frac{1}{2}e^{\frac{3}{2}s}$ is
\begin{align*}
|H[\partial_XU](X,s)|\lesssim \|\partial_X^2U(\cdot,s)\|_{L^\infty}(1+X^2)^{-\frac{1}{3}}+(1+X^2)^{-\frac{1}{3}}\log(1+X^2).
\end{align*}
The forcing term of $(1+X^2)^\frac{1}{3}\partial_XU$ is then $\lesssim e^{-s}\big(\|\partial_XU^2(\cdot,s)\|_{L^\infty}+\log(1+X^2)\big)$, which is $s$-integrable with the upper bound of Lagrangian trajectory (Lemma \ref{lem:upperLagrangian}), but we will end up having an uncontrollable $\log|X_0|$ term. So we will not be able to close the bootstrap argument in this case\footnote{The log factor is not a problem for estimating $(1+X^2)^\frac{1}{3}\partial_XU$ or $(1+X^2)^\frac{1}{3}\partial_X\widetilde{U}$ for $(X,s)$ such that $l\leq |X|\leq \frac{1}{2}e^{\frac{3}{2}s}$, since for $X$ in the middle field, $\log(1+X^2)\lesssim s$, and $\int_{s_0}^ss'e^{-s'}\,dx'\leq (s_0+1)e^{-s_0}$ is very small.}.

Instead, if we follow Step 3, then as an intermediate step of \eqref{eq:H1xUfar} we get a term $se^{-s}$ (the $s$ comes from $\log{e^{-s}}$), which is not only $s$-integrable but also small uniformly in $X$. A byproduct of the weaker far field assumption in Step 3 is that we get a $(1+X^2)^{-\frac{1}{4}}$ term in the bound of $H[\partial_XU](X,s)$ for $l\leq |X|\leq \frac{1}{2}e^{\frac{3}{2}s}$ (see \eqref{eq:H1xUmiddle}). This smaller exponent results from the use of H\"older's inequality. This bound is not sharp, and when multiplying by the weight $(1+X^2)^\frac{1}{3}$, we will get $(1+X^2)^\frac{1}{12}$. Nevertheless, we can use a fractional power of $e^{-s}$ to kill the positive power of $X$.
\end{remark}

\subsection{Paper outline}
The paper is organized as follows:
\begin{itemize}[leftmargin=*]
\item In Section \ref{sec:self}, we carry out the modulated self-similar transformation and rewrite the BH equation into the evolution equation \eqref{eq:ansatz} for the self-similar variable $U$. We record the explicit expression and quantitative properties of the stable steady-state solution to the self-similar Burgers equation. We also record the equations for $\partial_XU,...,\partial^5_XU$, and equations for $\widetilde{U},\partial_X\widetilde{U},...,\partial^4_X\widetilde{U}$ for later use. Imposing constraints on $U$ we get evolution equations for the dynamic modulation variables $\xi,\tau,\kappa$.
\item In Section \ref{sec:result}, we state the assumptions on the initial data both in the original physical variables and in the self-similar variables. Then we state Theorem \ref{thm:self} in terms of the self-similar variables, and our main result, Theorem \ref{thm:main}. We also state in Corollary \ref{cor:open} that the initial data can be taken in an open set in the $H^5$-topology.
\item In Section \ref{sec:L2estimate}, we prove the $L^2$-bounds on $U+e^\frac{s}{2}\kappa$, $\partial_XU$ and $\partial^5_XU$. This section is the Step 1 of the bootstrap argument described in the previous section.
\item In Section \ref{sec:forcingbound}, we record bounds for all the forcing terms for later use. These include forcing terms of $\partial^4_X\widetilde{U}$ for $X$ near 0, forcing terms of $\widetilde{U},\partial_X\widetilde{U}, \partial_X^2U,\partial_X^3U$ for $X$ in the middle field, the forcing terms of $\partial_XU,\partial_X^2U,\partial_X^3U$ for $X$ in the far field, and the forcing term of $U$ uniformly in $X$.
\item In Section \ref{sec:closure}, we close our bootstrap assumptions stated in Section \ref{sec:assumptionself} and \ref{sec:assumptiondynamic}. This section is the Step 2 to 4 of the bootstrap argument. 
\item In Section \ref{sec:proofmain}, we prove the pointwise convergence part of Theorem \ref{thm:self}. Then we prove our main result, Theorem \ref{thm:main}, from the bootstrap part of Theorem \ref{thm:self}. We also prove Corollary \ref{cor:open}.
\item In Appendix A, we prove the local well-posedness theorem. In Appendix B, we derive the evolution equation for the self-similar variable step by step. In Appendix C, we list a few interpolation lemmas. In Appendix \ref{ap:simple}, we state and prove two simple lemmas used in the bootstrap argument.
\end{itemize}

\subsection*{Acknowledgment}
The author would like to thank Tristan Buckmaster for suggesting the problem, both Tristan Buckmaster and Gigliola Staffilani for providing helpful suggestions and comments, and referees and various other people for their careful reading and suggestions.

\section{Self-similar transformation}\label{sec:self}
Since the BH equation is translation invariant in time, we can set the initial time $t_0=-\epsilon$ for some small parameter $\epsilon$ to be determined later\footnote{This choice of initial time is only for convenience. If we set $t_0=0$ instead, then we need to change the initial condition of $\tau$ to be $\tau(0)=\epsilon$ in order to properly define the self-similar time $s$, but all the arguments are unaffected.}. Let $u_0=u(\cdot,-\epsilon)$ be the initial data. Let $T_*$ denote the maximal time of existence of the solution $u$. 

We perform a self-similar transformation 
\begin{align}
X:=\frac{x-\xi(t)}{(\tau(t)-t)^\frac{3}{2}},\qquad s:=-\log{\big(\tau(t)-t\big)},\label{eq:selftransform}
\end{align}
and define the self-similar variable $U$ by
\begin{equation}
\begin{aligned}
u(x,t)&=\big(\tau(t)-t\big)^\frac{1}{2}U\bigg(\frac{x-\xi(t)}{(\tau(t)-t)^\frac{3}{2}},-\log\big(\tau(t)-t\big)\bigg)+\kappa(t)\\
&=e^{-\frac{s}{2}}U(X,s)+\kappa(t),
\end{aligned}\label{eq:uUrelation}
\end{equation}
where the dynamic modulation variables $\xi,\tau,\kappa:[-\epsilon,T_*]\to\mathbb{R}$ control the shock location, blowup time and wave amplitude, respectively. We require that at time $t_0=-\epsilon$,
\begin{align*}
    \tau(-\epsilon)=0,\quad\xi(-\epsilon)=0,\quad \kappa(-\epsilon)=\kappa_0=u_0(0),
\end{align*}
and $\tau(T_*)=T_*$. We will design $u_0$ so that $T_*\sim o(\epsilon)$ and $\tau(t)>t$ for all $t\in [-\epsilon,T_*)$. In the self-similar time $s$, the blowup as $t\to T_*$ corresponds to $s\to+\infty$.

By \eqref{eq:selftransform} and \eqref{eq:uUrelation}, we have: 
\begin{align}
\partial^j_xu(x,t)&=e^{(-\frac{1}{2}+\frac{3}{2}j)s}\partial_X^jU(X,s)\nonumber\\
\implies\qquad \|\partial_x^ju(\cdot,t)\|_{L^2}&=e^{(-\frac{5}{4}+\frac{3}{2}j)s}\|\partial_X^jU(\cdot,s)\|_{L^2}.\label{eq:uUL2relation}
\end{align}

Plugging \eqref{eq:selftransform} into \eqref{eq:BH}, we obtain that $U$ satisfies the equation
\begin{equation}
\big(\partial_s-\frac{1}{2}\big)U+\Big(\frac{U+e^\frac{s}{2}(\kappa-\dot{\xi})}{1-\dot{\tau}}+\frac{3}{2}X\Big)\partial_XU=-\frac{e^{-\frac{s}{2}}\dot{\kappa}}{1-\dot{\tau}}+\frac{e^{-s}H[U+e^\frac{s}{2}\kappa]}{1-\dot{\tau}},\label{eq:ansatz}
\end{equation}
where $\dot{f}=df/dt$. The detailed calculation can be found in Appendix \ref{ap:derivation}. It is very important that the self-similar transformation works well with the Hilbert transform, which is what makes the whole argument even possible. 

The two equations for $U$ and $u$ are equivalent as long as the solutions remain smooth. In particular, if $u_0\in H^5(\mathbb{R}) \subset \mathcal{C}^4_c(\mathbb{R})$, then \eqref{eq:ansatz} also has $H^5(\mathbb{R})\subset  \mathcal{C}^4_c(\mathbb{R})$ initial data and is well-posed in $\mathcal{C}^4$ and $H^5$-norm. We will show that $U$ is asymptotically close to the stable steady state self-similar Burgers solution $\overline{U}$ defined in \eqref{eq:barUformula} in the next section in a manner to be made precise below. It is also convenient to introduce the difference 
\begin{align*}
\widetilde{U}(X,s):=U(X,s)-\overline{U}(X).
\end{align*}

\subsection{The stable self-similar solution to the inviscid Burgers equation}\label{sec:stableprofile}
The inviscid Burgers equation
\begin{align*}
\partial_tu+u\partial_xu=0
\end{align*}
has a family of self-similar solutions
\begin{align*}
u(x,t)=(-t+T_*)^\frac{1}{2i} U_{i}\Big(\frac{x-x_*}{(-t+T_*)^{\frac{1}{2i}+1}}\Big)
\end{align*}
for all $i\in\mathbb{N}$, where $T_*$ is the blowup time and $x_*$ is the blowup location. It turns out that only $\overline{U}=U_{1}$ is a stable blowup self-similar profile \cite{Eggers}; other profiles $U_i$ for $i\geq 2$ are all unstable. The profile $\overline{U}$ solves the steady-state self-similar Burgers equation\footnote{We use $\partial_X$ even though $\overline{U}$ is a function of only one variable $X$ to align with $U$.}
\begin{align}
-\frac{1}{2}\overline{U}+\Big(\frac{3}{2}X+\overline{U}\Big)\partial_X\overline{U}=0.\label{eq:selfBurgers}
\end{align}
This profile has an explicit expression
\begin{align}
\overline{U}(X)=\Big(-\frac{X}{2}+\big(\frac{1}{27}+\frac{X^2}{4}\big)^\frac{1}{2}\Big)^\frac{1}{3}-\Big(\frac{X}{2}+\big(\frac{1}{27}+\frac{X^2}{4}\big)^\frac{1}{2}\Big)^\frac{1}{3}.\label{eq:barUformula}
\end{align}
This closed form is very important for our analysis. We note that $\overline{U}$ has Taylor expansion
\begin{align*}
   \text{for }X\approx 0:\quad \overline{U}(X)&=-X+X^3-3X^5+O(X^7),\\
    \partial_X\overline{U}(X)&=-1+3X^2-15X^4+O(X^5),\\
    \text{for }X\approx\infty:\quad\overline{U}(X)&=-X^\frac{1}{3}+\frac{1}{3}X^{-\frac{1}{3}}-\frac{1}{81}X^{-\frac{5}{3}}+O(X^{-\frac{7}{3}}),\\
    \partial_X\overline{U}(X)&=-\frac{1}{3}X^{-\frac{2}{3}}-\frac{1}{9}X^{-\frac{4}{3}}+\frac{5}{243}X^{-\frac{7}{3}}+O(X^{-\frac{10}{3}}).
\end{align*}
In particular, 
\begin{align}
\overline{U}(0)=0,\quad\partial_X\overline{U}(0)=-1,\quad\partial^2_X\overline{U}(0)=0,\quad\partial^3_X\overline{U}(0)=6.\label{eq:barUat0}
\end{align}
Examining \eqref{eq:barUformula}, we can find that for all $X\in\mathbb{R}$,
\begin{align}
|\overline{U}(X)|\leq (1+X^2)^\frac{1}{6},\quad|\partial_X\overline{U}(X)|\leq (1+X^2)^{-\frac{1}{3}},\quad|\partial^2_X\overline{U}(X)|\leq (1+X^2)^{-\frac{5}{6}}.\label{eq:barUall}
\end{align}
For $|X|\geq 100$,
\begin{align}
\frac{1}{4}(1+X^2)^{-\frac{1}{3}}\leq |\partial_X\overline{U}(X)|\leq  \frac{7}{20}(1+X^2)^{-\frac{1}{3}}.\label{eq:1xbarUfar}
\end{align}
For $|X|\leq \frac{1}{5},$\begin{gather}
\begin{aligned}
|\overline{U}(X)|\leq \frac{1}{6},\quad |\partial_X\overline{U}(X)|\leq 1,\quad |\partial_X^2\overline{U}(X)|\leq 1,\\
|\partial^3_X\overline{U}(X)|\leq 6,\quad
|\partial^4_X\overline{U}(X)|\leq 30,\quad |\partial_X^5\overline{U}(X)|\leq 360.
\end{aligned}\label{eq:barUnear0}
\end{gather}
And for $|X|\geq l$ whenever $0<l<\frac{1}{5}$,
\begin{align}
|\partial_X\overline{U}(X)|\leq (1-2l^2)(1+X^2)^{-\frac{1}{3}}.\label{eq:1xbarUmiddle}
\end{align}

\begin{remark}
Unlike the Burgers equation, the BH equation \eqref{eq:BH} does not have self-similar solutions. To see this, without loss of generality assume the blowup time $T_*=0$ and the blowup location $x_*=0$ since the BH equation is invariant under both time and space translations. Plugging in the ansatz $$u(x,t)=(-t)^\alpha U\Big(\frac{x}{(-t)^\beta}\Big),\qquad t<0,\ \alpha,\beta\text{ not both 0}$$ into \eqref{eq:BH}, and writing $X=x/(-t)^\beta$, we would get
$$ \frac{\alpha}{t}U(X)-\frac{\beta}{t}X\partial_XU(X)+(-t)^{\alpha-\beta}U(X)\partial_XU(X)=H[U](X).$$
But there is no way to eliminate $t$ in the above expression.
\end{remark}

\subsection{Equations of $\partial^n_XU$ for $n=1,...,5$}
Here we record the equations of $\partial^n_XU$ for $n=1,\dots,5$ by differentiating \eqref{eq:ansatz} repeatedly. For convenience, we introduce the transport speed
\begin{align}
V:=\frac{U+e^\frac{s}{2}(\kappa-\dot{\xi})}{1-\dot{\tau}}+\frac{3}{2}X.\label{eq:speed}
\end{align}
Then we have
\begin{align}
    \Big(\partial_s+1+\frac{\partial_XU}{1-\dot{\tau}}\Big)\partial_XU+V\partial^2_XU&=\frac{e^{-s}}{1-\dot{\tau}}H[\partial_XU],\label{eq:1partialU}\\
\big(\partial_s+\frac{3}{2}n-\frac{1}{2}+\frac{n+1}{1-\dot{\tau}}\partial_XU\big)\partial_X^nU+V\partial_X^{n+1}U&=F_{\partial_X^nU},\qquad n=2,3,4,5, \label{eq:npartialU}
\end{align}
where the forcing terms are
\begin{align}
F_{\partial_X^nU}:=\frac{e^{-s}}{1-\dot{\tau}}H[\partial_X^nU]+\mathbbm{1}_{\{n\text{ odd}\}}\frac{1}{1-\dot{\tau}}{n\choose \frac{n-1}{2}}(\partial_X^\frac{n+1}{2}U)^2+\frac{1}{1-\dot{\tau}}\sum_{j=2}^{\lfloor \frac{n}{2}\rfloor}{n+1\choose j}\partial_X^jU\partial_X^{n-j+1}U.
\end{align}
%
%

\subsection{Equations of $\partial^n_X\widetilde{U}$ for $n=0,\dots,4$}
We also need the equation of $\widetilde{U}$ and its derivatives up to the 4-th order. We have
\begin{align}
    \Big(\partial_s-\frac{1}{2}+\frac{\partial_X\overline{U}}{1-\dot{\tau}}\Big)\widetilde{U}+V\partial_X\widetilde{U}&=F_{\tilde{U}},\label{eq:eqtildeU}\\
    \Big(\partial_s+1+\frac{2\partial_X\overline{U}+\partial_X\widetilde{U}}{1-\dot{\tau}}\Big)\partial_X\widetilde{U}+V\partial^2_X\widetilde{U}&=F_{\partial_X\tilde{U}},\label{eq:eq1xtildeU}\\
\big(\partial_s+\frac{3}{2}n-\frac{1}{2}+\frac{n+1}{1-\dot{\tau}}\partial_XU\big)\partial_X^n\widetilde{U}+V\partial_X^5\widetilde{U}&=F_{\partial_X^n\tilde{U}},\qquad n=2,3,4,\label{eq:eqnxtildeU}
\end{align} 
where the forcing terms are 
\begin{align}
F_{\tilde{U}}:=&-\frac{e^{-\frac{s}{2}}\dot{\kappa}}{1-\dot{\tau}}+\frac{e^{-s}}{1-\dot{\tau}}H[U+e^\frac{s}{2}\kappa]-\partial_X\overline{U}\frac{\dot{\tau}\overline{U}+e^\frac{s}{2}(\kappa-\dot{\xi})}{1-\dot{\tau}},\\
F_{\partial_X\tilde{U}}:=&\frac{e^{-s}}{1-\dot{\tau}}H[\partial_XU]-\frac{\dot{\tau}}{1-\dot{\tau}}(\partial_X\overline{U})^2-\Big(\frac{\widetilde{U}+\dot{\tau}\overline{U}}{1-\dot{\tau}}+\frac{e^\frac{s}{2}(\kappa-\dot{\xi})}{1-\dot{\tau}}\Big)\partial^2_X\overline{U},\\
F_{\partial_X^n\tilde{U}}:=&\frac{e^{-s}}{1-\dot{\tau}}H[\partial_X^nU]-\frac{\dot{\tau}}{1-\dot{\tau}}\Big[\sum_{j=0}^{\lfloor \frac{n}{2}\rfloor}{n+1\choose j}\partial_X^j\overline{U}\partial_X^{n-j+1}\overline{U}+{n\choose \frac{n+1}{2}}\mathbbm{1}_{\{n\text{ odd}\}}(\partial_X^\frac{n+1}{2}\overline{U})^2\Big]\label{eq:ntildeUforce}\\
&-\frac{1}{1-\dot{\tau}}\Big[e^\frac{s}{2}(\kappa-\dot{\xi})\partial_X^{n+1}\overline{U}+\sum_{j=0}^{n-1}{n+1\choose j}\partial_X^j\widetilde{U}\partial_X^{n+1-j}\overline{U}\nonumber\\
&+\sum_{j=2}^{\lfloor\frac{n}{2}\rfloor}{n+1 \choose j}\partial_X^j\widetilde{U}\partial_X^{n+1-j}\widetilde{U}+\mathbbm{1}_{\{n\text{ odd}\}}{n\choose \frac{n+1}{2}}(\partial_X^\frac{n+1}{2}\widetilde{U})^2  \Big],\qquad n=2,3,4.\nonumber
\end{align}

\subsection{Constraints on $U$ and equations of dynamic modulation variables}
As mentioned in Section \ref{sec:strategy}, we impose the following constraints at $X=0$
\begin{align}
U(0,s)=0,\quad \partial_XU(0,s)=-1,\quad\partial_X^2U(0,s)=0.
\label{eq:constraints}
\end{align}
Plugging these constraints into \eqref{eq:uUrelation}, \eqref{eq:ansatz}, \eqref{eq:1partialU} and \eqref{eq:npartialU} with $n=2$, we obtain the equations of the dynamic variables
\begin{gather}
\kappa(t)=u\big(\xi(t),t\big),\label{eq:kappa}\\
\dot{\tau}=e^{-s}H[\partial_XU](0,s),\label{eq:tau}\\
e^\frac{s}{2}(\kappa-\dot{\xi})=e^{-\frac{s}{2}}\dot{\kappa}-e^{-s}H[U+e^\frac{s}{2}\kappa](0,s)=\frac{e^{-s}H[\partial_X^2U](0,s)}{\partial^3_XU(0,s)}.\label{eq:kappaxi}
\end{gather}
Recall that
\begin{align*}
    \tau(-\epsilon)=0,\quad\xi(-\epsilon)=0,\quad \kappa(-\epsilon)=\kappa_0=u_0(0),
\end{align*}
and $\tau(T_*)=T_*$. We define $x_*:=\xi(T_*)$ and we will show that it is the unique singular point.

\section{Main results}\label{sec:result}
We introduce a large parameter $M\geq 1$ to be determined later. As we mentioned above, by a time translation we can take $t_0=-\epsilon$ for some small $\epsilon=\epsilon(M)\leq 10^{-4}$ to be determined later, hence the self-similar time $s\geq -\log\epsilon
$. 

We also choose a small scale $l=l(M)=(\log{M})^{-2}$ and a large scale $\frac{1}{2}e^{\frac{3}{2}s}$.
Note that the large scale is $s$-dependent. The self-similar time-dependence of the large scale captures the asymptotic self-similarity.

\subsection{Initial data in physical variables}
We choose the initial data $u_0=u(\cdot,-\epsilon)$ to be such that
\begin{gather}
\|u_0\|_{L^\infty}\leq \frac{1}{2}M,\label{eq:u0infty}\\
\mathrm{supp}\,u_0\subset [-1,1],\label{eq:u0supp}\\
-\|\partial_xu_0\|_{L^\infty}=\min_x \partial_x u_0=\partial_xu_0(0)=-\epsilon^{-1},\\
\partial_x^2u_0(0)=0,\\
\|\partial_x^2u_0\|_{L^\infty}\leq 2\epsilon^{-\frac{5}{2}},\\
|\partial_x^3u_0(0)-6\epsilon^{-4}|\leq \frac{1}{4}\epsilon^{-\frac{15}{4}},\\
\|\partial_x^3u_0\|_{L^\infty}\leq \frac{1}{2}M^\frac{3}{4}\epsilon^{-4},\\
\|\partial_x^5u_0\|_{L^2}\leq \frac{1}{2}M^4\epsilon^{-\frac{25}{4}}.
\end{gather}
In particular, by \eqref{eq:u0infty} and \eqref{eq:u0supp} we have
\begin{align}
\|u_0\|_{L^2}\leq\frac{\sqrt{2}}{2}M.\label{eq:u0L2}
\end{align}
Hence, the initial data $u_0\in H^5(\mathbb{R})$.
\label{sub:initial physical}

\subsection{Initial data in self-similar variable}\label{sub:initialself}
Directly from Section \ref{sub:initial physical}, the transformation \eqref{eq:uUrelation}, \eqref{eq:uUL2relation} and the constraint \eqref{eq:constraints} at $s=-\log\epsilon$, we have
\begin{gather}
\|U(\cdot,-\log\epsilon)+\epsilon^{-\frac{1}{2}}\kappa_0\|_{L^\infty}\leq \frac{1}{2}M\epsilon^{-\frac{1}{2}},\\
\mathrm{supp}\,U(\cdot,-\log\epsilon)+\epsilon^{-\frac{1}{2}}\kappa_0\subset[-\epsilon^{-\frac{3}{2}},\epsilon^{-\frac{3}{2}}],\label{eq:Usupport}\\
U(0,-\log\epsilon)=0,\label{eq:initU0}\\
-\|\partial_XU(\cdot,-\log\epsilon)\|_{L^\infty}=\min_X\partial_XU(X,-\log\epsilon)=\partial_XU(0,-\log\epsilon)=-1,\label{eq:init1xU0}\\
\partial_X^2U(0,-\log\epsilon)=0,\label{eq:init2xU0}\\
\|\partial_X^2U(\cdot,-\log\epsilon)\|_{L^\infty}\leq 2,\label{eq:init2xUfar}\\
|\partial^3_X\widetilde{U}(0,-\log\epsilon)|\leq \frac{1}{4} \epsilon^\frac{1}{4},\label{eq:initial3xtildeU}\\
\|\partial_X^3U(\cdot,-\log\epsilon)\|_{L^\infty}\leq \frac{1}{2}M^\frac{3}{4},\label{eq:init3xUfar}\\
\|\partial_X^5U(\cdot,-\log\epsilon)\|_{L^2}\leq \frac{1}{2}M^4.
\end{gather}
We further impose the following assumptions: for $|X|\leq \frac{1}{2}\epsilon^{-\frac{3}{2}}$, 
\begin{align}
|\widetilde{U}(X,-\log\epsilon)|&\leq \epsilon^\frac{1}{10}(1+X^2)^\frac{1}{6},\label{eq:inittildeUmiddle}\\
|\partial_X\widetilde{U}(X,-\log\epsilon)|&\leq \epsilon^\frac{1}{11}(1+X^2)^{-\frac{1}{3}},\label{eq:init1xtildeUmiddle}\\
|\partial_X^2U(X,-\log\epsilon)|&\leq 2(1+X^2)^{-\frac{1}{3}},\label{eq:init2xUmiddle}\\
|\partial^3_XU(X,-\log\epsilon)|&\leq M^\frac{1}{2}.\label{eq:init3xmiddle}
\end{align}
For $|X|\geq \epsilon^{-\frac{3}{2}}$, we assume\footnote{This is compatible with the assumptions \eqref{eq:init1xtildeUmiddle} at $|X|=\frac{1}{2}\epsilon^{-\frac{3}{2}}$, since by \eqref{eq:1xbarUfar}
\begin{align*}
|\partial_XU(\pm\frac{1}{2}\epsilon^{-\frac{3}{2}},-\log\epsilon)|\leq (\epsilon^\frac{1}{11}+\frac{7}{20})\big(1+\frac{1}{4}\epsilon^{-3}\big)^{-\frac{1}{3}}\leq\frac{2}{5}\cdot 2^\frac{2}{3}\epsilon\leq \frac{3}{4}\epsilon.
\end{align*}
The compatibility of \eqref{eq:init2xUfar} with \eqref{eq:init2xUmiddle} is trivial.} 
\begin{align}
|\partial_XU(X,-\log\epsilon)|\leq \frac{3}{4}\epsilon,\label{eq:init1xUfar}\\
|\partial_X^2U(X,-\log\epsilon)|\leq 
2M^\frac{1}{4}\epsilon.\label{eq:init2xUfar}
\end{align}
For $|X|\leq l$,
\begin{align}
|\partial_X^4\widetilde{U}(X,-\log\epsilon)|\leq \epsilon^\frac{1}{8}.\label{eq:initial4xtildeU}
\end{align}
As a consequence of \eqref{eq:init1xtildeUmiddle},\eqref{eq:init1xUfar},\eqref{eq:barUall}, we have\footnote{The numerical value of $\sqrt{\pi}\Gamma(\frac{1}{6})/\Gamma(\frac{2}{3})$ is 7.28595.}
\begin{align}
\|\partial_XU(\cdot,-\log\epsilon)\|_{L^2}&\leq \Big(\int_{|X|\leq \frac{1}{2}\epsilon^{-\frac{3}{2}}}(1+\epsilon^\frac{1}{11})^2(1+X^2)^{-\frac{2}{3}}\,dX+\int_{\frac{1}{2}\epsilon^{-\frac{3}{2}}\leq|X|\leq \epsilon^{-\frac{3}{2}}}\epsilon^2\,dX\Big)^\frac{1}{2}\\
&\leq \Big((1+\epsilon^\frac{1}{11})^2\int_{\mathbb{R}} (1+X^2)^{-\frac{2}{3}}\,dX+\epsilon^\frac{1}{2}\Big)^\frac{1}{2}\\
&\leq \Big((1+\epsilon^\frac{1}{11})^2\frac{\sqrt{\pi}\Gamma(\frac{1}{6})}{\Gamma(\frac{2}{3})}+\epsilon^\frac{1}{2}\Big)^\frac{1}{2}\leq 2\sqrt{2}
\label{eq:init1xUL2}.
\end{align}
And by Lemma \ref{lem:Sinterpolation}, interpolating between $\partial_XU$ and $\partial_X^5U$, we also get
\begin{align*}
\|\partial_X^jU(\cdot,-\log\epsilon)\|_{L^2}&\leq \|\partial_XU(\cdot,-\log\epsilon)\|_{L^2}^{\frac{5}{4}-\frac{j}{4}}\|\partial_X^5U(\cdot,-\log\epsilon)\|_{L^2}^{\frac{j}{4}-\frac{1}{4}}\leq (2\sqrt{2})^{\frac{5}{4}-\frac{j}{4}}M^{j-1},\quad j=2,3,4.
\end{align*}

\begin{remark} Later we will relax these initial data assumptions in both the physical variables and self-similar variables. We will only require $u_0$ to be close to the initial data given by this and the previous section in the $H^5$-topology, as stated in Corollary \ref{cor:open} below.
\end{remark}

\subsection{Bootstrap assumptions for the self-similar variable}\label{sec:assumptionself}
For $|X|\leq l$, we assume
\begin{gather}
|\partial_X^j\widetilde{U}(X,s)|\leq (\epsilon^\frac{1}{8}+\log M\epsilon^\frac{1}{10})l^{4-j}\leq 2\log M\epsilon^\frac{1}{10}l^{4-j},\quad j=0,1,2,3,\label{eq:xtildeUnear0}\\
|\partial^4_X\widetilde{U}(X,s)|\leq\epsilon^\frac{1}{10}.\label{eq:4xtildeUnear0}
    \end{gather}
At $X=0$, we assume that
    \begin{align}
    |\partial_X^3\widetilde{U}(0,s)|\leq \epsilon^\frac{1}{4}.\label{eq:3xtildeUat0}
    \end{align}
For any $(X,s)$ such that $l\leq |X|\leq\frac{1}{2}e^{\frac{3}{2}s}$, we assume that
\begin{align}
|\widetilde{U}(X,s)|&\leq \epsilon^\frac{1}{11}(1+X^2)^\frac{1}{6},\label{eq:tildeUmiddle}\\
|\partial_X\widetilde{U}(X,s)|&\leq \epsilon^\frac{1}{12}(1+X^2)^{-\frac{1}{3}},\label{eq:1xtildeUmiddle}\\
|\partial_X^2U(X,s)|&\leq M^\frac{1}{4}(1+X^2)^{-\frac{1}{3}},\label{eq:2xUmiddle}\\
|\partial_X^3U(X,s)|&\leq \frac{1}{2}M^\frac{3}{4}\label{eq:3xUmiddle}.
\end{align}
And for any $(X,s)$ such that $|X|\geq \frac{1}{2}e^{\frac{3}{2}s}$, we assume that\footnote{By Lemma \ref{lem:GNSinterpolation}, interpolating between $\|\partial_XU(\cdot,s)\|_{L^\infty}$ and $\|\partial^5_X(\cdot,s)\|_{L^5}$, we can already bound $\|\partial_X^jU(\cdot,s)\|_{L^\infty}$ for $j=2,3,4$. But we will not be able to close the bootstrap when doing the $L^2$-estimate of $\partial_X^5U$ if we only use interpolation. So we must impose a better (in terms of the power of $M$) $L^\infty$-bound on $\partial_X^3U$.}
\begin{align}
|\partial_XU(X,s)|&\leq 2e^{-s},\label{eq:1xUfar}\\
|\partial^2_XU(X,s)|&\leq 4M^\frac{1}{4}e^{-s},\label{eq:2xUfar}\\
|\partial^3_XU(X,s)|&\leq M^\frac{3}{4}\label{eq:3xUfar}.
\end{align}
We also assume that
\begin{align}
\|\partial_XU(\cdot,s)\|_{L^2}&\leq 3,\label{eq:1xUL2}\\
    \|\partial_X^5U(\cdot,s)\|_{L^2}&\leq M^{4},\label{eq:5xUL2}\\
    \|U(\cdot,s)+e^\frac{s}{2}\kappa\|_{L^\infty}&\leq Me^\frac{s}{2}.\label{eq:ULinfty}
\end{align}

\subsection{Bootstrap assumptions for the dynamic modulation variables}\label{sec:assumptiondynamic}
We assume that
\begin{align}
|\dot{\tau}(t)|\leq e^{-\frac{3}{4} s},\quad |\tau(t)|\leq 2\epsilon^\frac{7}{4},\quad |T_*|\leq 2\epsilon^\frac{7}{4}.\label{eq:taubound}
\end{align}
We will show in Section \ref{sec:blowuptime} that, as a result, $\tau(t)>t$ for all $t\in[-\epsilon,T_*)$. 

For the blowup location, we assume that
\begin{align}
|\dot{\xi}(t)|\leq 2M,\quad |\xi(t)|\leq 3M\epsilon.\label{eq:xibound}
\end{align}

We do not need to make assumptions on the wave amplitude $\kappa$, because first, in our closure of bootstrap, we can eliminate it using \eqref{eq:kappaxi}; second, by \eqref{eq:ULinfty}, $$\|u(\cdot,t)\|_{L^\infty}=\|e^{-\frac{s}{2}}
U(\cdot,s)+\kappa(t)\|_{L^\infty}\leq M$$ for all $t\in[-\epsilon,T_*]$, and by \eqref{eq:kappa},
we must have $|\kappa(t)|\leq M$.

\subsection{Consequence of bootstrap assumptions}\label{sec:assumptionconse}
By \eqref{eq:3xtildeUat0} and \eqref{eq:barUat0},
\begin{equation}
|\partial^3_XU(0,s)-6|\leq \epsilon^\frac{1}{4}<1\quad
\implies \quad 5<\partial^3_XU(0,s)<7.\label{eq:3xUat0}
\end{equation}
By \eqref{eq:barUall} and \eqref{eq:tildeUmiddle}, for $(X,s)$ such that $l\leq |X|\leq \frac{1}{2}e^{\frac{3}{2}s}$, we have
\begin{align}
|U(X,s)|\leq |\overline{U}(X)|+|\widetilde{U}(X,s)|\leq (1+\epsilon^\frac{1}{11})(1+X^2)^\frac{1}{6}.\label{eq:Umiddle}
\end{align}
By \eqref{eq:barUall}, \eqref{eq:xtildeUnear0} with $j=2,3$, \eqref{eq:2xUmiddle}, \eqref{eq:3xUmiddle},
\eqref{eq:2xUfar} and \eqref{eq:3xUfar}
\begin{align}
\|\partial_X^2U(\cdot,s)\|_{L^\infty}&\leq M^\frac{1}{4},\nonumber \\
\|\partial_X^3U(\cdot,s)\|_{L^\infty}&\leq M^\frac{3}{4}\label{eq:3xULinfty}.
\end{align}
By Lemma \ref{lem:Sinterpolation}, interpolating between $\partial_XU$ and $\partial_X^5U$, we also have
\begin{align}
\|\partial_X^jU(\cdot,s)\|_{L^2}&\leq \|\partial_XU(\cdot,s)\|_{L^2}^{\frac{5}{4}-\frac{j}{4}}\|\partial_X^5U(\cdot,s)\|_{L^2}^{\frac{j}{4}-\frac{1}{4}}\leq 3^{\frac{5}{4}-\frac{j}{4}}M^{j-1},\quad j=2,3,4.\label{eq:xUL2}
\end{align} 
We also claim that
\begin{align}
\|\partial_XU(\cdot,s)\|_{L^\infty}=1=-\partial_XU(0,s),\label{cor:1xULinfty}
\end{align}
with extremum attained uniquely at $X=0$.
This is indeed true for $(X,s)$ such that $|X|\geq \frac{1}{2}e^{\frac{3}{2}s}$ by \eqref{eq:1xUfar}. For $(X,s)$ such that $l\leq |X|\leq \frac{1}{2}e^{\frac{3}{2}s}$, by \eqref{eq:1xtildeUmiddle} and \eqref{eq:1xbarUmiddle},
\begin{equation}
\begin{aligned}
|\partial_XU(X,s)|&\leq |\partial_X\widetilde{U}(X,s)|+|\partial_X\overline{U}(X)|\\
&\leq (\epsilon^\frac{1}{12}+1-2l^2)(1+X^2)^{-\frac{1}{3}}\\
&\leq (1+l^2)^{-\frac{1}{3}}(1+X^2)^{-\frac{1}{3}}\leq(1+X^2)^{-\frac{1}{3}}.
\end{aligned}\label{eq:1xUmiddle}
\end{equation}
For $|X|\leq l$, we use the constraint \eqref{eq:constraints}
and consider the Taylor expansion of $\partial_XU$ at $X=0$:
\begin{align*}
\partial_XU(X,s)&=-1+\frac{1}{2}\partial^3_XU(0,s)X^2+\frac{1}{6}\partial^4_X(X',s)X^3\\
&=-1+\frac{1}{2}X^2\big[\partial^3_XU(0,s)+\frac{1}{3}\partial^4_XU(X',s)X\big]
\end{align*}
for some $X'$ between 0 and $X$, in particular, $|X'|< l$. By \eqref{eq:3xUat0}, \eqref{eq:xtildeUnear0} with $j=1$ and $l,\epsilon$ sufficiently small, we have for $0< |X|\leq l$,
\begin{align*}
\partial_XU(X,s)&\geq -1+\frac{1}{2}X^2(5-20l)\geq -1+\frac{1}{2}X^2> -1,\\
\partial_XU(X,s)&\leq -1+\frac{1}{2}X^2(7+20l)\leq -1+6l^2\leq-\frac{1}{2}. 
\end{align*}
We thus prove the claim.

\subsection{Statement of the theorems}
\begin{theorem}[Self-similar profile]\label{thm:self}
There exists a sufficiently large $M>200$ and a sufficiently small $0<\epsilon=\epsilon(M)<10^{-4}$, such that the equation \eqref{eq:ansatz} with initial data $U(\cdot,-\log\epsilon)$ given in Section \ref{sub:initialself} has a unique solution $U\in\mathcal{C}\big([-\log\epsilon,+\infty);\mathcal{C}^4\cap H^5(\mathbb{R})\big)$ that satisfies the conditions given in Section \ref{sec:assumptionself}-\ref{sec:assumptionconse}. Moreover, the limit $\nu:=\lim_{s\to\infty}\partial^3_XU(0,s)$ exists, and for each $X\in\mathbb{R}$, we have
\begin{align*}
\lim_{s\to\infty}U(X,s)=\big(\frac{\nu}{6}\big)^{-\frac{1}{2}}\overline{U}\Big(\big(\frac{\nu}{6}\big)^\frac{1}{2}X\Big):=\overline{U}_\nu(X).
\end{align*}
\end{theorem}

\begin{theorem}[Main result]
There exists 
a sufficiently large $M>200$ and a sufficiently small $0<\epsilon=\epsilon(M)<10^{-4}$ such that for smooth initial data $u_0$ at $t_0=-\epsilon$ given in Section \ref{sub:initial physical} and Section \ref{sub:initialself} after the self-similar transformation \eqref{eq:selftransform}, \eqref{eq:uUrelation}, there exists a unique solution $u \in \mathcal{C}\big([-\epsilon,T_*);\mathcal{C}^4\cap H^5(\mathbb{R})\big)$ to the BH equation such that the following hold
\begin{enumerate}[leftmargin=*]
\item The blowup time $T_*$ satisfies $|T_*|\leq 2\epsilon^\frac{7}{4}$.
\item $\|u(\cdot,t)\|_{L^\infty}\leq M$ for all $t\in [-\epsilon,T_*]$.
\item The blowup location $x_*:=\xi(T_*)$ satisfies $|x_*|\leq 3M\epsilon$ (which is small from the proof).
\item $\lim_{t\to T_*}\partial_xu\big(\xi(t),t\big)=-\infty$, and as $t\to T_*$,
\begin{align*}
\frac{1}{2(T_*-t)}\leq |\partial_xu\big(\xi(t),t\big)|=\|\partial_xu(\cdot,t)\|_{L^\infty}\leq \frac{2}{T_*-t}.
\end{align*}
For any $x\neq x_*$, $\partial_xu(x,t)$ remains finite as $t\to T_*$. Moreover, for $t$ sufficiently close to $T_*$ (depending on $x$), if $|x-x_*|>\frac{1}{2}$, then $|\partial_xu(x,t)|\leq 2$; if $|x-x_*|< \frac{1}{2}$, then $|\partial_xu(x,t)|\sim |x-x_*|^{-\frac{2}{3}}$. In particular, $u(\cdot,T_*)\in\mathcal{C}^\frac{1}{3}(\mathbb{R})$, and $u(\cdot,T_*)$ has a cusp-singularity at $x_*$.
\end{enumerate}\label{thm:main}
\end{theorem}

\begin{remark}
From the proof to the local well-posedness Theorem \ref{ap:wellposed}, the minimum blowup rate of the $H^k$-norm for $k>3/2$ is $(T_*-t)^{-1}$, which is exactly the rate given in Theorem \ref{thm:main}.
\end{remark}

\begin{corollary}[Open set of initial data]\label{cor:open}
There exists an open neighborhood in the $H^5$-topology of the set of initial data satisfying the hypothesis in Theorem \ref{thm:main} such that the conclusions of Theorem \ref{thm:main} hold.
\end{corollary}

\begin{remark}
The open set of initial data is possible thanks to the stability of $\overline{U}$. Also see Remark \ref{rem:future}.
\end{remark}

\section{$L^2$-estimates}\label{sec:L2estimate}
In this section, we close the bootstrap assumptions \eqref{eq:1xUL2} and \eqref{eq:5xUL2}. Here we note that by choosing $\epsilon$ sufficiently small, we can make $1/(1-\dot{\tau})$ as close to 1 as we need. A more precise statement and its short proof can be found in Lemma \ref{lem:1-dottau}.

\begin{proposition}[Evolution of $\|U(\cdot,s)+e^\frac{s}{2}\kappa\|_{L^2}$]
We have
\begin{align*}
\|U(\cdot,s)+e^\frac{s}{2}\kappa\|_{L^2}\leq \frac{\sqrt{2}}{2}M\epsilon^\frac{1}{4}e^{\frac{5}{4}s}.\label{prop:UL2}
\end{align*}
\end{proposition}

\begin{proof}
Conservation of $\|u(\cdot,t)\|_{L^2}$ for sufficiently smooth solution $u$ to the BH equation is proved in Lemma \ref{lem:conservation}. By \eqref{eq:uUL2relation} and the $L^2$-conservation, we have
\begin{align*}
\|U(\cdot,s)+e^\frac{s}{2}\kappa\|_{L^2}=e^{\frac{5}{4}s}\|u(t)\|_{L^2}&=e^{\frac{5}{4}s}\|u_0\|_{L^2}.
\end{align*}
Then by \eqref{eq:u0L2}, we get
\begin{align*}
\|U(\cdot,s)+e^\frac{s}{2}\kappa\|_{L^2}\leq \frac{\sqrt{2}}{2}Me^{\frac{5}{4}s}.
\end{align*}
\end{proof}

\begin{proposition}[Uniform $L^2$-bound on $\partial_XU$]
We have $\|\partial_XU(\cdot,s)\|_{L^2}\leq 3$.\label{prop:1L2}
\end{proposition}

\begin{proof}
Multiplying \eqref{eq:1partialU} by $\partial_XU$ and integrating, we get
\begin{align*}
\frac{1}{2}\frac{d}{ds}\|\partial_XU\|_{L^2}^2&+\|\partial_XU\|_{L^2}^2+\frac{1}{1-\dot{\tau}}\int_\mathbb{R}(\partial_XU)^3\,dX\\
&+\frac{1}{1-\dot{\tau}}\int_\mathbb{R}\big(U+e^\frac{s}{2}(\kappa-\dot{\xi})\big)\partial_X^2U\partial_XU\,dX+\frac{3}{2}\int_\mathbb{R}X\partial^2_XU\partial_XU\,dX=0,
\end{align*} 
where the right hand side vanishes due to the orthogonality of Hilbert transform. Integration by parts yields\footnote{Even though we will not prove a uniform spatial decay of $\partial_XU$, since $\partial_XU$ is a priori $L^2$-integrable, $(\partial_XU)^2$ must be able to beat $X$ at infinity, i.e.\ 
\begin{align*}
\lim_{X\to\infty}X\partial_XU(X,s)^2=0\quad\text{for every fixed }s.
\end{align*}
So the second integration by parts is justified. The first one can be justified using \eqref{eq:ULinfty}, \eqref{eq:xibound}
\begin{align*}
\lim_{X\to\infty}|U+e^\frac{s}{2}\kappa-e^\frac{s}{2}\dot{\xi}|\partial_XU(X,s)^2\leq 3Me^\frac{s}{2}\lim_{X\to\infty}\partial_XU(X,s)^2=0\quad\text{for every fixed }s.
\end{align*}}
\begin{gather*}
\int_\mathbb{R}\big(U+e^\frac{s}{2}(\kappa-\dot{\xi})\big)\partial_X^2U\partial_XU\,dX=-\frac{1}{2}\int_\mathbb{R}(\partial_XU)^3\,dX,\\
\int_\mathbb{R}X\partial^2_XU\partial_XU\,dX=-\frac{1}{2}\|\partial_XU\|_{L^2}^2.
\end{gather*}
Hence, upon rewriting the equation for $\|\partial_XU\|_{L^2}$, then using \eqref{eq:1xUmiddle}, \eqref{eq:1xUfar} and Lemma \ref{lem:1-dottau}, we have
\begin{align*}
\frac{d}{ds}\|\partial_XU\|_{L^2}^2+\frac{1}{2}\|\partial_XU\|_{L^2}^2&=-\frac{1}{1-\dot{\tau}}\int_\mathbb{R}(\partial_XU)^3\,dX\\
&\leq (1+2\epsilon^\frac{3}{4})\Big[\int_{|X|\leq \frac{1}{2}e^{\frac{3}{2}s}}(1+X^2)^{-1}\,dX+\int_{|X|\geq \frac{1}{2}e^{\frac{3}{2}s}}2e^{-s}\partial_XU(X,s)^2\,dX\Big]\\
&\leq (1+2\epsilon^\frac{3}{4})\big(2\arctan{(\frac{1}{2}e^{\frac{3}{2}s})}+2\epsilon\|\partial_XU\|_{L^2}^2\big)\\
&\leq \frac{5}{4}\big(\pi+\frac{1}{1000}\|\partial_XU\|_{L^2}\big),\\
\implies\qquad\frac{d}{ds}\|\partial_XU\|_{L^2}^2&+\frac{9}{20}\|\partial_XU\|_{L^2}^2\leq 4.
\end{align*}
Using Gr\"onwall's inequality we get
\begin{align*}
\|\partial_XU(\cdot,s)\|_{L^2}^2&\leq \|\partial_XU(\cdot,-\log\epsilon)\|_{L^2}^2e^{-\frac{9}{20}(s+\log\epsilon)}+4\int_{-\log\epsilon}^se^{-\frac{9}{20}(s-s')}\,ds'\\
&=\frac{80}{9}+\Big(\|\partial_XU(\cdot,-\log\epsilon)\|_{L^2}^2-\frac{80}{9}\Big)e^{-\frac{9}{20}(s+\log\epsilon)}.
\end{align*}
Since $\|\partial_XU(\cdot,-\log\epsilon)\|_{L^2}^2\leq 8<\frac{80}{9}$ from \eqref{eq:init1xUL2}, the right hand side is increasing and tends to $\frac{80}{9}$. Hence, 
\begin{align*}
\|\partial_XU(\cdot,s)\|_{L^2}^2\leq \frac{80}{9} <9,
\end{align*}
and we close the bootstrap \eqref{eq:1xUL2}.
\end{proof}

\begin{proposition}[Uniform $L^2$-bound on $\partial_X^5U$]
We have $\|\partial_X^5U(\cdot,s)\|_{L^2}\leq M^4$.
\end{proposition}

\begin{proof}
Taking $L^2$ inner product of \eqref{eq:npartialU}, $n=5$, with $\partial^5_XU$, integrating by parts, then using \eqref{eq:3xULinfty}, \eqref{eq:xUL2} with $j=4$ and Lemma \ref{lem:1-dottau}, we get
\begin{align*}
\frac{1}{2}\frac{d}{ds}\|\partial^5_XU\|_{L^2}^2+7\|\partial^5_XU\|_{L^2}^2+\frac{6}{1-\dot{\tau}}\int_\mathbb{R}\partial_XU(\partial_X^5U)^2\,dX&-\frac{1}{2(1-\dot{\tau})}\int_\mathbb{R}\partial_XU(\partial_X^5U)^2\,dX\\-\frac{3}{4}\|\partial^5_XU\|_{L^2}^2&=\frac{1}{1-\dot{\tau}}\big(20+\frac{15}{2}\big)\int_\mathbb{R}\partial_X^3U(\partial^4_XU)^2\,dX\\
\frac{d}{ds}\|\partial^5_XU\|_{L^2}^2+\frac{25}{2}\|\partial_X^5U\|_{L^2}^2+\frac{11}{1-\dot{\tau}}\int_\mathbb{R}\partial_XU(\partial^5_XU)^2\,dX&=\frac{55}{1-\dot{\tau}}\int_\mathbb{R}\partial_X^3U(\partial_X^4U)^2\,dX\\
\frac{d}{ds}\|\partial^5_XU\|_{L^2}^2+\frac{5}{4}\|\partial_X^5U\|_{L^2}^2\leq \frac{55}{1-\dot{\tau}}\|\partial_X^3U\|_{L^\infty}\|\partial_X^4U\|_{L^2}^2&\leq 60\sqrt{3}M^{6+\frac{3}{4}}\\
\implies\qquad \|\partial_X^5U(\cdot,s)\|_{L^2}^2\leq 48\sqrt{3}M^\frac{27}{4}+\Big(\|\partial_X^5U(\cdot,-\log\epsilon)\|_{L^2}^2&-48\sqrt{3}M^{\frac{27}{4}}\Big)e^{-\frac{5}{4}(s+\log\epsilon)}.
\end{align*}
If $\|\partial^5_XU(\cdot,-\log\epsilon)\|^2_{L^2}\leq 48\sqrt{3}M^\frac{27}{4}$, then $\|\partial^5_XU(\cdot,s)\|_{L^2}^2\leq 48\sqrt{3}M^\frac{27}{4}$; otherwise, we arrange the right hand side 
\begin{align*}
\|\partial_X^5U(\cdot,s)\|^2_{L^2}&\leq \|\partial_X^5U(\cdot,-\log\epsilon)\|_{L^2}^2e^{-\frac{5}{4}(s+\log\epsilon)}+48\sqrt{3}M^\frac{27}{4}(1-e^{-\frac{5}{4}(s+\log\epsilon)})\\
&\leq  \|\partial_X^5U(\cdot,-\log\epsilon)\|_{L^2}^2e^{-\frac{5}{4}(s+\log\epsilon)}+\|\partial_X^5U(\cdot,-\log\epsilon)\|_{L^2}^2(1-e^{-\frac{5}{4}(s+\log\epsilon)})\\
&\leq \|\partial_X^5U(\cdot,-\log\epsilon)\|^2_{L^2}\leq \frac{1}{4}M^8.
\end{align*}
In the former case, we need
\begin{align*}
48\sqrt{3}M^\frac{27}{4}\leq \frac{1}{4}M^8\quad\implies\quad M^\frac{5}{4}\geq 192\sqrt{3}.
\end{align*}
In either case, we can prove that 
\begin{align*}
\|\partial^5_XU(\cdot,s)\|_{L^2}\leq \frac{1}{2}M^4.
\end{align*}
Since we improve the bound on $\|\partial_X^5U\|_{L^2}$, we close the bootstrap \eqref{eq:5xUL2}.
\end{proof}

\section{Bounds on forcing terms}\label{sec:forcingbound}
We record bounds for all the forcing terms we will later use. Many of these contain Hilbert transforms of $U$ or its derivatives, so we bound them first.
\begin{lemma}[$\|H[U+e^\frac{s}{2}\kappa\rbrack(\cdot,s)\|_{L^\infty}$]
We have
\begin{align*}
\|H[U+e^\frac{s}{2}\kappa](\cdot,s)\|_{L^\infty}&\lesssim Me^{\frac{3}{4}s}.
\end{align*}\label{lem:H[U]}
\end{lemma}
\begin{proof} 
We split the principal value integral into three parts
\begin{align*}
H[U+e^\frac{s}{2}\kappa](X,s)&=\frac{1}{\pi}\mathrm{p.v.}\int_\mathbb{R}\frac{U(Y,s)+e^\frac{s}{2}\kappa}{X-Y}\,dY\\
&\sim\lim_{\delta\to 0^+}\int_{\delta\leq |X-Y|\leq 1}\frac{U(Y,s)-U(X,s)}{X-Y}\,dY+\int_{1\leq |X-Y|\leq \frac{1}{2}e^{\frac{3}{2}s}}\frac{U(Y,s)+e^\frac{s}{2}\kappa}{X-Y}\,dY\\
&\qquad\qquad+\int_{|X-Y|\geq \frac{1}{2}e^{\frac{3}{2}s}}\frac{U(Y,s)+e^\frac{s}{2}\kappa}{X-Y}\,dY,
\end{align*}
where we used the fact that the kernel is odd. Hence, by \eqref{eq:ULinfty} and Proposition \ref{prop:UL2}, \begin{align*}
    |H[U+e^\frac{s}{2}\kappa](X,s)|&\lesssim \int_{|X-Y|\leq 1}\|\partial_XU(\cdot,s)\|_{L^\infty}\,dY+\int_{1<|X-Y|<\frac{1}{2}e^{\frac{3}{2}s}}\frac{\|U(\cdot,s)+e^\frac{s}{2}\kappa\|_{L^\infty}}{|X-Y|}\,dY\\
    &\qquad\qquad+\|U(\cdot,s)+e^\frac{s}{2}\kappa\|_{L^2}\Big(\int_{|X-Y|\geq \frac{1}{2}e^{\frac{3}{2}s}}\frac{1}{|X-Y|^2}\,dY\Big)^\frac{1}{2}\\
    &\lesssim 1+Me^\frac{s}{2}(\frac{3}{2}s-\log 2)+e^{\frac{5}{4}s}Me^{-\frac{3}{4}s}\\
    &\lesssim 1+ Me^{\frac{3}{4}s}+Me^{\frac{1}{2}s}\\
    &\lesssim Me^{\frac{3}{4}s},
\end{align*}
where in the third $\lesssim$ we use $se^\frac{s}{2}\leq e^{\frac{3}{4}s}$ by choosing $\epsilon$ small enough. Since this holds for all $X$, we complete the proof.
\end{proof}

\begin{lemma}[$\|H[\partial^j_XU\rbrack(\cdot,s)\|_{L^\infty}$ for $j=1,...,4$]
We have 
\begin{align*}
\|H[\partial_X^jU](\cdot,s)\|_{L^\infty}\lesssim M^{j-\frac{1}{2}},\quad j=1,2,3,4.
\end{align*}\label{lem:H[partialjU]}
\end{lemma}

\begin{proof}
We use Lemma \ref{lem:GNSinterpolation} to interpolate between $\partial_XU$ and $\partial_X^5U$, and also the fact that Hilbert transform preserves the $L^2$-norm:
\begin{align*}
\|H[\partial_X^jU](\cdot,s)\|_{L^\infty}&\lesssim \|H[\partial_X^5U](\cdot,s)\|_{L^2}^\frac{2j-1}{8}\|H[\partial_XU](\cdot,s)\|_{L^2}^\frac{9-2j}{8}\\
&\leq\|\partial_X^5U(\cdot,s)\|_{L^2}^\frac{2j-1}{8}\|\partial_XU(\cdot,s)\|_{L^2}^\frac{9-2j}{8}\lesssim  M^\frac{2j-1}{2}.
\end{align*}
\end{proof}

\begin{lemma}[$H[\partial_XU\rbrack$]
For $(X,s)$ such that $l\leq |X|\leq \frac{1}{2}e^{\frac{3}{2}s}$, 
\begin{align}
|H[\partial_XU(X,s)|\lesssim M^\frac{1}{4}(1+X^2)^{-\frac{1}{3}}+(1+X^2)^{-\frac{1}{3}}\log(1+X^2)+(1+X^2)^{-\frac{1}{4}}.\label{eq:H1xUmiddle}
\end{align}
And for $(X,s)$ such that $|X|\geq \frac{1}{2}e^{\frac{3}{2}s}$,
\begin{align}
|H[\partial_XU(X,s)|\lesssim e^{-\frac{1}{4}s}.\label{eq:H1xUfar}
\end{align}\label{lem:H1x}
\end{lemma}

\begin{proof}
For $(X,s)$ such that $l\leq |X|\leq \frac{1}{2}e^{\frac{3}{2}s}$, we split the principal integral as follows
\begin{align*}
H[\partial_XU](X,s)&=\frac{1}{\pi}\Big(\int_{|X-Y|<(1+X^2)^{-\frac{1}{3}}}\frac{\partial_XU(Y,s)-\partial_XU(X,s)}{X-Y}\,dY\\
&\qquad+\int_{(1+X^2)^{-\frac{1}{3}}\leq |X-Y|\leq 1}\frac{\partial_XU(Y,s)}{X-Y}\,dY+\int_{|X-Y|>1} \frac{\partial_XU(Y,s)}{X-Y}\,dY\Big)\\
&=I_{\text{near}}+I_{\text{middle}}+I_{\text{far}}.
\end{align*}
By Mean Value Theorem and \eqref{eq:2xUmiddle}, we have
\begin{align*}
|I_{\text{near}}|\lesssim (1+X^2)^{-\frac{1}{3}}\|\partial_X^2U(\cdot,s)\|_{L^\infty}\lesssim M^\frac{1}{4}(1+X^2)^{-\frac{1}{3}}.
\end{align*}
Next, for $I_{\text{middle}}$, if $|X|\leq \frac{1}{2}e^{\frac{3}{2}s}-1$, then for $|X-Y|<1$, $|Y|\leq \frac{1}{2}e^{\frac{3}{2}s}$, so by \eqref{eq:1xUmiddle}
\begin{align*}
|\partial_XU(Y,s)|\leq (1+Y^2)^{-\frac{1}{3}}\leq 2(1+X^2)^{-\frac{1}{3}}.
\end{align*}
If $\frac{1}{2}e^{\frac{3}{2}s}-1<|X|\leq \frac{1}{2}e^{\frac{3}{2}s}$, then for $|Y|\leq \frac{1}{2}e^{\frac{3}{2}s}$ such that $|X-Y|<1$, same as above we have 
$|\partial_XU(Y,s)|\leq 2(1+X^2)^{-\frac{1}{3}}$, and for $|Y|\geq \frac{1}{2}e^{\frac{3}{2}s}$ such that $|X-Y|<1$, by \eqref{eq:1xUfar},
\begin{align*}
|\partial_XU(Y,s)|\leq 2e^{-s}\leq 2^\frac{1}{3}|X|^{-\frac{2}{3}}\leq 2(1+X^2)^{-\frac{1}{3}}.
\end{align*}
Hence,
\begin{align*}
|I_{\text{middle}}|\lesssim \int_{(1+X^2)^{-\frac{1}{3}}\leq |X-Y|\leq 1}\frac{(1+X^2)^{-\frac{1}{3}}}{|X-Y|}\,dY&\lesssim (1+X^2)^{-\frac{1}{3}}\log|X-Y|\Big|_{(1+X^2)^{-\frac{1}{3}}}^1\\
&\lesssim (1+X^2)^{-\frac{1}{3}}\log(1+X^2).
\end{align*}
Then, for $I_{\text{far}}$, without loss of generality we consider 3 cases: $2\leq X\leq \frac{1}{3}e^{\frac{3}{2}s}$, $\frac{1}{3}e^{\frac{3}{2}s}\leq X\leq \frac{1}{2}e^{\frac{3}{2}s}$ and $l\leq X <2$. The cases when $X<-l$ are similar.\\
Case 1: $2\leq X\leq \frac{1}{3}e^{\frac{3}{2}s}$, i.e. $\frac{1}{2}X\geq 1$, $\frac{3}{2}X\leq \frac{1}{2}e^{\frac{3}{2}s}$. We further split the integral as follows
\begin{align*}
I_{\text{far}}&=\frac{1}{\pi}\Big(\int_{|X-Y|>\frac{1}{2}X}+\int_{1<|X-Y|\leq \frac{1}{2}X}\frac{\partial_XU(Y,s)}{X-Y}\,dY\Big)\\
|I_{\text{far}}|&\lesssim \Big(\int_{|X-Y|>\frac{1}{2}|X|}\frac{1}{|X-Y|^2}\,dY\Big)^\frac{1}{2}\|\partial_XU(\cdot,s)\|_{L^2}+\int_{\frac{1}{2}X}^{X-1}+\int_{X+1}^{\frac{3}{2}X}\frac{(1+Y^2)^{-\frac{1}{3}}}{|X-Y|}\,dY\\
&\lesssim |X|^{-\frac{1}{2}}+(1+\frac{1}{4}X^2)^{-\frac{1}{3}}\int_{1<|X-Y|<\frac{1}{2}X}\frac{1}{|X-Y|}\,dY\qquad\text{by }\eqref{eq:1xUL2}\\
&\lesssim (1+X^2)^{-\frac{1}{4}}+(1+X^2)^{-\frac{1}{3}}\log{X}\\
&\lesssim (1+X^2)^{-\frac{1}{4}}+(1+X^2)^{-\frac{1}{3}}\log(1+X^2).
\end{align*}
Case 2: $\frac{1}{3}e^{\frac{3}{2}s}\leq X\leq \frac{1}{2}e^{\frac{3}{2}s}$, i.e.  $\frac{3}{2}X\geq \frac{1}{2}e^{\frac{3}{2}s}$. We split the integral in a different way
\begin{align*}
I_{\text{far}}&=\frac{1}{\pi}\Big(\int_{|X-Y|\geq\frac{1}{2}X}+\int_{\frac{1}{2}X}^{X-1}+\int_{\min(X+1,\frac{1}{2}e^{\frac{3}{2}s})}^{\frac{1}{2}e^{\frac{3}{2}s}}+\int_{\max(\frac{1}{2}e^{\frac{3}{2}s},X+1)}^{\frac{3}{2}X}\frac{\partial_XU(Y,s)}{X-Y}\,dY\Big)\\
&=I_{\text{far}}^1+I_{\text{far}}^2+I_{\text{far}}^3+I_{\text{far}}^4.
\end{align*}
We bound the four terms in different ways
\begin{align*}
|I_{\text{far}}^1|&\lesssim\Big(\int_{|X-Y|\geq \frac{1}{2}X}\frac{1}{|X-Y|^2}\Big)^\frac{1}{2}\|\partial_XU(\cdot,s)\|_{L^2}\lesssim |X|^{-\frac{1}{2}}\lesssim (1+X^2)^{-\frac{1}{4}},\\
|I_{\text{far}}^2|&\lesssim \int_{\frac{1}{2}X}^{X-1}\frac{(1+Y^2)^{-\frac{1}{3}}}{|X-Y|}\,dY\leq (1+\frac{1}{4}X^2)^{-\frac{1}{3}}\int_{\frac{1}{2}X}^{X-1}\frac{1}{|X-Y|}\,dY\\
&\lesssim (1+X^2)^{-\frac{1}{3}}\log{X}\lesssim (1+X^2)^{-\frac{1}{3}}\log(1+X^2),\\
|I_{\text{far}}^3+I_{\text{far}}^4|&\lesssim \int_{X+1}^{\frac{3}{2}X} \frac{(1+Y^2)^{-\frac{1}{3}}}{|X-Y|}\,dY+\int_{\max(\frac{1}{2}e^{\frac{3}{2}s},X+1)}^{\frac{3}{2}X}\frac{2e^{-s}}{|X-Y|}\,dY\qquad \text{by \eqref{eq:1xUmiddle}, \eqref{eq:1xUfar}}\\
&\lesssim (1+X^2)^{-\frac{1}{3}}\int_{X+1}^{\frac{3}{2}X}\frac{1}{|X-Y|}\,dY+e^{-s}\int_{X+1}^{\frac{3}{2}X}\frac{1}{|X-Y|}\,dY\\
&\lesssim \big((1+X^2)^{-\frac{1}{3}}+e^{-s}\big)\log{X}\\
&\lesssim (1+X^2)^{-\frac{1}{3}}\log(1+X^2).
\end{align*}
Case 3: $l<X<2$, i.e. $(1+X^2)^{-\frac{1}{4}}\geq 5^{-\frac{1}{4}}$. 
\begin{align*}
|I_{\text{far}}|&\lesssim \|\partial_XU(\cdot,s)\|_{L^2}\Big(\int_{|X-Y|>1}\frac{1}{|X-Y|^2}\,dY\Big)^\frac{1}{2}\lesssim 3\sqrt{2}\lesssim (1+X^2)^{-\frac{1}{4}}.
\end{align*}
In all three cases, we have
\begin{align*}
|I_{\text{far}}|\lesssim (1+X^2)^{-\frac{1}{4}}+(1+X^2)^{-\frac{1}{3}}\log(1+X^2).
\end{align*}
Therefore, putting all terms together, we get
\begin{align*}
|H[\partial_XU](X,s)|&\lesssim M^\frac{1}{4}(1+X^2)^{-\frac{1}{3}}+(1+X^2)^{-\frac{1}{4}}+(1+X^2)^{-\frac{1}{3}}\log(1+X^2).
\end{align*}
This gives \eqref{eq:H1xUmiddle}.

For $(X,s)$ such that $|X|\geq \frac{1}{2}e^{\frac{3}{2}s}$, we split the principal integral as follows
\begin{align*}
H[\partial_XU](X,s)&=\frac{1}{\pi}\Big(\int_{|X-Y|<e^{-s}}\frac{\partial_XU(Y,s)-\partial_XU(X,s)}{X-Y}\,dY\\
&\qquad+\int_{e^{-s}\leq |X-Y|\leq 1}\frac{\partial_XU(Y,s)}{X-Y}\,dY+\int_{|X-Y|>1}\frac{\partial_XU(Y,s)}{X-Y}\,dY\Big)\\
&=\mathit{II}_{\text{near}}+\mathit{II}_{\text{middle}}+\mathit{II}_{\text{far}}.
\end{align*}
By Mean Value Theorem and \eqref{eq:2xUfar}, we have
\begin{align*}
|\mathit{II}_{\text{near}}|\leq \frac{2}{\pi}e^{-s}\|\partial_X^2U(\cdot,s)\|_{L^\infty}\lesssim M^\frac{1}{4}e^{-s}.
\end{align*}
Next, for $\mathit{II}_{\text{middle}}$, if $|X|\geq \frac{1}{2}e^{\frac{3}{2}s}+1$ and $|X-Y|\leq 1$, then $|Y|\geq \frac{1}{2}e^{\frac{3}{2}s}$, so $|\partial_XU(Y,s)|\leq 2e^{-s}$; this bound also holds for those $|Y|\geq \frac{1}{2}e^{\frac{3}{2}s}$ such that $|X-Y|<1$ when $\frac{1}{2}e^{\frac{3}{2}s}\leq |X|<\frac{1}{2}e^{\frac{3}{2}s}+1$. And for those $|X-Y|<1$ such that $|Y|\leq \frac{1}{2}e^{\frac{3}{2}s}$ when $\frac{1}{2}e^{\frac{3}{2}s}\leq |X|<\frac{1}{2}e^{\frac{3}{2}s}+1$,
by \eqref{eq:1xUmiddle} and \eqref{eq:1xbarUfar}, 
\begin{align*}
|\partial_XU(Y,s)|&\leq (\frac{7}{20}+\epsilon^\frac{1}{12})(1+Y^2)^{-\frac{1}{3}}\\
&\leq \frac{2}{5}\big(1+(X-1)^2\big)^{-\frac{1}{3}}\leq\big(\frac{1}{2}e^{\frac{3}{2}s}-1\big)^{-\frac{2}{3}}\leq 2e^{-s},
\end{align*}
(here without loss of generality we consider the positive $Y$ part; the negative $Y$ part has the same result) since $\frac{2}{5}\cdot 2^\frac{2}{3}\approx 0.635$ and $s\geq -\log\epsilon$ is very large.
Hence,
\begin{align*}
|\mathit{II}_{\text{middle}}|\leq \frac{1}{\pi}\int_{e^{-s}\leq |X-Y|\leq 1}\frac{2e^{-s}}{|X-Y|}\,dY\leq \frac{4}{\pi}e^{-s}\log|X-Y|\Big|_{e^{-s}}^1
\lesssim se^{-s}.
\end{align*}
Then, for $\mathit{II}_{\text{far}}$, without loss of generality consider two cases: $\frac{1}{2}e^{\frac{3}{2}s}\leq X<\frac{3}{4}e^{\frac{3}{2}s}$ and $X\geq\frac{3}{4}e^{\frac{3}{2}s}$. The two cases when $X\leq -\frac{1}{2}e^{\frac{3}{2}s}$ are similar. \\
Case 1: $\frac{1}{2}e^{\frac{3}{2}s}\leq X<\frac{3}{4}e^{\frac{3}{2}s}$. We further split the integral as follows
\begin{align*}
\mathit{II}_{\text{far}}&=\frac{1}{\pi}\Big(\int_{|X-Y|\geq \frac{1}{4}e^{\frac{3}{2}s}}+\int_{X-\frac{1}{4}e^{\frac{3}{2}s}}^{\min(\frac{1}{2}e^{\frac{3}{2}s},X-1)}+\int_{\min(\frac{1}{2}e^{\frac{3}{2}s},X-1)}^{X-1}+\int_{X+1}^{X+\frac{1}{4}e^{\frac{3}{2}s}}\frac{\partial_XU(Y,s)}{X-Y}\,dY\Big)\\
&=\mathit{II}_{\text{far}}^1+\mathit{II}_{\text{far}}^2+\mathit{II}_{\text{far}}^3+\mathit{II}_{\text{far}}^4.
\end{align*}
We bound the four terms in different ways
\begin{align*}
|\mathit{II}_{\text{far}}^1|&\lesssim \Big(\int_{|X-Y|\geq \frac{1}{4}e^{\frac{3}{2}s}}\frac{1}{|X-Y|^2}\,dY\Big)^\frac{1}{2}\|\partial_XU(\cdot,s)\|_{L^2}\lesssim e^{-\frac{3}{4}s}\qquad\text{by \eqref{eq:1xUL2}},\\
|\mathit{II}_{\text{far}}^2|&\lesssim \int_{X-\frac{1}{4}e^{\frac{3}{2}s}}^{\min(\frac{1}{2}e^{\frac{3}{2}s},X-1)}\frac{|Y|^{-\frac{2}{3}}}{|X-Y|}\,dY\qquad\text{by \eqref{eq:1xUmiddle}}\\
&\lesssim \Big(\int_{\frac{1}{4}e^{\frac{3}{2}s}}^{\frac{1}{2}e^{\frac{3}{2}s}}|Y|^{-\frac{4}{3}}\,dY\Big)^\frac{1}{2}\Big(\int_{X-\frac{1}{4}e^{\frac{3}{2}s}}^{X-1}\frac{1}{|X-Y|^2}\,dY\Big)^\frac{1}{2}\qquad\text{since }X-\frac{1}{4}e^{\frac{3}{2}s}\geq\frac{1}{4}e^{\frac{3}{2}s}\\
&\lesssim \Big((\frac{1}{4}e^{-\frac{3}{2}s})^{-\frac{1}{3}}-(\frac{1}{2}e^{-\frac{3}{2}s})^{-\frac{1}{3}}\Big)^\frac{1}{2}\lesssim e^{-\frac{1}{4}s},\\
|\mathit{II}_{\text{far}}^3+\mathit{II}_{\text{far}}^4|&\lesssim \int_{1<|X-Y|<\frac{1}{4}e^{\frac{3}{2}s}}\frac{2e^{-s}}{|X-Y|}\,dY\lesssim se^{-s}\qquad\text{by \eqref{eq:1xUfar}}\text{, and }X-\frac{1}{4}e^{\frac{3}{2}s}\leq \frac{1}{2}e^{\frac{3}{2}s}.
\end{align*}
Hence, $|\mathit{II}_\text{far}|\lesssim e^{-\frac{3}{4}s}+e^{-\frac{1}{4}s}+se^{-s}\lesssim e^{-\frac{1}{4}s}+se^s$.\\
Case 2: $X\geq \frac{3}{4}e^{\frac{3}{2}s}$. We split the integral in a simpler way
\begin{align*}
\mathit{II}_{\text{far}}&=\frac{1}{\pi}\Big(\int_{|X-Y|\geq\frac{1}{4}e^{\frac{3}{2}s}}+\int_{1<|X-Y|<\frac{1}{4}e^{\frac{3}{2}s}}\frac{\partial_XU(Y,s)}{X-Y}\,dY\Big)\\
|\mathit{II}_{\text{far}}|&\lesssim \Big(\int_{|X-Y|\geq \frac{1}{4}e^{\frac{3}{2}s}}\frac{1}{|X-Y|^2}\,dY\Big)^\frac{1}{2}\|\partial_XU(\cdot,s)\|_{L^2}+\int_{1<|X-Y|<\frac{1}{4}e^{\frac{3}{2}s}}\frac{2e^{-s}}{|X-Y|}\,dY\\
&\lesssim e^{-\frac{3}{4}s}+se^{-s}\lesssim e^{-\frac{1}{4}s}+se^{-s}.
\end{align*}
Therefore, putting all terms together, we get 
\begin{align*}
|H[\partial_XU](X,s)|\lesssim M^\frac{1}{4}e^{-s}+se^{-s}+e^{-\frac{1}{4}s}\lesssim e^{-\frac{1}{4}s},
\end{align*}
by choosing $\epsilon$ sufficiently small so that $s\geq-\log\epsilon$ is sufficiently large. This proves \eqref{eq:H1xUfar}.
\end{proof}

\begin{lemma}[$H[\partial_X^2U\rbrack$]
For $(X,s)$ such that $l\leq |X|\leq \frac{1}{2}e^{\frac{3}{2}s}$, 
\begin{align}
|H[\partial_X^2U(X,s)|\lesssim M^\frac{3}{4}(1+X^2)^{-\frac{1}{3}}+M^\frac{1}{4}(1+X^2)^{-\frac{1}{3}}\log(1+X^2)+           M(1+X^2)^{-\frac{1}{4}}.\label{eq:H2xUmiddle}
\end{align}
And for $(X,s)$ such that $|X|\geq \frac{1}{2}e^{\frac{3}{2}s}$,
\begin{align}
|H[\partial_X^2U(X,s)|\lesssim M^\frac{1}{4}e^{-\frac{1}{4}s}.\label{eq:H2xUfar}
\end{align}\label{lem:H2x}
\end{lemma}

\begin{proof}
The proof is very similar to the one for Lemma \ref{lem:H1x}, using exactly the same splittings. The difference is due to the powers of $M$ in the bootstrap assumptions.

In \eqref{eq:H2xUmiddle}, the first term is the ``near" part, and uses \eqref{eq:3xULinfty}. The second term comes from the ``middle" part and some ``far" part. For the ``middle" part, we use \eqref{eq:2xUmiddle}, \eqref{eq:2xUfar} together with the observation that for $Y$ within $\pm 1$ of the threshold $\frac{1}{2}e^{\frac{3}{2}s}$, $(1+Y^2)^{-\frac{1}{3}}\approx 2^\frac{2}{3}e^{-s}$. For the ``far" part, we use \eqref{eq:2xUmiddle} and $e^{-s}\lesssim (1+X^2)^{-\frac{1}{3}}$. The last term is the remaining of ``far" part, i.e.\ the ``very far" part where we can use H\"older's inequality and \eqref{eq:xUL2} with $j=2$.

The bound \eqref{eq:H2xUfar} has an intermediate step 
\begin{align*}
|H[\partial_X^2U](X,s)|\lesssim M^{\frac{3}{4}}e^{-s}+M^\frac{1}{4}se^{-s}+M^\frac{1}{4}e^{-\frac{1}{4}s},
\end{align*}
which is $\lesssim M^\frac{1}{4}e^{-\frac{1}{4}s}$ since $s$ is sufficiently large to absorb other factors.  The first term is the ``near" part, and uses \eqref{eq:3xULinfty}. The second term comes from the ``middle" part and some ``far" part. For the ``middle" part, again we use \eqref{eq:2xUmiddle}, \eqref{eq:2xUfar}, and the observation that for $Y$ within $\pm 1$ of the threshold $\frac{1}{2}e^{\frac{3}{2}s}$, $(1+Y^2)^{-\frac{1}{3}}\approx 2^\frac{2}{3}e^{-s}$. For the ``far" part, we use \eqref{eq:2xUfar}. The last term comes from the remaining of ``far" part, i.e. the ``very far" part, where we can use H\"older's inequality and \eqref{eq:xUL2} with $j=2$.
\end{proof}

Next, we prove bounds on various forcing terms. 
 
\begin{lemma}[$F_{\partial^4_X\tilde{U}}$ for $|X|\leq l$]
For $|X|\leq l$,
\begin{align*}
|F_{\partial^4_X\tilde{U}}|&\lesssim \epsilon^\frac{1}{8}+(\log M)^{-1}\epsilon^\frac{1}{10}.
\end{align*}\label{lem:4tildeU}
\end{lemma}

\begin{proof}
From \eqref{eq:ntildeUforce} with $n=4$, we write
\begin{align*}
F_{\partial^4_X\tilde{U}}=\frac{1}{1-\dot{\tau}}(e^{-s}H[\partial^4_XU]+\dot{\tau}\widetilde{F}^4_{\dot{\tau}}+\widetilde{F}^4_{\text{no }\dot{\tau}}),
\end{align*}
where 
\begin{gather*}
\widetilde{F}^4_{\dot{\tau}}=-\big(10\partial^2_X\overline{U}\partial^3_X\overline{U}+5\partial_X\overline{U}\partial^4_X\overline{U}+\overline{U}\partial^5_X\overline{U}\big),\\
\widetilde{F}^4_{\text{no }\dot{\tau}}=-\big(10\partial^3_X\widetilde{U}\partial^2_X\overline{U}+10\partial_X^2\widetilde{U}\partial^3_X\overline{U}+10\partial^2_X\widetilde{U}\partial^3_X\widetilde{U}+5\partial_X\widetilde{U}\partial_X^4\overline{U}+\widetilde{U}\partial^5_X\overline{U}+e^\frac{s}{2}(\kappa-\dot{\xi})\partial^5_X\overline{U}\big).
\end{gather*}
Using the observation \eqref{eq:barUnear0} for $|X|\leq l\leq \frac{1}{5}$, we have
\begin{align*}
|\widetilde{F}^4_{\dot{\tau}}|\leq 270.
\end{align*}
Using the assumptions \eqref{eq:xtildeUnear0}, Lemma \ref{lem:H[partialjU]} with $j=2$ and recall $l=(\log M)^{-2}$, we have
\begin{align*}
|\widetilde{F}^4_{\text{no }\dot{\tau}}|&\lesssim \epsilon^\frac{1}{8}+\log M\epsilon^\frac{1}{10}(l+l^2+l^3+l^3+l^4)+e^{-s}\frac{H[\partial^2_XU](0,s)}{\partial^3_XU(0,s)}\\
&\lesssim \epsilon^\frac{1}{8}+\log M\epsilon^\frac{1}{10}l+e^{-s}M^\frac{3}{2}\\
&\lesssim \epsilon^\frac{1}{8}+(\log M)^{-1}\epsilon^\frac{1}{10},
\end{align*}
where we use a power of $e^{-s}$ to absorb the $M$ factor in the last term. By Lemma \ref{lem:H[partialjU]} with $j=4$ and  assumption \eqref{eq:taubound}, we get
\begin{align*}
|F_{\partial^4_X\tilde{U}}|&\lesssim e^{-s} M^\frac{7}{2}+e^{-\frac{3}{4}s}+\epsilon^\frac{1}{8}+(\log M)^{-1}\epsilon^\frac{1}{10}\lesssim \epsilon^\frac{1}{8}+(\log M)^{-1}\epsilon^\frac{1}{10},
\end{align*}
where we use a power of $\epsilon$ to absorb the $M$ factor, and then combine terms.
\end{proof}

\begin{lemma}[$F_{\tilde{U}}$,   $F_{\partial_X\tilde{U}}$, $F_{\partial_X^2U},\ F_{\partial^3_XU}$ for $l\leq |X|\leq \frac{1}{2}e^{\frac{3}{2}s}$]
For any $(X,s)$ such that $l\leq |X|\leq \frac{1}{2}e^{\frac{3}{2}s}$,
\begin{align*}
|F_{\widetilde{U}}(X,s)|&\lesssim Me^{-\frac{1}{4}s},\\
|F_{\partial_X\widetilde{U}}(X,s)|&\lesssim
e^{-s}(1+X^2)^{-\frac{1}{3}}\big(M^\frac{1}{4}+\log(1+X^2)\big)+e^{-\frac{3}{4}s}(1+X^2)^{-\frac{1}{3}}+\epsilon^\frac{1}{11}(1+X^2)^{-\frac{2}{3}},
\\
|F_{\partial_X^2U}(X,s)|&\lesssim e^{-s}(1+X^2)^{-\frac{1}{3}}\big(M^\frac{3}{4}+M^\frac{1}{4}\log(1+X^2)\big)+Me^{-\frac{3}{4}s}(1+X^2)^{-\frac{1}{3}},\\
|F_{\partial^3_XU}(X,s)|&\lesssim e^{-s}M^\frac{5}{2}+M^\frac{1}{2}(1+X^2)^{-\frac{2}{3}}.
\end{align*}\label{lem:forcemiddle}
\end{lemma}

\begin{proof}
Using \eqref{eq:kappaxi}, we rewrite $F_{\widetilde{U}}(X,s)$ as
\begin{align*}
    F_{\widetilde{U}}(X,s)&=\frac{e^{-s}}{1-\dot{\tau}}\Big(H[U+e^\frac{s}{2}\kappa](X,s)-H[U+e^\frac{s}{2}\kappa](0,s)-\frac{H[\partial^2_XU](0,s)}{\partial^3_XU(0,s)}\Big)\\
    &\qquad-\partial_X\overline{U}(X)\Big(\frac{\dot{\tau}\overline{U}(X)}{1-\dot{\tau}}+\frac{e^{-s}}{1-\dot{\tau}}\frac{H[\partial_X^2U](0,s)}{\partial^3_XU(0,s)}\Big)\\
    &=\frac{e^{-s}}{1-\dot{\tau}}\Big(H[U+e^\frac{s}{2}\kappa](X,s)-H[U+e^\frac{s}{2}\kappa](0,s)\\
    &\qquad-(1+\partial_X\overline{U}(X))\frac{H[\partial^2_XU](0,s)}{\partial^3_XU(0,s)}\Big)-\overline{U}(X)\partial_X\overline{U}(X)\frac{\dot{\tau}}{1-\dot{\tau}}.
\end{align*} 
Using Lemma \ref{lem:H[U]} and \ref{lem:H[partialjU]} with $j=2$, \eqref{eq:3xUat0}, \eqref{eq:taubound} and Lemma \ref{lem:1-dottau}, we have
\begin{align*}
|F_{\widetilde{U}}(X,s)|&\lesssim e^{-s}\Big(Me^{\frac{3}{4}s}+Me^{\frac{3}{4}s}+M^\frac{3}{2}\Big)+(1+X^2)^\frac{1}{6}(1+X^2)^{-\frac{1}{3}}e^{-\frac{3}{4} s}\\
&\lesssim Me^{-\frac{1}{4}s}+M^\frac{3}{2}e^{-s}+(1+X^2)^{-\frac{1}{6}}e^{-\frac{3}{4} s}\\
&\lesssim Me^{-\frac{1}{4}s},
\end{align*} 
where we use a small power of $e^{-s}$ (which is $\leq \epsilon$) to absorb the $M$ factor.

For $F_{\partial_X\widetilde{U}}$, we rewrite it as 
\begin{align*}
F_{\partial_X\tilde{U}}(X,s)=&\frac{e^{-s}}{1-\dot{\tau}}H[\partial_XU](X,s)-\frac{\dot{\tau}}{1-\dot{\tau}}\partial_X\overline{U}(X)^2\\
&\quad-\partial^2_X\overline{U}(X)\Big(\frac{\widetilde{U}(X,s)+\dot{\tau}\overline{U}(X)}{1-\dot{\tau}}+\frac{e^{-s}}{1-\dot{\tau}}\frac{H[\partial^2_XU](0,s)}{\partial^3_XU(0,s)}\Big).
\end{align*}
By \eqref{eq:H1xUmiddle} in Lemma \ref{lem:H1x}, \eqref{eq:barUall}, \eqref{eq:tildeUmiddle}, \eqref{eq:taubound}, \eqref{lem:H[partialjU]}, we have
\begin{align*}
|F_{\partial_X\tilde{U}}|&\lesssim e^{-s}\Big(M^\frac{1}{4}(1+X^2)^{-\frac{1}{3}}+(1+X^2)^{-\frac{1}{3}}\log(1+X^2)+(1+X^2)^{-\frac{1}{4}}\Big)\\
&\qquad+e^{-\frac{3}{4}s}(1+X^2)^{-\frac{2}{3}}+(1+X^2)^{-\frac{5}{6}}\Big(\epsilon^\frac{1}{11}(1+X^2)^\frac{1}{6}+e^{-\frac{3}{4}s}(1+X^2)^\frac{1}{6}+e^{-s}M^\frac{3}{2}\Big)\\
&\lesssim e^{-s}(1+X^2)^{-\frac{1}{3}}\big(M^\frac{1}{4}+\log(1+X^2)\big)+e^{-s}(1+X^2)^\frac{1}{12}(1+X^2)^{-\frac{1}{3}}\\
&\qquad+(1+X^2)^{-\frac{2}{3}}(\epsilon^\frac{1}{11}+e^{-\frac{3}{4}s})+e^{-s}M^\frac{3}{2}(1+X^2)^{-\frac{5}{6}}\\
&\lesssim e^{-s}(1+X^2)^{-\frac{1}{3}}\big(M^\frac{1}{4}+\log(1+X^2)\big)+e^{-\frac{3}{4}s}(1+X^2)^{-\frac{1}{3}}+\epsilon^\frac{1}{11}(1+X^2)^{-\frac{2}{3}},
\end{align*}
where we used $(1+X^2)^\frac{1}{12}\lesssim e^{\frac{1}{4}s}$, $e^{-\frac{3}{4}s}\leq \epsilon^\frac{3}{4}\leq \epsilon^\frac{1}{11}$, $e^{-s}M^\frac{3}{2}\leq \epsilon^\frac{1}{11}$ by making $\epsilon$ sufficiently small.

The bound for $F_{\partial^2_XU}$ comes from multiplying \eqref{eq:H2xUmiddle} by $e^{-s}$, then using $(1+X^2)^\frac{1}{12}\lesssim e^{\frac{1}{4}s}$ to get the last term.

The bound for $F_{\partial^3_XU}$ follows directly from Lemma \ref{lem:H[partialjU]} with $j=3$ and \eqref{eq:2xUmiddle}.
\end{proof}

\begin{lemma}[$F_{\partial^3_XU}$ for $|X|\geq \frac{1}{2}e^{\frac{3}{2}s}$]
For any $(X,s)$ such that $|X|\geq \frac{1}{2}e^{\frac{3}{2}
s}$, 
\begin{align*}
|F_{\partial^3_XU}(X,s)|&\lesssim M^\frac{5}{2}e^{-s}.
\end{align*}\label{lem:forcefar}
\end{lemma}
\begin{proof}
By Lemma \ref{lem:H[U]} and \eqref{eq:2xUfar}, we get 
\begin{align*}
|F_{\partial^3_XU}(X,s)|\lesssim M^\frac{5}{2}e^{-s}+M^\frac{1}{2}e^{-2s}\lesssim M^\frac{5}{2}e^{-s}.
\end{align*}
\end{proof}

\begin{lemma}[$F_U$] Let $F_U$ be the right hand side of \eqref{eq:ansatz}, i.e.
\begin{align*}
F_U=-\frac{e^{-\frac{s}{2}}\dot{\kappa}}{1-\dot{\tau}}+\frac{e^{-s}}{1-\dot{\tau}}H[U+e^\frac{s}{2}\kappa].
\end{align*}
For any $s\geq -\log\epsilon$, we have
\begin{align*}
\|F_U(\cdot,s)\|_{L^\infty}\lesssim M^\frac{3}{2}e^{-\frac{1}{4}s}.
\end{align*}\label{lem:FU}
\end{lemma}

\begin{proof}
By the second equality in \eqref{eq:kappaxi}, we have
\begin{align*}
\frac{e^{-\frac{s}{2}}\dot{\kappa}}{1-\dot{\tau}}=\frac{e^{-s}}{1-\dot{\tau}}\Big(H[U+e^\frac{s}{2}\kappa](0,s)+\frac{H[\partial_X^2U](0,s)}{\partial_X^3U(0,s)}\Big).
\end{align*}
By Lemma \ref{lem:H[U]}, Lemma \ref{lem:H[partialjU]} with $j=2$ and \eqref{eq:3xUat0}, we have
\begin{align*}
\big|\frac{e^{-\frac{s}{2}}\dot{\kappa}}{1-\dot{\tau}}\big|\lesssim e^{-s}(Me^{\frac{3}{4}s}+M^\frac{3}{2}).
\end{align*}
Hence, 
\begin{align*}
\|F_U(\cdot,s)\|_{L^\infty}\lesssim e^{-s}(Me^{\frac{3}{4}s}+M^\frac{3}{2})+e^{-s}Me^{\frac{3}{4}s}\lesssim M^\frac{3}{2}e^{-\frac{1}{4}s}.
\end{align*}
\end{proof}

\section{Closure of bootstrap}\label{sec:closure}

\subsection{Lagrangian trajectories}
Let $\Phi:\mathbb{R}\times[s_0,\infty)\to\mathbb{R}$ be the Lagrangian trajectory of $U$, i.e. for each $X_0$, $\Phi(X_0,s)$ is the position at time $s$ such that
\begin{equation}
\begin{aligned}
\frac{d}{ds}\Phi(X_0,s)&=V\circ \Phi(X_0,s),\\
\Phi(X_0,s_0)&=X_0,
\end{aligned}\label{eq:trajectory}
\end{equation}
where $V$ is the transport speed defined in \eqref{eq:speed}.

\begin{lemma}[Lower bound on transport speed]
For $|X_0|\geq l$, we have
\begin{align*}
|\Phi(X_0,s)|\geq |X_0|e^{\frac{1}{5}(s-s_0)}.
\end{align*}
In other words, once a particle is at least away from 0 by a distance $l$, it will escape to infinity exponentially fast. 
\label{lem:lowerLagrangian}
\end{lemma}

\begin{proof}
By Mean Value Theorem,  \eqref{cor:1xULinfty}, and the constraint \eqref{eq:constraints},
we have
\begin{align*}
|U(X,s)|\leq |U(0,s)|+ \|\partial_XU(\cdot,s)\|_{L^\infty}|X|\leq |X|.
\end{align*}
By \eqref{eq:kappaxi}, \eqref{eq:3xUat0} and Lemma \ref{lem:H[U]}, we have
\begin{align*}
|e^\frac{s}{2}(\kappa-\dot{\xi})|&=e^{-s}\Big|\frac{H[\partial_X^2U](0,s)}{\partial^3_XU(0,s)}\Big|\lesssim e^{-s}M^\frac{3}{2}\leq \frac{1}{10}l.
\end{align*}
Then for $|X|\geq l$, by \eqref{eq:speed} and Lemma \ref{lem:1-dottau}, we have
\begin{align*}
|V(X,s)|&\geq \frac{3}{2}|X|-(1+2\epsilon^\frac{3}{4})|X|-(1+2\epsilon^\frac{3}{4})\frac{1}{10}|X|\\
&=\frac{3}{2}|X|-(1+2\epsilon^\frac{3}{4})(1+\frac{1}{10})|X|\\
&\geq \frac{1}{5}|X|.
\end{align*}
Hence,
\begin{align*}
\frac{1}{2}\frac{d}{ds}\Phi^2(X_0,s)&=\big(V\circ\Phi(X_0,s)\big)\Phi(X_0,s)\geq \frac{1}{5}\Phi^2(X_0,s)\\
\implies\qquad|\Phi(X_0,s)|^2&\geq|X_0|^2e^{\frac{2}{5}(s-s_0)}\\
|\Phi(X_0,s)|&\geq|X_0|e^{\frac{1}{5}(s-s_0)}.
\end{align*}
\end{proof}

\begin{lemma}[Upper bound on Lagrangian trajectory]
For all $(X_0,s_0)$,
we have
\begin{align*}
|\Phi(X_0,s)|\leq (|X_0|+\frac{7}{2}Me^{\frac{1}{2}s_0})e^{\frac{3}{2}(s-s_0)}.
\end{align*}\label{lem:upperLagrangian}
\end{lemma}

\begin{proof}
Multiplying the first equation of \eqref{eq:trajectory} by the integrating factor $e^{-\frac{3}{2}s}$ and plugging in the expression for $V$, we get
\begin{align*}
\frac{d}{ds}\big(e^{-\frac{3}{2}s}\Phi(X_0,s)\big)&=\frac{1}{1-\dot{\tau}}\Big(e^{-\frac{3}{2}s}U\circ\Phi(X_0,s)+e^{-s}\kappa-e^{-s}\dot{\xi}\Big)\\
e^{-\frac{3}{2}s}\Phi(X_0,s)&=e^{-\frac{3}{2}s_0}X_0+\frac{1}{1-\dot{\tau}}\int_{s_0}^s\Big( e^{-\frac{3}{2}s'}U\circ\Phi(X_0,s')+e^{-s'}\kappa-e^{-s'}\dot{\xi}\Big)\,ds'.
\end{align*}
Hence, by Lemma \ref{lem:1-dottau}, \eqref{eq:ULinfty}, \eqref{eq:xibound}, we have
\begin{align*}
|\Phi(X_0,s)-X_0e^{\frac{3}{2}(s-s_0)}|&\leq e^{\frac{3}{2}s}(1+2\epsilon^\frac{3}{4})\int_{s_0}^s 3Me^{-s'}\,ds'\\
&\leq 3\cdot\frac{7}{6} Me^{\frac{1}{2}s_0}(1-e^{-(s-s_0)})e^{\frac{3}{2}(s-s_0)}\\
&\leq \frac{7}{2}Me^{\frac{1}{2}s_0}
e^{\frac{3}{2}(s-s_0)}.
\end{align*}
By triangle inequality, the proof is complete.
\end{proof}

\begin{proposition}
The $L^\infty$-norm of $e^{-\frac{s}{2}}U(\cdot,s)+\kappa$ is bounded uniformly in $s$. More precisely,
\begin{align*}
\|e^{-\frac{s}{2}}U(\cdot,s)+\kappa\|_{L^\infty}\leq M,\quad\text{i.e.}\quad
\|U(\cdot,s)+e^\frac{s}{2}\kappa\|_{L^\infty}\leq Me^\frac{s}{2}.
\end{align*}
\end{proposition}

\begin{proof}
From \eqref{eq:BH} or \eqref{eq:ansatz} we get equation for $e^{-\frac{s}{2}}U(\cdot,s)+\kappa$ (see Appendix \ref{ap:derivation} for the derivation)
\begin{align}
    \partial_s(e^{-\frac{s}{2}}U+\kappa)+V\partial_X(e^{-\frac{s}{2}}U+\kappa)=\frac{e^{-\frac{3}{2}s}}{1-\dot{\tau}}H[U+e^\frac{s}{2}\kappa].\label{eq:U+kappa?}
\end{align}
By Lemma \ref{lem:H[U]}, we can bound the forcing term on the right hand side
\begin{align*}
\Big|\frac{e^{-\frac{3}{2}s}}{1-\dot{\tau}}H[U+e^\frac{s}{2}\kappa]\Big|\lesssim Me^{-\frac{3}{4}s}.
\end{align*}
Composing \eqref{eq:U+kappa?} with Lagrangian trajectory we have
\begin{align*}
    |(e^{-\frac{s}{2}}U+\kappa)\circ\Phi(X_0,s)|&\leq |e^{\frac{\log{\epsilon}}{2}}U(X_0,-\log{\epsilon})+\kappa(0)|\\
    &\qquad+\int_{-\log{\epsilon}}^s\frac{e^{-\frac{3}{2}s'}}{1-\dot{\tau}}\big|H[U+e^\frac{s'}{2}\kappa]\circ\Phi(X_0,s')\big|\,ds'\\
    &\leq \frac{M}{2}+C\int_{-\log{\epsilon}}^\infty Me^{-\frac{3}{4}s'}\,ds'\\
    &\leq \frac{M}{2}+CM\epsilon^\frac{3}{4}\leq \frac{3}{4}M 
\end{align*}
for all $X_0$. We thus close the assumption \eqref{eq:ULinfty}.
\end{proof}

\subsection{Derivatives of $\widetilde{U}$ for $|X|\leq l$}
We first prove that $\partial^4_X\widetilde{U}(X,s)$ is small for $X$ near 0 in order to use Fundamental Theorem of Calculus for lower derivatives of $\widetilde{U}$.
\begin{proposition}
For $|X|\leq l$, we have $$|\partial^4_X\widetilde{U}(X,s)|\leq \epsilon^\frac{1}{10}.$$
\end{proposition}

\begin{proof}
Let $D_{\partial_X^4\tilde{U}}$ be the damping term in \eqref{eq:eqnxtildeU} with $n=4$, then
\begin{align*}
D_{\partial_X^4\tilde{U}}=\frac{11}{2}+\frac{5\partial_XU}{1-\dot{\tau}}\geq \frac{11}{2}-5(1+2\epsilon^\frac{3}{4})\geq \frac{1}{4}.
\end{align*}
We also recall the bound on $F_{\partial^4_X\tilde{U}}$ given by Lemma \ref{lem:4tildeU}. Composing \eqref{eq:eqnxtildeU}, $n=4$  with the Lagrangian trajectory (note that we must have $|X_0|\leq l$) and using Gr\"onwall's inequality, we get
\begin{align*}
|\partial^4_X\tilde{U}\circ\Phi(X_0,s)|\leq& |\partial^4_X\widetilde{U}(X_0,-\log{\epsilon})|e^{-\frac{1}{4}(s+\log\epsilon)}\\
&+C\big(\epsilon^\frac{1}{8}+(\log M)^{-1}\epsilon^\frac{1}{10}\big)\int_{-\log\epsilon}^se^{-\frac{1}{4}(s-s')}\,ds'\\
\leq& \epsilon^\frac{1}{8}+C\epsilon^\frac{1}{8}+4C(\log M)^{-1}\epsilon^\frac{1}{10}\leq \frac{3}{4}\epsilon^\frac{1}{10},
\end{align*}
where we use the initial data assumption \eqref{eq:initial4xtildeU}
and use either a small power of $\epsilon$ or $(\log M)^{-1}$ to absorb the constant. Since we improve \eqref{eq:4xtildeUnear0} by a factor of $3/4$, we can close the bootstrap.
\end{proof}

\begin{proposition}
For $|X|\leq l$,
\begin{gather*}
|\partial_X^j\widetilde{U}(X,s)|\leq(\epsilon^\frac{1}{8}+\log M\epsilon^\frac{1}{10})l^{4-j}\leq 2 \log M\epsilon^\frac{1}{10}l^{4-j},\quad j=0,1,2,3.
\end{gather*}
\end{proposition}

\begin{proof}
We first prove that $|\partial^3_X\widetilde{U}(0,s)|\leq\epsilon^\frac{1}{4}$. Plugging in $X=0$ into \eqref{eq:eqnxtildeU} with $n=3$ and using the constraint \eqref{eq:constraints}, we get
\begin{align*}
(\partial_s-\frac{4\dot{\tau}}{1-\dot{\tau}})\partial^3_X\widetilde{U}(0,s)+\frac{e^\frac{s}{2}(\kappa-\dot{\xi})}{1-\dot{\tau}}\partial^4_X\widetilde{U}(0,s)=\frac{e^{-s}}{1-\dot{\tau}}H[\partial^3_XU](0,s)+\frac{24\dot{\tau}}{1-\dot{\tau}}.
\end{align*}
Let $F_{\partial^3_X\tilde{U}(0)}$ denote the forcing on the right hand side. By Lemma \ref{lem:H[U]} and \eqref{eq:taubound}, we have
\begin{align*}
|F_{\partial^3_X\tilde{U}(0)}|&\lesssim e^{-s}M^\frac{5}{2}+e^{-\frac{3}{4} s}\leq e^{-\frac{2}{5}s}.
\end{align*}
where we used a small power of $e^{-s}$ to absorb the constant and $M$ factors. 

We then bound the transport term using Lemma \ref{lem:H[U]}, \eqref{eq:3xUat0} and \eqref{eq:4xtildeUnear0}
\begin{align*}
\Big|\frac{e^\frac{s}{2}(\kappa-\dot{\xi})}{1-\dot{\tau}}\partial^4_X\widetilde{U}(0,s)\Big|&=\frac{e^{-s}}{1-\dot{\tau}}\frac{|H[\partial^2_XU](0,s)|}{|\partial^3_XU(0,s)|}|\partial^4_X\widetilde{U}(0,s)|\\
&\lesssim e^{-s}M^\frac{3}{2}\epsilon^\frac{1}{10}\leq e^{-\frac{2}{5}s}.
\end{align*}
Again we used $\epsilon^\frac{1}{10}$ and a small power of $e^{-s}$ to absorb the constant and $M$ factor.

Finally, we bound the damping term using \eqref{eq:3xtildeUat0}
\begin{align*}
\Big|\frac{4\dot{\tau}}{1-\dot{\tau}}\partial^3_X\widetilde{U}(0,s)\Big|\lesssim e^{-\frac{3}{4} s}\leq e^{-\frac{2}{5}s}.
\end{align*}
Putting the forcing, transport and damping terms together, we have
\begin{align}
|\partial_s\partial_X^3\widetilde{U}(0,s)|\leq 3e^{-\frac{2}{5}s}.\label{eq:partials3X}
\end{align} 
Hence, by Fundamental Theorem of Calculus and \eqref{eq:initial3xtildeU}
\begin{align*}
|\partial_X^3\widetilde{U}(0,s)|&=\Big|\partial_X^3\widetilde{U}(0,-\log\epsilon)+\int_{-\log\epsilon}^s\partial_s\partial_X^3\widetilde{U}(0,s')\,ds'\Big|\\
&\leq |\partial_X^3\widetilde{U}(0,-\log\epsilon)|+\int_{\log\epsilon}^s 3e^{-\frac{2}{5}s'}\,ds'\\
&\leq \frac{1}{4}\epsilon^\frac{1}{4}+\frac{15}{2}\epsilon^\frac{2}{5}\leq \frac{1}{2}\epsilon^\frac{1}{4}.
\end{align*}
This closes the bootstrap of \eqref{eq:3xtildeUat0}. Then by Fundamental Theorem of Calculus, for $|X|\leq l$
\begin{align*}
|\partial_X^3\widetilde{U}(X,s)|&=\Big|\partial_X^3\widetilde{U}(0,s)+\int_0^X\partial_X^4\widetilde{U}(X',s)\,dX'\Big|\\
&\leq |\partial_X^3\widetilde{U}(0,s)|+\int_0^l|\partial_X^4\widetilde{U}(X',s)|\,dX'\\
&\leq \frac{1}{2}\epsilon^\frac{1}{4}+\epsilon^\frac{1}{10}l\\
&\leq \frac{1}{2}(\epsilon^\frac{1}{8}+\log M\epsilon^\frac{1}{10})l,
\end{align*}
We then use the constraints
$$\widetilde{U}(0,s)=\partial_X\widetilde{U}(0,s)=\partial^2_X\widetilde{U}(0,s)=0$$
and Fundamental Theorem of Calculus repeatedly, to get
\begin{align*}
|\partial^j_X\widetilde{U}(X,s)|&\leq |\partial^j_X\widetilde{U}(0,s)|+\int_0^l|\partial_X^{j+1}\widetilde{U}(X',s)|\,dX'\\
&\leq \frac{1}{2}(\epsilon^\frac{1}{8}+\log M\epsilon^\frac{1}{10})l^{4-j},\quad j=2,1,0.
\end{align*}
Therefore, we closed the bootstrap \eqref{eq:xtildeUnear0} for $j=0,1,2,3$.
\end{proof}

\subsection{Weighted estimates for $\widetilde{U}$, $\partial_X\widetilde{U}$, $\partial^2_XU$ when $l\leq |X|\leq \frac{1}{2}e^{\frac{3}{2}s}$}\label{sec:weightmiddle}
Suppose we have a transport equation for function $g(X,s)$
\begin{align*}
\partial_sg+D_g g+V\partial_X g=F_g
\end{align*}
where $D_g$ and $F_g$ denote the damping and forcing terms respectively. We consider $$G(X,s):=(1+X^2)^\mu g(X,s)$$ for some $\mu\in\mathbb{R}$. Then 
\begin{align}
\partial_sG+(\underbracket[0.55pt]{D_g-\frac{2\mu X}{1+X^2}V}_{:=D_G})G+V\partial_XG=\underbracket[0.55pt]{(1+X^2)^\mu F_g}_{:=F_G}.\label{eq:weighted}
\end{align}
Plugging in the expression of $V$ as in \eqref{eq:speed}, we get
\begin{align*}
D_G=D_g-\frac{3\mu X^2}{1+X^2}-\frac{2\mu X}{1+X^2}\frac{U+e^\frac{s}{2}(\kappa-\dot{\xi})}{1-\dot{\tau}}.
\end{align*}
Composing \eqref{eq:weighted} with Lagrangian trajectory $\Phi(X_0,s)$, we have
\begin{align*}
\frac{d}{ds}G\circ\Phi(X_0,s)&+(D_G)\circ \Phi(X_0,s)=F_G\circ\Phi(X_0,s)\\
G\circ \Phi(X_0,s)&=G(X_0,s_0)\exp{\Big(-\int_{s_0}^sD_G\circ\Phi(X_0,s')\,ds'\Big)}\\
&\qquad+\int_{s_0}^sF_G\circ \Phi(X_0,s')\exp{\Big(-\int_{s'}^sD_G\circ\Phi(X_0,s'')\,ds''\Big)}\,ds'.
\end{align*}
Therefore,
\begin{equation}
\begin{aligned}
|G\circ \Phi(X_0,s)|&\leq |G(X_0,s_0)|\exp{\Big(-\int_{s_0}^sD_G\circ\Phi(X_0,s')\,ds'\Big)}\\
&\qquad+\int_{s_0}^s|F_G\circ \Phi(X_0,s')|\exp{\Big(-\int_{s'}^sD_G\circ\Phi(X_0,s'')\,ds''\Big)}\,ds'.
\end{aligned}\label{eq:gronwall}
\end{equation}
To conclude the proof, we need the size of $G(X_0,s_0)$. Since by Lemma \ref{lem:lowerLagrangian}, Lagrangian trajectories escape exponentially fast once $|X_0|\geq l$, for any $s>-\log\epsilon$ and any $l\leq |X|\leq \frac{1}{2}e^{\frac{3}{2}s}$, there exists $s_0\in [-\log\epsilon,s)$ and $l \leq |X_0|\leq \frac{1}{2}e^{\frac{3}{2}s_0}$ such that $X=\Phi(X_0,s)$. When $s_0=-\log\epsilon$, we can use the initial data assumptions. On the other hand, when $s_0>-\log\epsilon$, $s_0$ is the first time that the trajectory enters the range $l\leq |X|\leq \frac{1}{2}e^{\frac{3}{2}s}$, so we must have $|X_0|=l$, hence we can use the bootstrap assumptions on $\partial^j_X\widetilde{U}$ for $|X|\leq l$ (in fact, at $|X|=l$).

\begin{proposition}[$\widetilde{U}$]
For any $(X,s)$ such that $l\leq |X|\leq \frac{1}{2}e^{\frac{3}{2}s}$,
\begin{align*}
|\widetilde{U}(X,s)|\leq \epsilon^\frac{1}{11}(1+X^2)^\frac{1}{6}.
\end{align*}
\end{proposition}

\begin{proof}
Here $g=\widetilde{U}$. We consider \eqref{eq:eqtildeU}, then $D_{\widetilde{U}}=-\frac{1}{2}+\frac{\partial_X\overline{U}}{1-\dot{\tau}}$, $F_g=F_{\tilde{U}}$. We set $\mu=-\frac{1}{6}$ so $G=(1+X^2)^{-1/6}\widetilde{U}$, and we use \eqref{eq:kappaxi} to rewrite $D_G$ into
\begin{align*}
D_G(X,s)&=-\frac{1}{2}+\frac{\partial_X\overline{U}(X)}{1-\dot{\tau}}+\frac{X^2}{2(1+X^2)}+\frac{1}{1-\dot{\tau}}\frac{X}{3(1+X^2)}\Big(U(X,s)+e^{-s}\frac{H[\partial^2_XU](0,s)}{\partial^3_XU(0,s)}\Big).
\end{align*} 
We note that, by \eqref{eq:1xbarUmiddle} and Lemma \ref{lem:1-dottau} (the first three terms of $D_G$ with a minus sign)
\begin{align*}
\frac{1}{2}-\frac{\partial_X\overline{U}(X)}{1-\dot{\tau}}-\frac{X^2}{2(1+X^2)}&=\frac{1}{2(1+X^2)}-\frac{\partial_X\overline{U}(X)}{1-\dot{\tau}}\\
&\leq\frac{1}{2(1+X^2)}+(1+2\epsilon^\frac{3}{4})(1-2l^2)(1+X^2)^{-\frac{1}{3}}\\
&\leq (\frac{1}{2}+1+2\epsilon^\frac{3}{4})(1+X^2)^{-\frac{1}{3}}\\
&\leq 2(1+X^2)^{-\frac{1}{3}},
\end{align*}
and that by \eqref{eq:Umiddle},  \eqref{eq:3xUat0}, Lemma \ref{lem:1-dottau} and \ref{lem:H[partialjU]},
the last term of $D_G$ is bounded by
\begin{align*}
&\Big|\frac{1}{1-\dot{\tau}}\frac{X}{3(1+X^2)}\Big(U(X,s)+e^{-s}\frac{H[\partial^2_XU](0,s)}{\partial^3_XU(0,s)}\Big)\Big|\\
\leq& \frac{1}{3}(1+2\epsilon^\frac{3}{4})(1+X^2)^{-\frac{1}{2}}\Big((1+\epsilon^\frac{1}{11})(1+X^2)^\frac{1}{6}+CM^\frac{3}{2}e^{-s}\Big)\\
\leq& \frac{1}{3}(1+2\epsilon^\frac{3}{4})(1+\epsilon^\frac{1}{11})(1+X^2)^{-\frac{1}{3}}+CM^\frac{3}{2}e^{-s}(1+X^2)^{-\frac{1}{2}}\\
\leq& (1+X^2)^{-\frac{1}{3}}.
\end{align*}
where we use the $e^{-s}$ to absorb the constant and $M$ factor. So 
\begin{align*}
-D_G(X,s)\leq 3(1+X^2)^{-\frac{1}{3}}.
\end{align*}
By Lemma \ref{lem:forcemiddle}, we  have
\begin{align*}
|F_G(X,s)|=(1+X^2)^{-\frac{1}{6}}|F_{\tilde{U}}(X,s)|\lesssim Me^{-\frac{1}{4}s}(1+X^2)^{-\frac{1}{6}}.
\end{align*}
Plugging $D_G$ and $F_G$ into \eqref{eq:gronwall}, and use Lemma \ref{lem:lowerLagrangian}, we have
\begin{align*}
|(1+X^2)^{-\frac{1}{6}}\widetilde{U}\circ\Phi(X_0,s)|&\leq (1+X_0^2)^{-\frac{1}{6}}|\widetilde{U}(X_0,s_0)|\exp{\Big(\int_{s_0}^s 3\big(1+X_0^2e^{\frac{2}{5}(s'-s_0)}\big)^{-\frac{1}{3}}ds'\Big)}\\
&+CM\int_{s_0}^s e^{-\frac{1}{4}s'}\big(1+X_0^2e^{\frac{2}{5}(s'-s_0)}\big)^{-\frac{1}{6}}\exp{\Big(\int_{s'}^s 3\big(1+X_0^2e^{\frac{2}{5}(s''-s_0)}\big)^{-\frac{1}{3}}ds'' \Big)}ds'.
\end{align*}
Since $|X_0|\geq l$, we have
\begin{equation}
\begin{aligned}
\int_{s_0}^s \big(1+X_0^2e^{\frac{2}{5}(s'-s_0)}\big)^{-\frac{1}{3}}ds'&\leq \int_{s_0}^s \big(1+l^2e^{\frac{2}{5}(s'-s_0)}\big)^{-\frac{1}{3}}ds'\\
&\leq \int_{s_0}^{s_0+5\log{\frac{1}{l}}}1\,ds'+\int_{s_0+5\log{\frac{1}{l}}}^s l^{-\frac{2}{3}}e^{-\frac{2}{15}(s-s_0)}\,ds'\\
&\leq 5\log{\frac{1}{l}}+\frac{15}{2}\leq 10\log\frac{1}{l},\label{eq:logl}
\end{aligned}
\end{equation}
if $M$ is sufficiently large. Hence, we can bound the contribution from the forcing term by
\begin{equation*}
\begin{aligned}
&CM\int_{s_0}^s e^{-\frac{1}{4}s'}\big(1+X_0^2e^{\frac{2}{5}(s'-s_0)}\big)^{-\frac{1}{6}}\exp{\Big(\int_{s'}^s 3\big(1+X_0^2e^{\frac{2}{5}(s''-s_0)}\big)^{-\frac{1}{3}}ds'' \Big)}ds'\\
\leq &CM\int_{s_0}^s e^{-\frac{1}{4}s'}\big(1+l^2e^{\frac{2}{5}(s'-s_0)}\big)^{-\frac{1}{6}}l^{-30}\,ds'\\
\leq &CM \epsilon^\frac{1}{4}l^{-30}\int_{s_0}^s\big(1+l^2e^{\frac{2}{5}(s'-s_0)}\big)^{-\frac{1}{6}}\,ds'\\
\leq &CM\epsilon^\frac{1}{4}l^{-30}\log{\frac{1}{l}}.
\end{aligned}
\end{equation*}
Therefore,
\begin{align*}
|(1+X^2)^{-\frac{1}{6}}\widetilde{U}\circ\Phi(X_0,s)|\leq l^{-30}(1+X_0^2)^{-\frac{1}{6}}|\widetilde{U}(X_0,s_0)|+CM\epsilon^\frac{1}{4}l^{-30}\log{\frac{1}{l}}.
\end{align*}
By the discussion after \eqref{eq:gronwall}, either $s_0=-\log\epsilon,\ l\leq |X_0|\leq \frac{1}{2}\epsilon^{-\frac{3}{2}}$, or $s_0>-\log\epsilon,\ X_0=l$.
In the first case, we use the initial data assumption \eqref{eq:inittildeUmiddle}  at $|X_0|=l$ to get
\begin{align*}
|(1+X^2)^{-\frac{1}{6}}\widetilde{U}\circ\Phi(X_0,s)|\leq l^{-30}\epsilon^\frac{1}{10}+CM\epsilon^\frac{1}{4}l^{-30}\log\frac{1}{l},
\end{align*}
while in the second case, we use the assumption \eqref{eq:xtildeUnear0} with $j=0$ to get
\begin{align*}
|(1+X^2)\widetilde{U}\circ\Phi(l,s)|\leq 2l^{-30}\log M\epsilon^\frac{1}{10}l^4+CM\epsilon^\frac{1}{4}l^{-30}\log\frac{1}{l}.
\end{align*}
In either case, by using a small power of $\epsilon$ to absorb all the $M$ and $l$ factors, we can show that
\begin{align*}
|(1+X^2)^{-\frac{1}{6}}\widetilde{U}(X,s)|&\leq \frac{1}{2}\epsilon^\frac{1}{11}.
\end{align*}
So we can improve the bootstrap assumption \eqref{eq:tildeUmiddle} by a factor $1/2$, thus we can close the bootstrap.
\end{proof}

\begin{proposition}[$\partial_X\widetilde{U}$]
For any $(X,s)$ such that $l\leq |X|\leq \frac{1}{2}e^{\frac{3}{2}s}$,
\begin{align*}
|\partial_X\widetilde{U}(X,s)|&\leq \epsilon^\frac{1}{12}(1+X^2)^{-\frac{1}{3}}.
\end{align*}\label{prop:1xmiddleclosure}
\end{proposition}
\begin{proof}
The proof is very similar to the previous one for $\widetilde{U}$, so we omit some details. We consider \eqref{eq:eq1xtildeU}.  Here $g=\partial_X\widetilde{U}$, $D_{\partial_X\tilde{U}}=1+\frac{2\partial_X\overline{U}+\partial_X\widetilde{U}}{1-\dot{\tau}}$, and $F_g=F_{\partial_X\tilde{U}}$. We choose $\mu=\frac{1}{3}$. So $G=(1+X^2)^\frac{1}{3}\partial_X\widetilde{U}$, and
\begin{align*}
D_G(X,s)&=1+\frac{2\partial_X\overline{U}(X)+\partial_X\widetilde{U}(X,s)}{1-\dot{\tau}}-\frac{X^2}{1+X^2}-\frac{2X}{3(1+X^2)}\frac{1}{1-\dot{\tau}}\Big(U(X,s)+\frac{e^{-s}H[\partial^2_XU](0,s)}{\partial^3_XU(0,s)}\Big)\\
&\geq -6(1+X^2)^{-\frac{1}{3}}
\end{align*}
by \eqref{eq:barUall}, \eqref{eq:1xUmiddle}, Lemma \ref{lem:1-dottau}, \eqref{eq:Umiddle}, Lemma \ref{lem:H[U]}, and we use the $e^{-s}$ to absorb the constant and $M$ factor.
By Lemma \ref{lem:forcemiddle}, we have
\begin{align*}
|F_G(X,s)|=(1+X^2)^{\frac{1}{3}}|F_{\partial_X\tilde{U}}(X,s)|\lesssim e^{-s}\big(M^\frac{1}{4}+\log(1+X^2)\big)+e^{-\frac{3}{4}s}+\epsilon^\frac{1}{11}(1+X^2)^{-\frac{1}{3}}.
\end{align*}
By \eqref{eq:logl}, we have
\begin{align*}
\int_{s_0}^s-D_G\circ\Phi(X_0,s')\,ds\leq 60\log\frac{1}{l}.
\end{align*}
To bound the contribution of the forcing term, it remains to bound $\int_{s_0}^s |F_G\circ\Phi(X_0,s')|\,ds'$. By Lemma \ref{lem:lowerLagrangian} and Lemma \ref{lem:upperLagrangian},
\begin{align*}
\int_{s_0}^s |F_G\circ\Phi(X_0,s')|\,ds'&\lesssim \int_{s_0}^s M^\frac{1}{4}e^{-s'}+e^{-s'}\big[2\log(|X_0|+\frac{7}{2}Me^{\frac{1}{2}s_0})+3(s'-s_0)\big]\\
&\qquad+e^{-\frac{3}{4}s'}+\epsilon^\frac{1}{11}\big(1+X_0^2e^{\frac{2}{5}(s-s_0)}\big)^{-\frac{1}{3}}\,ds'\\
(\text{since }|X_0|\leq \frac{1}{2}e^{\frac{3}{2}s_0})\qquad&\lesssim M^\frac{1}{4}e^{-s_0}+\int_{s_0}^s e^{-s'}\log\big(\frac{1}{2}e^{\frac{3}{2}s_0}+\frac{7}{2}Me^{\frac{1}{2}s_0}\big)\,ds'+\int_{s_0}^s e^{-s'}s'\,ds'\\
&\qquad+\int_{s_0}^se^{-\frac{3}{4}s'}\,ds'+\epsilon^\frac{1}{11}\int_{s_0}^s\big(1+X_0^2e^{\frac{2}{5}(s-s_0)}\big)^{-\frac{1}{3}}\,ds'\\
(\text{log term}\lesssim \log(4Me^{\frac{3}{2}s_0}))\qquad&\lesssim M^\frac{1}{4}e^{-s_0}+s_0e^{-s_0}\log{M}+(1+s_0)e^{-s_0}+e^{-\frac{3}{4}s_0}+\epsilon^\frac{1}{11}\log\frac{1}{l}\\
(s_0\geq-\log\epsilon\text{ large enough})\qquad  &\lesssim \epsilon^\frac{1}{11}\log\frac{1}{l}.
\end{align*}
Hence, using \eqref{eq:gronwall}, we have
\begin{align*}
    |(1+X^2)^\frac{1}{3}\partial_X\widetilde{U}\circ\Phi(X_0,s)|&\leq (1+X_0^2)^\frac{1}{3}|\partial_X\widetilde{U}(X_0,s_0)|l^{-60}+C\epsilon^\frac{1}{11}l^{-60}\log\frac{1}{l}.
\end{align*}
If $l\leq X_0\leq \frac{1}{2}\epsilon^{-\frac{3}{2}},\ s_0=-\log\epsilon$, we use the initial data assumption \eqref{eq:init1xtildeUmiddle} to obtain
\begin{align*}
|(1+X^2)^\frac{1}{3}\partial_X\widetilde{U}\circ\Phi(X_0,s)|\leq \epsilon^\frac{1}{11}l^{-60}+C\epsilon^\frac{1}{11}l^{-60}\log\frac{1}{l}.
\end{align*}
Otherwise, $X_0=l$ for some $s_0>-\log\epsilon$, and we use \eqref{eq:xtildeUnear0} with $j=1$ to obtain
\begin{align*}
|(1+X^2)^\frac{1}{3}\partial_X\widetilde{U}\circ\Phi(X_0,s)|\leq 
2\log M\epsilon^\frac{1}{10}l^3(1+l^2)^\frac{1}{3}l^{-60}+C\epsilon^\frac{1}{11}l^{-60}\log\frac{1}{l}.
\end{align*}
In either case, by using a small power of $\epsilon$ to absorb the constant and $M$ factor, we can show that
\begin{align*}
|(1+X^2)^\frac{1}{3}\partial_X\widetilde{U}(X,s)|&\leq \frac{1}{2}\epsilon^\frac{1}{12}.
\end{align*}
Thus we close the bootstrap assumption \eqref{eq:1xtildeUmiddle}.
\end{proof}

\begin{proposition}[$\partial_X^2U$]
For any $(X,s)$ such that $l\leq |X|\leq \frac{1}{2}e^{\frac{3}{2}s}$,
\begin{align*}
|\partial_X^2U(X,s)|\leq M^\frac{1}{4}(1+X^2)^{-\frac{1}{3}}.
\end{align*}
\end{proposition}

\begin{proof}
Again we will skip some intermediate steps. We consider \eqref{eq:eqnxtildeU} with $n=2$. Here $g=\partial_X^2U$, $D_{\partial_X^2U}=\frac{5}{2}+\frac{3\partial_XU}{1-\dot{\tau}}$, and $F_g=F_{\partial_X^2U}$. We set $\mu=\frac{1}{3}$, so $G=(1+X^2)^\frac{1}{3}\partial_X^2U$, and
\begin{align*}
D_G(X,s)
&=\frac{3}{2}+\frac{1}{1+X^2}+\frac{3\partial_XU(X,s)}{1-\dot{\tau}}-\frac{2X}{3(1+X^2)}\frac{1}{1-\dot{\tau}}\Big(U(X,s)+\frac{e^{-s}H[\partial^2_XU](0,s)}{\partial^3_XU(0,s)}\Big)\\
&\geq\frac{3}{2}-6(1+X^2)^{-\frac{1}{3}},
\end{align*}
by \eqref{eq:1xUmiddle}, \eqref{eq:Umiddle}, \eqref{eq:3xUat0}, Lemma \ref{lem:H[U]}, and using $e^{-s}$ to absorb implicit constants and powers of $M$. By  \eqref{eq:logl}, we get
\begin{align*}
\int_{s_0}^s-D_G\circ\Phi(X_0,s')\,ds'\leq -\frac{3}{2}(s-s_0)+60\log\frac{1}{l}.
\end{align*}
By Lemma \ref{lem:forcemiddle}, 
\begin{align*}
|F_G(X,s)|=(1+X^2)^\frac{1}{3}|F_{\partial_X^2U}(X,s)|\lesssim e^{-s}\big(M^\frac{3}{4}+M^\frac{1}{4}\log(1+X^2)\big)+Me^{-\frac{3}{4}s}.
\end{align*}
Using a similar argument as the forcing contribution in the proof of $\partial_X\widetilde{U}$, we have
\begin{align*}
\int_{s_0}^s|F_G\circ\Phi(X_0,s')|\,ds'
\lesssim M\epsilon^\frac{3}{4}.
\end{align*}
Hence, by \eqref{eq:gronwall}, we have
\begin{align*}
|(1+X^2)^\frac{1}{3}\partial^2_XU\circ\Phi(X_0,s)|
&\leq (1+X_0^2)^\frac{1}{3}|\partial_X^2U(X_0,s_0)|e^{-\frac{3}{2}(s-s_0)}l^{-60}+CM\epsilon^\frac{3}{4}l^{-60}.
\end{align*}
If $l\leq|X_0|\leq \frac{1}{2}\epsilon^{-\frac{3}{2}},\ s_0=-\log\epsilon$, then we use the initial data assumption \eqref{eq:init2xUmiddle} to get
\begin{align*}
|(1+X^2)^\frac{1}{3}\partial^2_XU\circ\Phi(X_0,s)|&\leq 2e^{-\frac{3}{2}(s+\log\epsilon)}l^{-60}+CM\epsilon^\frac{3}{4}l^{-60}.
\end{align*}
Otherwise, $X_0=l$ for some $s_0>-\log\epsilon$, so we can use \eqref{eq:xtildeUnear0} with $j=2$ and \eqref{eq:barUall} to obtain
\begin{align*}
|(1+X^2)^\frac{1}{3}\partial^2_XU\circ\Phi(X_0,s)|&\leq (1+l^2)^\frac{1}{3}\big(2\log M\epsilon^\frac{1}{10}l^2+1\big)e^{-\frac{3}{2}(s-s_0)}l^{-60}+CM\epsilon^\frac{3}{4}l^{-60}.
\end{align*}
In either case, recall that $l=(\log M)^{-2}$, we can show that
\begin{align*}
(1+X^2)^\frac{1}{3}|\partial_X^2U(X,s)|&\leq 2l^{-60}+CM\epsilon^\frac{3}{4}l^{-60}\leq \frac{3}{4}M^\frac{1}{4},
\end{align*}
if we take $M$ sufficiently large so that $(\log M)^{120}\leq \frac{3}{16}M^\frac{1}{4}$, and $\epsilon$ sufficiently small to absorb the $CM$ factor. This closes \eqref{eq:2xUmiddle}.
\end{proof}

\subsection{$\partial^3_XU$ for $l\leq |X|\leq \frac{1}{2}e^{\frac{3}{2}s}$}
\begin{proposition}
For any $(X,s)$ such that $l\leq|X|\leq\frac{1}{2}e^{\frac{3}{2}s}$,
\begin{align*}
|\partial^3_XU(X,s)|\leq \frac{1}{2}M^\frac{3}{4}.
\end{align*}
\end{proposition}
\begin{proof}
We consider \eqref{eq:npartialU} with $n=3$. By Lemma \ref{lem:1-dottau} and \eqref{eq:1xUmiddle}, we have 
\begin{align*}
D_{\partial^3_XU}(X,s)=4+\frac{4\partial_XU}{1-\dot{\tau}}&\geq 4-4(1+2\epsilon^\frac{3}{4})(1+X^2)^{-\frac{1}{3}}\geq 4-6(1+X^2)^{-\frac{1}{3}}.
\end{align*}
By \eqref{eq:logl}, we get
\begin{align*}
\int_{s_0}^s-D_{\partial_X^3U}\circ\Phi(X_0,s')\,ds'\leq -4(s-s_0)+60\log\frac{1}{l}.
\end{align*}
From the bound  for $F_{\partial_X^3U}$ in Lemma \ref{lem:forcemiddle}, and a similar argument as in \eqref{eq:logl} with a better decay exponent, we get the forcing contribution
\begin{align*}
\int_{s_0}^s|F_{\partial_X^3U}\circ\Phi(X_0,s')|\,ds'&\lesssim \int_{s_0}^s M^\frac{5}{2}e^{-s'}\,ds'+\int_{s_0}^s M^\frac{1}{2}\big(1+l^2e^{\frac{2}{5}(s'-s_0)}\big)^{-\frac{2}{3}}\,ds'\\
&\lesssim M^\frac{5}{2}e^{-s_0}+\frac{15}{2}M^\frac{1}{2}\log\frac{1}{l}.
\end{align*}
Hence, by \eqref{eq:gronwall}, we have
\begin{align*}
|\partial_X^3U\circ\Phi(X_0,s)|&\leq |\partial_X^3U(X_0,s_0)|e^{-4(s-s_0)}l^{-60}+CM^\frac{5}{2}e^{-s_0}l^{-60}+CM^\frac{1}{2}l^{-60}\log\frac{1}{l}.
\end{align*}
If $l\leq |X_0|\leq \frac{1}{2}\epsilon^{-\frac{3}{2}}$, $s_0=-\log\epsilon$, then we use the initial data assumption \eqref{eq:init3xmiddle} to get
\begin{align*}
|\partial^3_XU(X,s)|\leq M^\frac{1}{2}e^{-4(s+\log\epsilon)}l^{-60}+CM^\frac{1}{2}
l^{-60}\log\frac{1}{l}.
\end{align*}
Otherwise, $s_0>-\log\epsilon$, $X_0=l$ and we use \eqref{eq:barUnear0}, \eqref{eq:xtildeUnear0} with $j=3$ to get
\begin{align*}
|\partial^3_XU(X,s)|&\leq
\big(|\partial^3_X\overline{U}(l)|+|\partial^3_X\widetilde{U}(l,s_0)|\big)e^{-4(s-s_0)}l^{-60}+CM^\frac{1}{2}l^{-60}\log\frac{1}{l}\\
&\leq\big(6+2\log M\epsilon^\frac{1}{10}l\big)l^{-60}+CM^\frac{1}{2}l^{-60}\log\frac{1}{l}.
\end{align*}
In either case, by making $M$ sufficiently large so that 
\begin{gather*}
(\log M)^{120}(1+C\log\log M)\leq\frac{3}{8}M^\frac{1}{4},\\
(\log M)^{120}\big(6+2(\log M)^{-1}\epsilon^\frac{1}{10}+CM^\frac{1}{2}\log\log M\big)\leq \frac{3}{8}M^\frac{3}{4},
\end{gather*}
we can obtain
\begin{align*}
|\partial^3_XU(X,s)|\leq \frac{3}{8}M^\frac{3}{4},
\end{align*}
which closes the bootstrap \eqref{eq:3xUmiddle}.
\end{proof}

\subsection{Temporal decay of $\partial_XU$, $\partial_X^2U$ for $|X|\geq \frac{1}{2}e^{\frac{3}{2}s}$}
From \eqref{eq:1partialU}, \eqref{eq:npartialU} with $n=2$, we get equations for $e^s\partial_XU$ and $e^s\partial_X^2U$
\begin{gather}
\Big(\partial_s+\frac{\partial_XU}{1-\dot{\tau}}\Big)(e^s\partial_XU)+V\partial_X(e^s\partial_XU)=\frac{1}{1-\dot{\tau}}H[\partial_XU],\label{eq:eq1xfar}\\
\Big(\partial_s+\frac{3}{2}+\frac{3\partial_XU}{1-\dot{\tau}}\Big)(e^s\partial_X^2U)+V\partial_X(e^s\partial_X^2U)=\frac{1}{1-\dot{\tau}}H[\partial_X^2U]\label{eq:eq2xfar}.
\end{gather}

\begin{proposition}[$\partial_XU$]
For any $(X,s)$ such that $|X|\geq \frac{1}{2}e^{\frac{3}{2}s}$, 
\begin{align*}
|\partial_XU(X,s)|&\leq 2e^{-s}.
\end{align*}
\end{proposition}

\begin{proof}
We consider \eqref{eq:eq1xfar}.  Denote $W:=e^s\partial_XU$. Then $D_W=\frac{\partial_XU}{1-\dot{\tau}}$. By Lemma \ref{lem:1-dottau} and  \eqref{eq:1xUfar}, we have
\begin{align*}
-D_W(X,s)\leq 2(1+2\epsilon^\frac{3}{4})e^{-s}\leq 3e^{-s}.
\end{align*}
Hence, for $(X_0,s_0)$ such that $|X_0|\geq \frac{1}{2}e^{\frac{3}{2}s_0}$, we have
\begin{align}
\exp\Big(\int_{s_0}^s -D_W\circ\Phi(X_0,s')\,ds'\Big)\leq e^{3e^{-s_0}}\leq e^{3\epsilon}\leq 1+2\cdot 3\epsilon\leq \frac{3}{2},\label{eq:Taylor}
\end{align}
by Taylor expansion $e^x\approx 1+x+o(x)$ for $x\approx 0$. Also $F_W=\frac{1}{1-\dot{\tau}}H[\partial_XU]$. By \eqref{eq:H1xUfar} in Lemma \ref{lem:H1x}, we have
$|F_W(X,s)|\lesssim e^{-\frac{1}{4}s}$. By composing \eqref{eq:eq1xfar} with Lagrangian trajectory and using a similar Gr\"onwall type argument as \eqref{eq:gronwall}, and \eqref{eq:Taylor},
we have
\begin{align*}
e^s|\partial_XU\circ\Phi(X_0,s)|&\leq e^{s_0}|\partial_XU(X_0,s_0)|\cdot\frac{3}{2}+C\int_{s_0}^se^{-\frac{1}{4}s'}\cdot\frac{3}{2}\,ds'\\
&\leq \frac{3}{2}e^{s_0}|\partial_XU(X_0,s_0)|+Ce^{-\frac{1}{4}s_0}.
\end{align*}
By the same reasoning after \eqref{eq:gronwall}, either $|X_0|\geq \frac{1}{2}\epsilon^{-\frac{3}{2}},\ s_0=-\log\epsilon$, or $|X_0|=\frac{1}{2}e^{\frac{3}{2}s_0}$ for some $s_0>-\log\epsilon$. In the first case, we use initial data assumption \eqref{eq:1xUfar}; in the second case, we use \eqref{eq:1xUmiddle} at $|X_0|=\frac{1}{2}e^{\frac{3}{2}s_0}$ and \eqref{eq:1xbarUfar} to get
\begin{align*}
|\partial_XU(\pm\frac{1}{2}e^{\frac{3}{2}s_0},s_0)|\leq (\frac{7}{20}+\epsilon^\frac{1}{12})\big(1+\frac{1}{4}e^{-3s_0}\big)^{-\frac{1}{3}}\leq \frac{9}{20}\cdot 2^\frac{3}{2}e^{-s_0}\leq \frac{3}{4}e^{-s_0}.
\end{align*}
Combining both cases, by using $\epsilon^\frac{1}{4}$ to absorb the constant, we get
\begin{align*}
e^s|\partial_XU(X,s)|\leq \frac{3}{2}\cdot \frac{3}{4}+C\epsilon^\frac{1}{4}\leq \frac{3}{2}<2.
\end{align*}
This closes the bootstrap assumption \eqref{eq:1xUfar}.
\end{proof}

\begin{proposition}[$\partial_X^2U$]
For any $(X,s)$ such that $|X|\geq \frac{1}{2}e^{\frac{3}{2}s}$,
\begin{align*}
|\partial_X^2U(X,s)|\leq 4M^\frac{1}{4}e^{-s}.
\end{align*}
\end{proposition}

\begin{proof}
The $\partial_X^2U$ estimate is similar. We consider \eqref{eq:eq2xfar}.  Denote $Z:=e^s\partial_X^2U$, so that $D_Z=\frac{3}{2}+\frac{3\partial_XU}{1-\dot{\tau}}$, $F_Z=\frac{1}{1-\dot{\tau}}H[\partial_X^2U]$. By \eqref{eq:1xUfar},
\begin{align*}
-D_Z(X,s)\leq -\frac{3}{2}+3(1+2\epsilon^\frac{3}{4})\cdot 2e^{-s}\leq -\frac{3}{2}+8\epsilon\leq -1.
\end{align*}
And by \eqref{eq:H2xUfar} in Lemma \ref{lem:H2x}, we have $|F_Z(X,s)|\lesssim M^\frac{1}{4}e^{-\frac{1}{4}s}$. Composing \eqref{eq:eq2xfar} with Lagrangian trajectory, and using the damping and forcing bounds, we have
\begin{align*}
e^s|\partial_X^2U\circ\Phi(X_0,s)|&\leq e^{s_0}|\partial_X^2U(X_0,s_0)|e^{-(s-s_0)}+CM^\frac{1}{4}\int_{s_0}^se^{-\frac{1}{4}s'}e^{-(s-s')}\,ds'\\
&\leq e^{s_0}|\partial_X^2U(X_0,s_0)|+CM^\frac{1}{4}e^{-\frac{1}{4}s}.
\end{align*}
If $|X_0|\geq \frac{1}{2}\epsilon^{-\frac{3}{2}},s_0=-\log\epsilon$, then we use the initial data assumption \eqref{eq:init2xUfar}; if $|X_0|=\frac{1}{2}e^{\frac{3}{2}s_0}$ for some $s_0>-\log\epsilon$, then we use \eqref{eq:2xUmiddle} at $|X_0|=\frac{1}{2}e^{\frac{3}{2}s_0}$
\begin{align*}
|\partial_X^2U(\pm\frac{1}{2}e^{\frac{3}{2}s_0},s_0)|\leq M^\frac{1}{4}\big(1+\frac{1}{4}e^{3s_0}\big)^{-\frac{1}{3}}\leq 2M^\frac{1}{4}e^{-s_0}.
\end{align*}
So in both cases, by using $\epsilon^\frac{1}{4}$ to absorb the constant, we have
\begin{align*}
e^s|\partial_X^2U(X,s)|\leq 2M^\frac{1}{4}+CM^\frac{1}{4}\epsilon^\frac{1}{4}\leq 3M^\frac{1}{4}<4M^\frac{1}{4}.
\end{align*}
This closes the bootstrap \eqref{eq:2xUfar}.
\end{proof}

\subsection{$L^\infty$-norm of $U+e^\frac{s}{2}\kappa$}
\begin{proposition}
The $L^\infty$-norm of $e^{-\frac{s}{2}}U(\cdot,s)+\kappa$ is bounded uniformly in $s$. More precisely,
\begin{align*}
\|e^{-\frac{s}{2}}U(\cdot,s)+\kappa\|_{L^\infty}\leq M,\quad\text{i.e.}\quad
\|U(\cdot,s)+e^\frac{s}{2}\kappa\|_{L^\infty}\leq Me^\frac{s}{2}.
\end{align*}
\end{proposition}

\begin{proof}
From \eqref{eq:BH} or \eqref{eq:ansatz} we get equation for $e^{-\frac{s}{2}}U(\cdot,s)+\kappa$ (see Appendix \ref{ap:derivation} for the derivation)
\begin{align}
    \partial_s(e^{-\frac{s}{2}}U+\kappa)+V\partial_X(e^{-\frac{s}{2}}U+\kappa)=\frac{e^{-\frac{3}{2}s}}{1-\dot{\tau}}H[U+e^\frac{s}{2}\kappa].\label{eq:U+kappa}
\end{align}
By Lemma \ref{lem:H[U]}, we can bound the forcing term on the right hand side
\begin{align*}
\Big|\frac{e^{-\frac{3}{2}s}}{1-\dot{\tau}}H[U+e^\frac{s}{2}\kappa]\Big|\lesssim Me^{-\frac{3}{4}s}.
\end{align*}
Composing \eqref{eq:U+kappa} with Lagrangian trajectory we have
\begin{align*}
    |(e^{-\frac{s}{2}}U+\kappa)\circ\Phi(X_0,s)|&\leq |e^{\frac{\log{\epsilon}}{2}}U(X_0,-\log{\epsilon})+\kappa(0)|\\
    &\qquad+\int_{-\log{\epsilon}}^s\frac{e^{-\frac{3}{2}s'}}{1-\dot{\tau}}\big|H[U+e^\frac{s'}{2}\kappa]\circ\Phi(X_0,s')\big|\,ds'\\
    &\leq \frac{M}{2}+C\int_{-\log{\epsilon}}^\infty Me^{-\frac{3}{4}s'}\,ds'\\
    &\leq \frac{M}{2}+CM\epsilon^\frac{3}{4}\leq \frac{3}{4}M 
\end{align*}
for all $X_0$. We thus close the assumption \eqref{eq:ULinfty}.
\end{proof}

\subsection{$\partial_X^3U$ for $|X|\geq \frac{1}{2}e^{\frac{3}{2}s}$}
\begin{proposition}[$\partial^3_XU$]
For any $(X,s)$ such that $|X|\geq \frac{1}{2}e^{\frac{3}{2}s}$, we have
\begin{align*}
|\partial^3_XU(X,s)|\leq M^\frac{3}{4}.
\end{align*}
\end{proposition}
\begin{proof}
Consider \eqref{eq:npartialU} with $n=3$. For $|X|\geq \frac{1}{2}e^{-\frac{3}{2}s}$, by \eqref{eq:1xUfar} and Lemma \ref{lem:1-dottau}, we have
\begin{align*}
D_{\partial^3_XU}(X,s)\geq 4-4(1+2\epsilon^\frac{3}{4})2e^{-s}\geq 4-8\cdot \frac{3}{2}\epsilon\geq   \frac{7}{2}.
\end{align*}
Hence, by Lemma \ref{lem:forcefar},
\begin{align*}
|\partial^3_XU\circ\Phi(X_0,s)|&\leq |\partial^3_XU(X_0,s_0)|e^{-\frac{7}{2}(s-s_0)}+CM^\frac{5}{2}\int_{s_0}^se^{-s'}e^{-\frac{7}{2}(s-s')}\,ds'\\
&\leq |\partial^3_XU(X_0,s_0)|+CM^\frac{5}{2}e^{-s}.
\end{align*}
If $|X_0|\geq\frac{1}{2}\epsilon^{-\frac{3}{2}}, s_0=-\log\epsilon$, then we use the initial data assumption \eqref{eq:init3xUfar}; if $|X_0|=\frac{1}{2}e^{\frac{3}{2}s_0}$ for some $s_0>-\log\epsilon$, then we use \eqref{eq:3xUmiddle} at $|X_0|=\frac{1}{2}e^{\frac{3}{2}s_0}$. In both cases, by using $\epsilon$ to absorb $M^\frac{7}{4}$ and the constant, we get
\begin{align*}
|\partial_X^3U(X,s)|\leq \frac{1}{2}M^\frac{3}{4}+CM^\frac{5}{2}\epsilon\leq \frac{3}{4}M^\frac{3}{4}.
\end{align*}
This closes the bootstrap \eqref{eq:3xUfar}.
\end{proof}

\subsection{Dynamic modulation variables}\label{sec:blowuptime}
\begin{proposition}[$\tau$ and $T_*$]
The dynamic variable $\tau$ satisfies 
\begin{gather*}
|\dot{\tau}(t)|\leq e^{-\frac{3}{4}s},\quad |\tau(t)|\leq 2\epsilon^\frac{7}{4}.
\end{gather*}
Moreover, $\tau(t)>t$ for all $t\in [-\epsilon,T_*)$. The blowup time $T_*$ satisfies $|T_*|\leq 2\epsilon^\frac{7}{4}$. 

\end{proposition}
\begin{proof}
By \eqref{eq:tau} and Lemma \ref{lem:H[U]}, we have
\begin{align*}
|\dot{\tau}|\leq e^{-s}\|H[\partial_XU](\cdot,s)\|_{L^\infty}\lesssim e^{-s}M^\frac{1}{2}\leq \frac{3}{4}e^{-\frac{3}{4}s},
\end{align*}
where we use a small power of $e^{-s}$ to absorb the constant and $M$ factor. By Fundamental Theorem of Calculus, the assumption that $|T_*|\leq 2\epsilon^\frac{7}{4}$ and $\tau(-\epsilon)=0$, we get
\begin{align*}
|\tau(t)|\leq |\tau(-\epsilon)|+\int_{-\epsilon}^{T_*}|\dot{\tau}(t')|\,dt'\leq (2\epsilon^\frac{7}{4}+\epsilon)\epsilon^\frac{3}{4}\leq \frac{3}{2}\epsilon^\frac{7}{4}
\end{align*}
by choosing $\epsilon$ sufficiently small. 

Consider $h(t):=t-\tau(t)$. We have 
\begin{gather*}
h(-\epsilon)=-\epsilon,\quad h(T_*)=0,\\
\dot{h}=1-\dot{\tau}\geq 1-\epsilon^\frac{3}{4}>0.
\end{gather*}
Hence, we must have $
h(t)<0$, which implies $t<\tau(t)$ for all $t\in [-\epsilon,T_*)$. 

For $T_*$, note that $\tau(-\epsilon)=0$, $\tau(T_*)=T_*$ is equivalent to 
\begin{align*}
\int_{-\epsilon}^{T_*}\big(1-\dot{\tau}(t)\big)\,dt=\epsilon.
\end{align*}
Since $|\dot{\tau}(t)|\leq \epsilon^\frac{3}{4}$, we have
\begin{align*}
(1-\epsilon^\frac{3}{4})(T_*+\epsilon)&=\int_{-\epsilon}^{T_*}(1-\epsilon^\frac{3}{4})\,dt\leq \epsilon\\
\implies\qquad T_*&\leq \frac{\epsilon^\frac{7}{4}}{1-\epsilon^\frac{3}{4}}\leq \frac{3}{2}\epsilon^\frac{7}{4},
\end{align*}
if $\epsilon$ is sufficiently small. This closes the bootstrap \eqref{eq:taubound}.
\end{proof}

\begin{proposition}[$\kappa$ and $\xi$] The dynamic variable $\kappa$ satisfies $|\kappa(t)|\leq M$, and $\xi$ satisfies 
\begin{align*}
|\dot{\xi}(t)|\leq 2M,\quad |\xi(t)|\leq 3M\epsilon.
\end{align*}
\end{proposition}
\begin{proof}
Since $\kappa(t)=u\big(\xi(t),t\big)$, and by \eqref{eq:ULinfty}
\begin{align*}
\|u(\cdot,t)\|_{L^\infty}=\|e^{-\frac{s}{2}}U(\cdot,s)+\kappa\|_{L^\infty}\leq M,
\end{align*}
we have $|\kappa(t)|\leq M$ for all $t\in[-\epsilon,T_*]$.

For $\xi$, by \eqref{eq:kappaxi}, Lemma \ref{lem:H[U]} and \eqref{eq:3xUat0}, 
\begin{align*}
|\dot{\xi}(t)|\leq |\kappa(t)|+Ce^{-\frac{3}{2}s}M^\frac{3}{2}\leq M+C\epsilon^\frac{3}{2}M^\frac{3}{2}\leq \frac{3}{2} M,
\end{align*}
where we use the $\epsilon^\frac{3}{2}$ to absorb the constant and $M$ factor. Since $\xi(-\epsilon)=0$, by Fundamental Theorem of Calculus, for all $t\in [-\epsilon,T_*]$,
\begin{align*}
|\xi(t)|\leq |\xi(-\epsilon)|+\int_{-\epsilon}^{T_*}|\dot{\xi}(t)|\,dt\leq (2\epsilon^\frac{7}{4}+\epsilon)\frac{3}{2} M\leq \frac{5}{2} M\epsilon.
\end{align*}
Hence, we close the bootstrap \eqref{eq:xibound}.
In particular, let $x_*:=\xi(T_*)$, then $|x_*|\leq 3M\epsilon$. 
\end{proof}

This section plus Section \ref{sec:L2estimate} together close the bootstrap part of Theorem \ref{thm:self}. 

\section{Proof of the main result}
\label{sec:proofmain}
\begin{proof}[Proof of Theorem \ref{thm:self}]
The bootstrap part was done in Section \ref{sec:L2estimate} and \ref{sec:closure}. It remains to prove the pointwise convergence of $U$. The convergence at $X=0$ is trivial due to the constraint \eqref{eq:constraints}. So we only need to prove the convergence for $X\neq 0$.

Let $\nu:=\lim_{s\to\infty}\partial_X^3U(0,s)$. We first check that this limit exists. Indeed, by Fundamental Theorem of Calculus, for every $s\geq -\log\epsilon$,
\begin{align*}
\partial_X^3U(0,s)=\partial_X^3U(0,-\log\epsilon)+\int_{-\log\epsilon}^s\partial_s\partial_X^3U(0,s')\,ds'.
\end{align*}
By \eqref{eq:partials3X}, $
|\partial_s\partial^3_XU(0,s)|=|\partial_s\partial^3_X\widetilde{U}(0,s)|\leq 3e^{-\frac{2}{5}s}.$
In particular, 
\begin{align*}
\int_{-\log\epsilon}^\infty|\partial_s\partial_X^3U(0,s)|\,ds<\infty.
\end{align*}
Hence, we have a well-defined limit
\begin{align*}
\lim_{s\to\infty}\partial_X^3U(0,s)=\partial_X^3U(0,-\log\epsilon)+\int_{-\log\epsilon}^\infty \partial_s\partial_X^3U(0,s)\,ds.
\end{align*}
By \eqref{eq:3xtildeUat0} and \eqref{eq:barUat0}, we have $|\nu-6|\leq \epsilon^\frac{1}{4}$.

Next, we compute the first 3 spatial derivatives of $\overline{U}_\nu$:
\begin{gather*}
\partial_X\overline{U}_\nu(X)=\partial_X\overline{U}\Big(\big(\frac{\nu}{6}\big)^\frac{1}{2}X\Big),\\
\partial_X^2\overline{U}_\nu(X)=\big(\frac{\nu}{6}\big)^\frac{1}{2}\partial_X^2\overline{U}\Big(\big(\frac{\nu}{6}\big)^\frac{1}{2}X\Big),\\
\partial_X^3\overline{U}_\nu(X)=\frac{\nu}{6}\partial_X^3\overline{U}\Big(\big(\frac{\nu}{6}\big)^\frac{1}{2}X\Big).
\end{gather*}
Thus, by \eqref{eq:barUat0}, $\partial_X\overline{U}_\nu(0)=-1$, $\partial_X^2\overline{U}_\nu(0)=0$, $\partial_X^3\overline{U}_\nu(0)=\nu$.

Let $\widetilde{U}_\nu:=U-\overline{U}_\nu$. By the constraint \eqref{eq:constraints}, the Taylor expansion of $\widetilde{U}_\nu(X,s)$ for $X\approx 0$ is
\begin{align*}
\widetilde{U}_\nu(X,s)&=\frac{1}{6}\partial_X^3\widetilde{U}_\nu(0,s)X^3+\frac{1}{24}\partial_X^4\widetilde{U}_\nu(X',s)X^4\qquad X'\text{ between }0\text{ and }X\\
\implies\qquad |\widetilde{U}_\nu(X,s)|&\leq\frac{1}{6}|\partial_X^3\widetilde{U}_\nu(0,s)||X|^3+\frac{1}{24}\|\partial_X^4\widetilde{U}(\cdot,s)\|_{L^\infty}X^4\\
&\leq \frac{1}{6}|\partial_X^3\widetilde{U}_\nu(0,s)||X|^3+CM^\frac{7}{2}X^4,
\end{align*}
since $\|\partial_X^4U(\cdot,s)\|_{L^\infty}\lesssim M^\frac{7}{2}$ (by the same proof as that of $\|H[\partial_X^4U](\cdot,s)\|_{L^\infty}$ in Lemma \ref{lem:H[partialjU]}), and
\begin{align*}
\|\partial_X^4\overline{U}_\nu\|_{L^\infty}=\big(\frac{\nu}{6}\big)^\frac{3}{2}\|\partial_X^4\overline{U}\|_{L^\infty}\leq 30 \big(\frac{\nu}{6}\big)^\frac{3}{2}.
\end{align*}
Fix $X_0\neq 0$ close to 0, and fix a small $\delta>0$ such that $\delta\leq M^\frac{7}{2}|X_0|^4$. Since $\lim_{s\to\infty}\partial_X^3\widetilde{U}_\nu(0,s)=0$, there exists a sufficiently large $s_0$ depending on $X_0$ and $\delta$, such that for all $s\geq s_0$,
\begin{align}
|\widetilde{U}_\nu(X_0,s)|\leq CM^\frac{7}{2}X_0^4+\delta.\label{eq:Taylor}
\end{align}

Now we need to find an equation for $\widetilde{U}_\nu$. First, note that $\overline{U}_\nu$ also satisfies the self-similar Burgers equation \eqref{eq:selfBurgers}. Second, we can rewrite \eqref{eq:ansatz} as 
\begin{align*}
\big(\partial_s-\frac{1}{2}\big)U+\big(U+\frac{3}{2}X\big)\partial_XU=F_U-\frac{e^\frac{s}{2}(\kappa-\dot{\xi})}{1-\dot{\tau}}\partial_XU-\frac{\dot{\tau}}{1-\dot{\tau}}U\partial_XU
\end{align*}
where $F_U$ is the right hand side of \eqref{eq:ansatz}. Note that the left hand side without the $\partial_sU$ matches the self-similar Burgers equation. Then
\begin{align}
\big(\partial_s-\frac{1}{2}+\partial_X\overline{U}_\nu\big)\widetilde{U}_\nu+\big(U+\frac{3}{2}X\big)\partial_X\widetilde{U}_\nu=F_U-\frac{e^\frac{s}{2}(\kappa-\dot{\xi})}{1-\dot{\tau}}\partial_XU-\frac{\dot{\tau}}{1-\dot{\tau}}U\partial_XU.\label{eq:eqtildenu}
\end{align}
Denote the right hand side of this equation as $F_{\tilde{U}_\nu}$. We claim that 
\begin{align}
\int_{-\log\epsilon}^\infty\|F_{\tilde{U}_\nu}(\cdot,s)\|_{L^\infty}\,ds'<\infty.\label{eq:Fintegrable}
\end{align}
Indeed, by Lemma \ref{lem:FU}, $\|F_U(\cdot,s)\|_{L^\infty}\lesssim M^\frac{3}{2}e^{-\frac{1}{4}s}$. By the first equality in \eqref{eq:kappaxi}, Lemma \ref{lem:H[partialjU]}, \eqref{eq:3xUat0} and \eqref{cor:1xULinfty},
\begin{align*}
\big|\frac{e^\frac{s}{2}(\kappa-\dot{\xi})}{1-\dot{\tau}}\big|\|\partial_XU(\cdot,s)\|_{L^\infty}\lesssim M^\frac{3}{2}e^{-s}.
\end{align*}
By \eqref{eq:ULinfty}, and $|\kappa|\leq M$, we have
\begin{align*}
\|U(\cdot,s)\|_{L^\infty}\leq \|U(\cdot,s)+e^\frac{s}{2}\kappa\|_{L^\infty}+e^\frac{s}{2}|\kappa|\leq 2Me^\frac{s}{2}.
\end{align*}
Then by \eqref{cor:1xULinfty} and \eqref{eq:taubound}, 
\begin{align*}
\big|\frac{\dot{\tau}}{1-\dot{\tau}}\big|\|U\partial_XU(\cdot,s)\|_{L^\infty}\lesssim e^{-\frac{3}{4}s}\cdot Me^\frac{s}{2}\lesssim Me^{-\frac{1}{4}s}.
\end{align*}
Putting the three terms together, we get
\begin{align*}
\|F_{\tilde{U}_\nu}(\cdot,s)\|_{L^\infty}\lesssim M^\frac{3}{2}e^{-\frac{1}{4}s}+M^\frac{3}{2}e^{-s}+Me^{-\frac{1}{4}s},
\end{align*}
which is $s$-integrable.

Let $\Psi(X_0,\cdot):[s_0,\infty)\to\mathbb{R}$ be the Lagrangian trajectory of $\widetilde{U}_\nu$, i.e. 
\begin{equation}
\begin{aligned}
\frac{d}{ds}\Psi(X_0,s)&=\big(U+\frac{3}{2}X\big)\circ \Psi(X_0,s),\\
\Psi(X_0,s_0)&=X_0,
\end{aligned}\label{eq:newtrajectory}
\end{equation}
By a similar argument as in Lemma \ref{lem:lowerLagrangian} (using Mean Value Theorem), we see that $\Psi$ is repelling\footnote{This statement says that the trajectory of $\Psi$ escapes to infinity exponentially fast for all $X_0\neq 0$, not just for $X_0$ away from 0 as in Lemma \ref{lem:lowerLagrangian}. In the end we will send $X_0\to 0$. This is the reason why we consider the new trajectory.} 
\begin{align}
|\Psi(X_0,s)|\geq |X_0|e^{\frac{2}{5}(s-s_0)}.\label{eq:repelling}
\end{align}
Let $G(X,s)=e^{-\frac{3}{2}(s-s_0)}\widetilde{U}_\nu(X,s)$, then
\begin{align*}
\frac{d}{ds}G\circ\Psi(X_0,s)&=-\frac{3}{2}e^{-\frac{3}{2}(s-s_0)}\widetilde{U}_\nu\circ\Psi(X_0,s)+e^{-\frac{3}{2}(s-s_0)}\frac{d}{ds}\widetilde{U}_\nu\circ\Psi(X_0,s)\\
&=-\frac{3}{2}G\circ\Psi(X_0,s)+e^{-\frac{3}{2}(s-s_0)}\Big[F_{\tilde{U}_\nu}+\big(\frac{1}{2}-\partial_X\overline{U}_\nu\big)\widetilde{U}_\nu\Big]\circ\Psi(X_0,s)\\
&=-\big(1+\partial_X\overline{U}_\nu\big)  G\circ\Psi(X_0,s)+e^{-\frac{3}{2}(s-s_0)}F_{\tilde{U}_\nu}\\
\implies\qquad \big(\frac{d}{ds}+1&+\partial_X\overline{U}_\nu\big)G\circ\Psi(X_0,s)=e^{-\frac{3}{2}(s-s_0)}F_{\tilde{U}_\nu}\circ\Psi(X_0,s).
\end{align*}
Since $\|\partial_X\overline{U}_\nu\|_{L^\infty}=\|\partial_X\overline{U}\|_{L^\infty}=1$, the damping term $1+\partial_X\overline{U}_\nu\geq 0$.
Hence, by Gr\"onwall's lemma, 
\begin{align*}
|G\circ \Psi(X_0,s)|&\leq |G(X_0,s_0)|+\int_{s_0}^s e^{-\frac{3}{2}(s'-s_0)}|F_{\tilde{U}_\nu}\circ\Psi(X_0,s')|\,ds'\\
e^{-\frac{3}{2}(s-s_0)}|\widetilde{U}_\nu\circ\Psi(X_0,s)|&\leq \widetilde{U}_\nu(X_0,s_0)+\int_{s_0}^s e^{-\frac{3}{2}(s'-s_0)}|F_{\tilde{U}_\nu}\circ\Psi(X_0,s')|\,ds'\\
&\leq CM^\frac{7}{2}X_0^4+\delta+\delta\leq (C+2)M^\frac{7}{2}X_0^4,
\end{align*}
where we used \eqref{eq:Taylor} and \eqref{eq:Fintegrable} upon making $s_0$ sufficiently large. Hence, 
\begin{align*}
|\widetilde{U}_\nu\circ\Psi(X_0,s)|\leq (C+2)e^{\frac{3}{2}(s-s_0)}M^\frac{7}{2}X_0^4.
\end{align*}
Then for $s_0\leq s\leq s_0+\frac{13}{5}\log|X_0|^{-1}$, we have
\begin{align}
|\widetilde{U}_\nu\circ\Psi(X_0,s)|\leq (C+2) M^\frac{7}{2}|X_0|^{4-\frac{3}{2}\cdot\frac{13}{5}}\leq (C+2) M^\frac{7}{2}|X_0|^\frac{1}{10}.\label{eq:positive}
\end{align}
For any $X$ between $X_0$ and $\Psi(X_0,s_0+\frac{13}{5}\log|X_0|^{-1})$ (e.g.\ if $X_0>0$, then this is $X_0\leq X\leq \Psi(X_0,s_0+\frac{13}{5}\log|X_0|^{-1})$), there exists $s_0\leq s\leq s_0+\frac{13}{5}\log|X_0|^{-1}$ such that $X=\Psi(X_0,s)$. So for such $(X,s)$, by \eqref{eq:positive},
\begin{align*}
|\widetilde{U}_\nu(X,s)|\leq (C+2)M^\frac{7}{2}|X_0|^\frac{1}{10}.
\end{align*}
By \eqref{eq:repelling}, this will cover at least all $X$ such that 
\begin{align*}
|X_0|\leq |X|\leq |X_0|e^{\frac{2}{5}\cdot\frac{13}{5}\log|X_0|^{-1}}=|X_0|^{-\frac{1}{25}}.
\end{align*}
So if we take the limit $s_0\to\infty$, then for all $X$ such that $|X_0|\leq |X|\leq |X_0|^{-\frac{1}{25}}$,
\begin{align*}
\limsup_{s\to\infty}|\widetilde{U}_\nu(X,s)|\leq (C+2) M^\frac{7}{2}|X_0|^\frac{1}{10}.
\end{align*}
 Finally, sending $X_0\to 0$, we get
\begin{align*}
\limsup_{s\to\infty}|\widetilde{U}_\nu(X,s)|=0
\end{align*}
for all $X\neq 0$. Therefore, the proof is complete.
\end{proof}

\begin{proof}[Proof of Theorem \ref{thm:main}]
\begin{enumerate}[leftmargin=*]
\item Solution in $\mathcal{C}\big([-\epsilon,T_*);\mathcal{C}^4\cap H^5(\mathbb{R})\big)$ exists and is unique: by \eqref{eq:1xUL2}, \eqref{eq:5xUL2}, \eqref{eq:xUL2}, $\|\partial^j_XU(\cdot,s)\|_{L^2}$ remains uniformly bounded in $s$ for $j=1,\dots,5$. By \eqref{eq:uUL2relation}, $\|\partial_x^ju(\cdot,t)\|_{L^2}$, $j=1,...,5$ remain finite before the blowup time. By Lemma \ref{lem:conservation}, $\|u(\cdot,t)\|_{L^2}=\|u_0\|_{L^2}$ for all $t\in[-\epsilon,T_*)$. Hence, for any $T<T_*$, $\|u(\cdot,t)\|_{H^5}<\infty$ for $t\in[-\epsilon,T]$. By Theorem \ref{thm:wellposed}, there exists a unique solution in $\mathcal{C}\big([-\epsilon,T_*);H^5(\mathbb{R})\big)$. Finally, we note that $H^5(\mathbb{R})\subset \mathcal{C}^4(\mathbb{R})$.
\item Blowup time $|T_*|\leq 2\epsilon^\frac{7}{4}$ and blowup location $|x_*|\leq 3M\epsilon$ are proved in Section \ref{sec:blowuptime}.
\item $L^\infty$-norm of solution $u$: by \eqref{eq:ULinfty}, $\|u(\cdot,t)\|_{L^\infty}=\|e^{-\frac{s}{2}}
U(\cdot,s)+\kappa(t)\|_{L^\infty}\leq M$ for all $t\in[-\epsilon,T_*]$.

\item Blow up of $\partial_xu$: since
\begin{align*}
\partial_xu(x,t)=\big(\tau(t)-t\big)^{\frac{1}{2}-\frac{3}{2}}\partial_XU\bigg(\frac{x-\xi(t)}{(\tau(t)-t)^\frac{3}{2}},s\bigg)=\frac{1}{\tau(t)-t}\partial_XU(X,s),
\end{align*}
we see that 
\begin{align*}
\partial_xu\big(\xi(t),t\big)=\frac{1}{\tau(t)-t}\partial_XU(0,s)=-\frac{1}{\tau(t)-t}.
\end{align*}
We claim that for all $t\in[-\epsilon,T_*)$,
\begin{align*}
\frac{1}{2}\leq \frac{\tau(t)-t}{T_*-t}\leq 2.
\end{align*}
Indeed, this is equivalent to 
\begin{align*}
\begin{cases}
T_*-t\leq 2\tau(t)-2t,\\
\tau(t)-t\leq 2T_*-2t,
\end{cases}
\quad\iff \quad\begin{cases}
2\tau(t)-t\geq T_*,\\
\tau(t)+t\leq 2T_*,
\end{cases}
\end{align*}
which is true since
\begin{align*}
\frac{d}{dt}\big(2\tau(t)-t\big)=2\dot{\tau}-1\leq 0,\quad 2\tau(T_*)-T_*=T_*,\\
\frac{d}{dt}\big(\tau(t)+t\big)=\dot{\tau}+1\geq 0,\quad \tau(T_*)+T_*=2T_*.
\end{align*}
Hence, as $t\to T_*$, $\partial_xu\big(\xi(t),t\big)\to -\infty$, so $\partial_xu$ blows up at $x_*=\xi(T_*)$. Moreover, we have the following rate
\begin{align*}
\frac{1}{2(T_*-t)}\leq |\partial_xu\big(\xi(t),t\big)|\leq \frac{2}{T_*-t}.
\end{align*}
In fact, since by \eqref{cor:1xULinfty}, $|\partial_XU(\cdot,s)|$ attains its maximum uniquely at $X=0$, we have $\|\partial_xu(\cdot,t)\|_{L^\infty}=|\partial_xu\big(\xi(t),t\big)|$. Hence, 
\begin{align*}
\frac{1}{2(T_*-t)}\leq \|\partial_xu(\cdot,t)\|_{L^\infty}\leq \frac{2}{T_*-t}.
\end{align*}
For $x\neq x_*$, if $|x-x_*|>\frac{1}{2}$, then there exists $t_1\in[-\epsilon,T_*)$ such that $|x-\xi(t)|\geq \frac{1}{2}$ for all $t\in [t_1,T_*]$. In terms of the self-similar variables, this implies 
\begin{align*}
|X|\geq \frac{1}{2}\big(\tau(t)-t\big)^{-\frac{3}{2}}=\frac{1}{2}e^{\frac{3}{2}s}\qquad\forall\,t\in[t_1,T_*).
\end{align*} 
Hence, by \eqref{eq:1xUfar}, we get 
\begin{align*}
|\partial_xu(x,t)|=e^s|\partial_XU(X,s)|\leq 2\qquad\forall\,t\in[t_1,T_*].
\end{align*}
If $|x-x_*|<\frac{1}{2}$, then there exists $t_2\in[-\epsilon,T_*)$ such that $\frac{1}{2}|x-x_*|\leq |x-\xi(t)|\leq \frac{1}{2}$ for all $t\in[t_2,T_*]$, i.e. $\frac{1}{2}|x-x_*|e^{\frac{3}{2}s}\leq |X|\leq \frac{1}{2}e^{\frac{3}{2}s}$. By choosing a larger $t_2$ if necessary, we may also assume $|X|\geq \frac{1}{2}e^{\frac{3}{2}s}\geq 1$. Hence, by \eqref{eq:1xUmiddle}, we get
\begin{align*}
|\partial_xu(x,t)|=\frac{1}{\tau(t)-t}|\partial_XU(X,s)|\leq \frac{1}{\tau(t)-t}|X|^{-\frac{2}{3}}\leq 2^\frac{2}{3}|x-x_*|^{-\frac{2}{3}}.
\end{align*}
In fact, by \eqref{eq:1xtildeUmiddle} and \eqref{eq:1xbarUfar}, for $x\neq x_*$ such that $|x-x_*|< \frac{1}{2}$,
\begin{align*}
|\partial_xu(x,t)|\sim |x-x_*|^{-\frac{2}{3}}\qquad\text{as }t\to T_*.
\end{align*}
This indicates that $u(\cdot,T_*)\in\mathcal{C}^\frac{1}{3}(\mathbb{R})$ and it has a cusp singularity at $(x_*,T_*)$, similar to the one of $|x|^\frac{1}{3}$ at $x=0$.
\end{enumerate}
\end{proof}

\begin{proof}[Proof of Corollary \ref{cor:open}]
\begin{enumerate}[leftmargin=*]
\item $\partial_xu_0(0)$ and $\partial^2_xu_0(0)$ can be taken in an open set of possible values: first note that $M$ and $\epsilon$ can be taken in an open set of values since they only need to be ``sufficiently large" and ``sufficiently small", respectively. Hence $\partial_xu_0(0)$ can be taken in an open set of values. Next, if $|\partial^4_xu_0(x)|\sim O(\epsilon^{-\frac{11}{2}})$ for $x$ near 0 (which is true for functions within a small $H^5$-open neighborhood of initial data given in Section \ref{sub:initial physical} and \ref{sub:initialself}), then we do a Taylor expansion near $x=0$:
\begin{align*}
\partial^2_xu_0(x)&=\partial^2_xu_0(0)+\partial^3_xu_0(0)x+O(\epsilon^{-\frac{11}{2}})x^2\\
&=\partial^2_xu_0(0)+6\epsilon^{-4}x+(\partial^3_xu_0(0)-6\epsilon^{-4})x+O(\epsilon^{-\frac{11}{2}})x^2.
\end{align*} 
If $\partial^2_xu_0(0)$ and $|\partial^3_xu_0(0)-6\epsilon^{-4}|$ are sufficiently small (both of which are true under small $H^5$-perturbation), then there exists an $x_0$ near 0 such that $\partial^2_xu_0(x_0)=0$, and that $\partial_xu_0$ attains minimum at $x_0$. By the change of coordinate $x\mapsto x-x_0$, our analysis does not change. So we can relax the condition $\partial^2_xu_0(0)=0$ into $|\partial^2_xu_0(0)|$ being sufficiently small.
\item Compact support property can be relaxed: note that we choose $u_0$ to have compact support only for convenience. We only need to bound $\|u_0\|_{L^2}$ and set a proper ``large scale" in $X$, in this case is $\frac{1}{2}e^{\frac{3}{2}s}$. Also any element in a small open neighborhood of $u_0$ in $H^5$-topology should have its bulk in $[-1,1]$. Hence, not having compact support only changes \eqref{eq:u0L2} and Proposition \ref{prop:UL2} by a constant factor, while keeping the ``large scale" threshold the same, so up to a constant factor it does not affect the closure of all the bootstrap assumptions in Section \ref{sec:assumptionself} and \ref{sec:assumptiondynamic}.
\item Inequalities can be replaced by strict inequalities: for all the initial data assumptions that are inequalities, even though they are not open conditions, we can introduce a pre-factor sufficiently close to 1 in terms of $M$ and $\epsilon$, and replace the ``$\leq$" with ``$<$" without affecting the proof.
\item An $H^5$-perturbation of $u_0$, if small enough (in terms of $M$ and $\epsilon$), plus a change of coordinate as above, leads to a small $H^5$-perturbation of $U(X,-\log\epsilon)+\kappa_0$, while keeping \eqref{eq:initU0}-\eqref{eq:init2xU0}. The $L^2$-estimates still hold with a slight change in value, and the closure of bootstrap, up to a pre-factor as discussed above, still holds. 
\end{enumerate}
\end{proof}

\appendix
\section{Local well-posedness of the BH equation}\label{ap:wellposed}
\begin{theorem}
The initial value problem for the BH equation 
\begin{equation*}
     \partial_t u+u\partial_xu=H[u],\qquad u(x,0)=u_0(x)
\end{equation*}\label{thm:wellposed}
is locally well-posed in $H^k$ for $k>3/2$.
\end{theorem}
\begin{proof}
For simplicity and relevance to our main theorem \ref{thm:main}, we only consider integer $k$, i.e. $k\geq 2$. We use a standard contraction mapping argument similar to the one for Burgers equation: multiplying \eqref{eq:BH} by $u$ and taking $H^k$-inner product
\begin{align*}
    \frac{1}{2}\frac{d}{dt}\|u\|_{H^k}^2+\langle u\partial_x u,u\rangle_{H^k}=\langle H[u],u\rangle_{H^k}.
\end{align*}
We recall that Hilbert transform of a function is orthogonal to itself. Also since Hilbert transform commutes with differentiation, we have $\langle H[u],u\rangle_{H^k}=0$. So the rest of the argument is identical to that of Burgers equation
\begin{align*}
    \frac{1}{2}\frac{d}{dt}\|u\|_{H^k}^2&\leq |\langle u\partial_xu,u\rangle_{H^k}|\\
    &\leq \sum_{j=0}^k\Big|\int_\mathbb{R}\partial^j_x(u\partial_xu)\partial^j_xu\,dx\Big|\\
    &=\sum_{j=0}^k\sum_{\alpha=2}^{j-1}|\langle \partial^\alpha_xu\partial^{j-\alpha+1}_xu,\partial^j_xu\rangle|+\langle \partial_xu\partial^j_xu+\partial^j_xu\partial_xu,\partial^j_xu\rangle+\Big|\int_\mathbb{R}u\partial^{j+1}_xu\partial^j_xu\,dx\Big|\\
&\lesssim \sum_{j=0}^k\sum_{\alpha=2}^{j-1}\|\partial^\alpha_xu\|_{L^\infty}\|\partial^{j-\alpha+1}_xu\|_{L^2}\|\partial^j_xu\|_{L^2}+\|\partial_xu\|_{L^\infty}\|\partial_x^ju\|_{L^2}^2+|\langle\partial_xu,(\partial_x^ju)^2\rangle|\\
&\lesssim \sum_{j=0}^k\sum_{\alpha=2}^{j-1}\|\partial_x^{\alpha}u\|_{H^1}\|\partial^{j-\alpha+1}_xu\|_{L^2}\|\partial^j_xu\|_{L^2}+\|\partial_xu\|_{L^\infty}\|\partial^j_xu\|_{L^2}^2\\
&\lesssim \|u\|_{H^k}^3,
\end{align*}
where we use integration by parts to treat the term 
\begin{align*}
    \int_\mathbb{R}u\partial^{j+1}_xu \partial^j_xu\,dx=\frac{1}{2}\int_\mathbb{R}u\partial_x\big[(\partial^j_xu)^2\big]\,dx=-\frac{1}{2}\int_\mathbb{R}\partial_x u(\partial_x^j u)^2\,dx,
\end{align*}
and we use Sobolev embedding $H^1\subset L^\infty$ in dimension 1 to treat the sum from $\alpha=2$ to $j-1$. Note that in this range of $\alpha$, both $\alpha,j-\alpha+1\leq j-1$. For the last $\lesssim$ we use Sobolev embedding $\|\partial_xu\|_{L^\infty}\lesssim \|u\|_{H^k}$ for $k>3/2$. Hence,
\begin{align*}
\frac{d}{dt}\|u\|_{H^k}&\leq C \|u\|_{H^k}^2\\
\|u(t)\|_{H^k}&\leq \big(\frac{1}{\|u_0\|_{H^k}}-Ct\big)^{-1}.
\end{align*}
Therefore, we get
\begin{align*}
    \|u(t)\|_{H^k}\leq 2\|u_0\|_{H^k}\qquad\forall\,t\in \big[0,\frac{1}{2C\|u_0\|_{H^k}}\big].
\end{align*}
In particular, the maximal time of existence $T_*\geq 1/2C\|u_0\|_{H^k}$. From this we also get a minimum blowup rate\footnote{This rate agrees with the blowup rate of $\|\partial_xu(\cdot,t)\|_{L^\infty}$ as stated in Theorem \ref{thm:main}.} $\sim (T_*-t)^{-1}.$
\end{proof}

\section{Derivation of the self-similar equations}\label{ap:derivation}
We first derive \eqref{eq:ansatz} from \eqref{eq:BH} step by step. 

From the self-similar transformation \eqref{eq:selftransform}, we have the identities
\begin{align*}
\tau(t)-t=e^{-s},\quad x=\big(\tau(t)-t\big)^\frac{3}{2}X+\xi(t)=e^{-\frac{3}{2}s}X+\xi(t).
\end{align*}
The self-similar transformation works well with Hilbert transform
\begin{align*}
H[u](x,t)&=\frac{1}{\pi}\mathrm{p.v.}\int_{\mathbb{R}}\frac{u(y,t)}{x-y}\,dy\\
&=\frac{1}{\pi}\mathrm{p.v.}\int_{\mathbb{R}}\frac{\big(\tau(t)-t\big)^\frac{1}{2}U\Big(\frac{x-\xi(t)}{(\tau(t)-t)^\frac{3}{2}},-\log\big(\tau(t)-t\big)\Big)+\kappa(t)}{x-y}\,dy.
\end{align*}
We do the change of variable 
\begin{gather*}
Y=\frac{y-\xi(t)}{(\tau(t)-t)^\frac{3}{2}},
\end{gather*}
and the same change for $x$, then use $\tau(t)-t=e^{-s}$ to get
\begin{align*}
H[u](x,t)&=\frac{1}{\pi}\mathrm{p.v.}\int_\mathbb{R}\frac{e^{-\frac{s}{2}}U(Y,s)+\kappa}{(\tau(t)-t)^\frac{3}{2}(X-Y)}\big(\tau(t)-t\big)^\frac{3}{2}\,dY\\
&=\frac{1}{\pi}\mathrm{p.v.}\int\frac{e^{-\frac{s}{2}}U(Y,s)+\kappa}{X-Y}\,dY\\
&=H[e^{-\frac{s}{2}}U+\kappa](X,s).
\end{align*}
Note that
\begin{align*}
\frac{\partial X}{\partial t}&=\frac{\partial}{\partial t}\Big[\frac{x-\xi(t)}{(\tau(t)-t)^\frac{3}{2}}\Big]\\
&=\frac{-\dot{\xi}\big(\tau(t)-t\big)^\frac{3}{2}-\frac{3}{2}\big(\tau(t)-t\big)^\frac{1}{2}(\dot{\tau}-1)\big(x-\xi(t)\big)}{(\tau(t)-t)^3}\\
&=-\dot{\xi}\big(\tau(t)-t\big)^{-\frac{3}{2}}+\frac{3}{2}(1-\dot{\tau})\big(\tau(t)-t\big)^{-1}\frac{x-\xi(t)}{(\tau(t)-t)^\frac{3}{2}}\\
&=-\dot{\xi}e^{\frac{3}{2}s}+\frac{3}{2}e^s(1-\dot{\tau})X,
\end{align*}
and that
\begin{align*}
\frac{ds}{dt}&=-\frac{d}{dt}\log\big(\tau(t)-t\big)=\frac{1-\dot{\tau}}{\tau(t)-t}=e^s(1-\dot{\tau}).
\end{align*}
So we have
\begin{align*}
\partial_tu(x,t)&=\partial_t\Big[\big(\tau(t)-t\big)^\frac{1}{2}U\Big(\frac{x-\xi(t)}{(\tau(t)-t)^\frac{3}{2}},-\log\big(\tau(t)-t\big)\Big)+\kappa(t)\Big]\\
&=\frac{1}{2}\big(\tau(t)-t\big)^{-\frac{1}{2}}(\dot{\tau}-1)U(X,s)+\big(\tau(t)-t\big)^\frac{1}{2}\Big[\partial_XU(X,s)\frac{\partial X}{\partial t}+\partial_sU(X,s)\frac{ds}{dt}\Big]+\dot{\kappa}\\
&=-\frac{1}{2}e^\frac{s}{2}(1-\dot{\tau})U(X,s)+e^{-\frac{s}{2}}\Big[\partial_XU(X,s)\big(-\dot{\xi}e^{\frac{3}{2}s}+\frac{3}{2}(1-\dot{\tau})e^sX\big)+\partial_sU(X,s)e^s(1-\dot{\tau})\Big]+\dot{\kappa}\\
&=-\frac{1}{2}e^\frac{s}{2}(1-\dot{\tau})U(X,s)+e^\frac{s}{2}\partial_XU(X,s)\big(-e^\frac{s}{2}\dot{\xi}+\frac{3}{2}(1-\dot{\tau})X\big)+e^\frac{s}{2}(1-\dot{\tau})\partial_sU(X,s)+\dot{\kappa},
\end{align*}
and 
\begin{align*}
\partial_xu(x,t)&=\partial_x\Big[\big(\tau(t)-t\big)^\frac{1}{2}U\Big(\frac{x-\xi(t)}{(\tau(t)-t)^\frac{3}{2}},-\log\big(\tau(t)-t\big)\Big)+\kappa(t)\Big]\\
&=\big(\tau(t)-t\big)^\frac{1}{2}\partial_XU(X,s)\big(\tau(t)-t\big)^{-\frac{3}{2}}\\
&=e^s\partial_XU(X,s).
\end{align*}
Plugging these into \eqref{eq:BH}, we get
\begin{align*}
-\frac{1}{2}e^\frac{s}{2}(1-\dot{\tau})U(X,s)+e^\frac{s}{2}\partial_XU(X,s)\big(-e^\frac{s}{2}\dot{\xi}+\frac{3}{2}(1-\dot{\tau})X\big)+e^\frac{s}{2}(1-\dot{\tau})\partial_sU(X,s)+\dot{\kappa}&\\
+e^s\partial_XU(X,s)\Big(e^{-\frac{s}{2}}U(X,s)+\kappa\Big)=H[e^{-\frac{s}{2}}U+\kappa](X,s)&\\
\big(\partial_s-\frac{1}{2}\big)U+\Big(\frac{U+e^\frac{s}{2}(\kappa-\dot{\xi})}{1-\dot{\tau}}+\frac{3}{2}X\Big)\partial_XU=-\frac{e^{-\frac{s}{2}}\dot{\kappa}}{1-\dot{\tau}}+\frac{e^{-s}}{1-\dot{\tau}}H[U+e^\frac{s}{2}\kappa]\qquad&
\end{align*}
which is \eqref{eq:ansatz}.

We then derive \eqref{eq:U+kappa}:
\begin{align*}
\partial_s(e^{-\frac{s}{2}}U+\kappa)&=-\frac{1}{2}e^{-\frac{s}{2}}U+e^{-\frac{s}{2}}\partial_sU+\dot{\kappa}\frac{dt}{ds}\\
&=e^{-\frac{s}{2}}(\partial_s-\frac{1}{2})U+\frac{e^{-s}}{1-\dot{\tau}}\dot{\kappa},\\
\partial_XU(e^{-\frac{s}{2}}U+\kappa)&=e^{-\frac{s}{2}}\partial_XU.
\end{align*}
So from \eqref{eq:ansatz} we have
\begin{align*}
\partial_s(e^{-\frac{s}{2}}U+\kappa)+V\partial_X(e^{-\frac{s}{2}}U+\kappa)&=e^{-\frac{s}{2}}\Big[(\partial_s-\frac{1}{2})U+V\partial_XU+\frac{e^{-\frac{s}{2}}\dot{\kappa}}{1-\dot{\tau}}\Big]\\
&=\frac{e^{-\frac{3}{2}s}}{1-\dot{\tau}}H[U+e^\frac{s}{2}\kappa],
\end{align*}
which is \eqref{eq:U+kappa}.

\section{Interpolation lemmas}
\begin{lemma}[Gagliardo-Nirenberg-Sobolev interpolation]
Let $f:\mathbb{R}^d\to\mathbb{R}$, and let $1\leq q,\,r\leq \infty$, $j,\,m\in \mathbb{N}$ (including 0) and $j/m\leq \alpha\leq 1$ be such that
\begin{align*}
\frac{1}{p}=\frac{j}{d}+\alpha(\frac{1}{r}-\frac{m}{d})+\frac{1-\alpha}{q},
\end{align*}
then
\begin{align*}
\|\partial^j f\|_{L^p}\lesssim \|\partial^mf\|_{L^r}^\alpha\|f\|_{L^q}^{1-\alpha},
\end{align*}
with two exceptions
\begin{enumerate}
\item If $j=0,\ mr<d$ and $q=\infty$, then we assume additionally that either $f$ tends to 0 at infinity or that $f\in L^{q'}$ for some $q'<\infty$.
\item If $1<r<\infty$ and $m-j-d/r\in\mathbb{N}$, then we also assume that $\alpha<1$. 
\end{enumerate}\label{lem:GNSinterpolation}
\end{lemma}

\begin{lemma}[Sobolev interpolation]
As a special case, when $p=q=r=2$, we can make the constant to be 1:
\begin{align*}
\|\partial^jf\|_{L^2}\leq \|\partial^mf\|_{L^2}^\alpha\|f\|_{L^2}^{1-\alpha}
\end{align*}
where $\alpha=j/m$.\label{lem:Sinterpolation}
\end{lemma}

\section{Two simple lemmas}\label{ap:simple}
\begin{lemma}
By choosing $\epsilon$ sufficiently small, we have
\begin{align*}
\Big|\frac{1}{1-\dot{\tau}}\Big|\leq 1+2\epsilon^\frac{3}{4}.
\end{align*}
Thus, by choosing $\epsilon$ sufficiently small, we can make $1/(1-\dot{\tau})$ as close to 1 as we need.\label{lem:1-dottau}
\end{lemma}

\begin{proof}
By \eqref{eq:taubound}, $|\dot{\tau}(t)|\leq \epsilon^\frac{3}{4}$. So we have
\begin{align*}
\frac{1}{1-\dot{\tau}}&=1+\dot{\tau}+o(\dot{\tau})\\
\implies\qquad\Big|\frac{1}{1-\dot{\tau}}\Big|&\leq 1+2|\dot{\tau}|\leq 1+2e^{-\frac{3}{4}s}\leq 1+2\epsilon^\frac{3}{4}.
\end{align*}
\end{proof}

\begin{lemma}[Conservation of $\|u(\cdot,t)\|_{L^2}$]
For sufficiently smooth solution $u$, we have $\|u(\cdot,t)\|_{L^2}=\|u_0\|_{L^2}$.\label{lem:conservation}
\end{lemma}
\begin{proof}
Taking inner product of \eqref{eq:BH} with $u$, and using the orthogonality of Hilbert transform, i.e. $\langle H[u],u\rangle=0$, we have
\begin{align*}
\frac{1}{2}\frac{d}{dt}\|u\|_{L^2}^2+\int_\mathbb{R}u^2\partial_xu\,dx&=0\\
\frac{1}{2}\frac{d}{dt}\|u\|^2_{L^2}+\frac{1}{3}\int_\mathbb{R}\partial_x(u^3)\,dx&=0\\
\frac{d}{dt}\|u\|^2_{L^2}&=0
\end{align*}
if the solution is sufficiently smooth for the integrals to make sense. Hence, for sufficiently smooth solution $u$, the $L^2$-norm is conserved.
\end{proof}

\end{document}